\documentclass[10pt]{cpamart1}     %% Input Standard CPAM class.
\authorheadline{S.~Chatterjee and P.~S.~Dey}
\titleheadline{Long-Range First-Passage Percolation}

\usepackage{amsthm,amssymb}
\usepackage{enumerate,graphicx}

\usepackage{mathrsfs,dsfont}
\usepackage[colorlinks=true,linkcolor=blue,citecolor=red]{hyperref}

\usepackage{times}

\theoremstyle{plain}
\newtheorem{thm}{Theorem}[section]
\newtheorem{lem}[thm]{Lemma}
\newtheorem{cor}[thm]{Corollary}
\newtheorem{prop}[thm]{Proposition}

\theoremstyle{definition}

\theoremstyle{remark}
\newtheorem{rem}[thm]{Remark}

\newenvironment{enumeratei}{\begin{enumerate}[\upshape (i)]\setlength{\itemsep}{4pt}}{\end{enumerate}}
\newenvironment{enumeratea}{\begin{enumerate}[\upshape (a)]\setlength{\itemsep}{4pt}}{\end{enumerate}}
\newenvironment{enumeraten}{\begin{enumerate}[\upshape 1.]\setlength{\itemsep}{4pt}}{\end{enumerate}}

\renewcommand{\le}{\leqslant}
\renewcommand{\ge}{\geqslant}

\newcommand{\ra}{\rangle}
\newcommand{\la}{\langle}

\newcommand{\eps}{\varepsilon}

\newcommand{\norm}[1]{\left\Vert#1\right\Vert}
\newcommand{\abs}[1]{\left\vert#1\right\vert}

\newcommand{\ie}{\emph{i.e.,}\ }
\newcommand{\as}{\emph{a.s.}~}
\newcommand{\iid}{\emph{i.i.d.}~}
\newcommand{\eg}{\emph{e.g.,}\ }

\newcommand{\equald}{\stackrel{\mathrm{d}}{=}}

\def\qed{\hfill $\blacksquare$}
\let\ga=\alpha \let\gb=\beta \let\gc=\gamma \let\gd=\delta 
     \let\gl=\lambda 
 \let\go=\omega   \let\gs=\sigma  
  
 \let\gD=\Delta  \let\gL=\Lambda

\newcommand{\cA}{\mathcal{A}}\newcommand{\cB}{\mathcal{B}}

\newcommand{\cL}{\mathcal{L}}

\newcommand{\cP}{\mathcal{P}}

\newcommand{\vzero}{\mathbf{0}}

\newcommand{\vp}{\mathbf{p}}
\newcommand{\vu}{\mathbf{u}}
\newcommand{\vv}{\mathbf{v}}\newcommand{\vw}{\mathbf{w}}\newcommand{\vx}{\mathbf{x}}
\newcommand{\vy}{\mathbf{y}}\newcommand{\vz}{\mathbf{z}}
\newcommand{\dG}{\mathds{G}}
\newcommand{\dN}{\mathds{N}}
\newcommand{\dQ}{\mathds{Q}}
\newcommand{\dR}{\mathds{R}}
\newcommand{\dZ}{\mathds{Z}}
\newcommand{\sA}{\mathscr{A}}
\newcommand{\sB}{\mathscr{B}}
\newcommand{\sE}{\mathscr{E}}
\newcommand{\sP}{\mathscr{P}}
\newcommand{\sS}{\mathscr{S}}
\DeclareMathOperator{\E}{\mathds{E}}
\DeclareMathOperator{\pr}{\mathds{P}}
\DeclareMathOperator{\argmin}{argmin}
\newcommand{\eqd}{\,{\buildrel d \over =}\,}
\newcommand{\ball}[1]{\ensuremath{\text{\textup B}(#1)}}
\newcommand{\ballt}[1]{\ensuremath{\text{\textup B}_2(#1)}}
\newcommand{\linf}[1]{\ensuremath{\left\Vert#1\right\Vert_{\scriptscriptstyle\infty}}}
\newcommand{\lone}[1]{\ensuremath{\left\Vert#1\right\Vert}}
\newcommand{\tpi}{{\tilde\pi}}
\newcommand{\tB}{{\tilde B}}
\newcommand{\tI}{{\tilde I}}

\begin{document}                        %% Standard LaTeX command

%%      ---------------------------------------------------------------------
%%      -------------------------------- TITLE -----------------------------
%%      ---------------------------------------------------------------------

\title{Multiple phase transitions in long-range first-passage percolation on square lattices}

%%      ---------------------------------------------------------------------
%%      ------------------------------- AUTHORS -----------------------------
%%      ---------------------------------------------------------------------
\author{Shirshendu Chatterjee}{Courant Institute, New York University}
\author{Partha S.~Dey}{Department of Statistics, University of Warwick}

%%      ---------------------------------------------------------------------
%%      --------------------------- DEDICATION  (OPTIONAL)-------------------
%%      ---------------------------------------------------------------------

%       Uncomment the following line to insert a dedication.

%\dedication{ *** DEDICATION *** }        %% Enter dedication between braces.

%%      ---------------------------------------------------------------------
%%      --------------------------- ABSTRACT (OPTIONAL)----------------------
%%      ---------------------------------------------------------------------

%% ***** UNCOMMENT THE FOLLOWING TO INSERT AN ABSTRACT

\begin{abstract}
    We consider a  model of long-range first-passage percolation on the $d$ dimensional square lattice $\dZ^d$ in which any two distinct vertices $\vx, \vy \in \dZ^d$ are connected by an edge having exponentially distributed passage time with mean $\lone{\vx-\vy}^{\ga+o(1)}$, where $\ga>0$ is a fixed parameter and $\lone{\cdot}$ is the $\ell_1$--norm on $\dZ^d$. We analyze the asymptotic growth rate of the set $\cB_t$,  which consists of all $\vx \in \dZ^d$ such that the first-passage time between the origin $\vzero$ and $\vx$ is at most $t$, as $t\to\infty$. We show that depending on the values of $\ga$ there are four growth regimes:
    (i) instantaneous growth for $\ga<d$,
    (ii) stretched exponential growth for  $\ga\in (d,2d)$,
    (iii) superlinear growth for $\ga\in (2d,2d+1)$ and finally
    (iv) linear growth for $\ga>2d+1$ like the nearest-neighbor first-passage percolation model corresponding to $\ga=\infty$.
\end{abstract}

% With AMS-LaTeX, \maketitle follows the abstract
\maketitle

%%      ---------------------------------------------------------------------
%%      ------------------- TABLE OF CONTENTS (OPTIONAL) --------------------
%%      ---------------------------------------------------------------------

%% ***** IF YOUR PAPER IS OVER 40 PAGES AND YOU WISH TO HAVE A TABLE
%% ***** OF CONTENTS, PLEASE UNCOMMENT THE FOLLOWING LINE

\tableofcontents

%%      ---------------------------------------------------------------------
%%      ---------------------------- BODY OF PAPER --------------------------
%%      ---------------------------------------------------------------------

%%      Please input or insert the body of your paper here.

%%%%%%%%%%%%%%%%%%%%%%%%%%%%%%%%%%%%%%%%%%%
\section{Introduction}\label{sec:int}

We consider the infinite complete graph on the vertex set $\dZ^d$, say $\dG^d$, and a non-increasing positive function $r:(0,\infty) \to (0,\infty]$. To each edge $e$ of $\dG^d$ we assign an independent random weight of the form $\omega_e/r_e$, where $r_e$ is given by the value of the function $r$ evaluated at the Euclidean distance between the endpoints of the edge $e$ and $\{\omega_e\}$'s are \iid nonnegative random variables with common distribution $F$. The weight of an edge is interpreted as its {\it passage time}. Based on these passage times, one can define a {\it first-passage} metric on $\dZ^d$, in which the distance between two vertices is the minimum time required to reach one of them from the other using any of the paths in $\dG^d$ joining the two, and study the asymptotic growth of the associated $t$--ball (the set of vertices which can be reached within time $t$ from the origin) as $t$ tends to infinity.

In this paper, we focus on the case in which $F$ is the exponential distribution, and show that the family of {\it stochastic growth models} (indexed by the set of nonnegative non-increasing functions) exhibits a wide variety of growth behavior including instantaneous growth, exponential growth, any stretched exponential growth (the $t$--ball can have diameter and volume of order $\exp(t^{\theta+o(1)})$ for any $\theta \in (0,1)$), any superlinear growth (the $t$--ball can have diameter of order $t^{\theta+o(1)}$ for any $1<\theta<\infty$) and linear growth for different choices of the function $r$. This phenomenon occurs in much more general set-up when $F$ satisfies certain moment condition. In particular the phase transition between different growth behaviors depend on the behavior of $F$ near 0 and near infinity, here we will depict this for the case when $F$ is a positive power of exponential distribution.

This problem bridges between two vast areas of study: long-range percolation and nearest-neighbor first-passage percolation. We briefly discuss both of these areas in the following two Sections~\ref{subsec:lrp} and~\ref{subsec:fpp}.

%%%%%%%%%%%%%%%%%%%%%%%%%%%%%%%%%%%%%%%%%%%
\subsection{Long-range Percolation}\label{subsec:lrp}

During the last few decades there has been numerous research contributions which have led to a thorough understanding about the existence and type of phase transitions in different models of statistical mechanics. The simplest among such models is perhaps the Bernoulli bond percolation model, where one obtains a random graph by retaining each of the edges of a ground graph independently with probability $p\in(0,1)$. The literature on percolation theory is vast, so we mention only a few relevant references here and ask the interested readers to look into them for further ones. For an introduction and motivation to the subject and for earlier works, when the ground graph  is $\dZ^d$ with nearest-neighbor edges, we recommend \cite{G99}. See also \cite[Chapter 7]{LP05} for the treatment of percolation on general transitive graphs including homogeneous trees. Most of the focus in research related to percolation on the transitive infinite graphs has been on proving the existence of phase transition (depending on the appearance of infinite cluster(s)), analytic and geometric properties of the connected components (see \eg \cite{CI02} for the properties of connectivity functions) and scaling limits for critical percolation on $\dZ^d$ (see \eg  \cite{HS00, HS00a, S01, SW01}). Percolation has also been considered on large finite ground graphs such as the complete graph on $n$ vertices (which gives rise to the famous Erd\"os-R\'enyi random graph model), {\it small-world} and {\it scale-free} networks (in the context of epidemiology \cite{MN00, SCB02}), sparse random graphs (in the context of robustness of networks \cite {CNSW00}) and $n$-dimensional hypercube \cite{BCH06}.

An extension of the Bernoulli bond percolation model is the long-range percolation (LRP) model, in which each pair of distinct vertices $\vx, \vy \in \dZ^d$ is connected by an edge with probability $p_{\vx,\vy} \sim \gb\lone{\vx-\vy}^{-\ga+o(1)}$ (as $\lone{\vx-\vy}$ goes to infinity) for some parameters $\ga, \gb>0$. We denote the associated random subgraph of $\dG^d$ by $\dG^d_\vp$, where $\vp=(p_{\vx,\vy}\mid\vx,\vy\in\dZ^{d})$. This model was originally introduced in the mathematical-physics literature as an example of a model which exhibits phase transition even in one dimension, it also displays discontinuous transition of percolation density for $\ga=2$ in one dimension as $\gb$ varies. We refer the readers to \cite{NS86, AN86, S99, IN88} for more details about these works. Later, Benjamini and Berger \cite{BB01} have proposed LRP on finite cubic lattices to be models for social networks in connection with the study of ``small world" phenomenon \cite{WS98}; and in general LRP on $\dZ^d$ has gained interest as models of graphs with nontrivial volume growth. Most of the research focus in LRP has been on

\begin{enumeratea}
    \item scaling properties of the random metric $T^\vp(\cdot,\cdot)$ on $\dZ^d$ induced by the LRP random graph $\dG^d_\vp$ (see \cite{B04a}),
    \item the volume growth of the associated balls $\cB^\vp_t:= \{\vx\in\dZ^d: T^\vp(\vzero, \vx) \le  t\}$ (see \cite{B09, T10}) and
    \item the growth behavior of the diameter $D_L^{\vp}$ of the largest connected component in $\dG^d_\vp\cap [-L, L]^d$ (the restriction of $\dG^d_\vp$ to $[-L , L]^d$).
\end{enumeratea}

Combining contributions of numerous authors, it is known (sometimes conjectured but unproved) that for $p_{\vx,\vy}=\gb\lone{\vx-\vy}^{-\ga+o(1)}$ there are five distinct regimes depending on the relative positions of $\ga$ and $d$. The diameter $D_L^{\vp}$ is
\begin{enumerate}[$\qquad$]\setlength{\itemsep}{4pt}
    \item $\to \lceil \alpha/(d-\alpha)\rceil$ for $\ga < d$ due to~\cite[Example 6.1]{BKPS11}

    \item $\asymp \log L/\log\log L$ for $\alpha=d$ due to~\cite{CGS02}

    \item $= (\log L)^{\Delta(\alpha)+o(1)}$ for $d<\alpha<2d$ due to~\cite{B04a, B09}

    \item $= L^{\theta(\beta)+o(1)}$ for $p_{\vx\vy}=\beta\lone{\vx-\vy}^{-2d}$ (conjectured in~\cite{BB01} for any $d\ge 1$. See~\cite{CGS02} for a general upper bound for $\theta(\gb)$ and~\cite{DS13} for existence of $\theta(\gb)$ in $d=1$)

     \item $\asymp L$ for $\alpha>2d$ (expected~\cite{BB01} for any $d \ge 1, \gb>0$. See~\cite{B04} for a lower bound for all $\gb$. The upper bound holds for $\gb$ large)

\end{enumerate}

Here $a_L \asymp b_L$ means $a_L/b_L$ stays away from 0 and infinity with probability tending to 1 as $L \to \infty$. Other related areas of research involving LRP models include study of simple random walk on $\dG^d_\vp$ (\eg conditions for transience and recurrence \cite{B02}, bounds for spectral gap and heat kernel \cite{CS09}, scaling limits to Brownian motion or stable processes) and on its restrictions to the $d$-dimensional box $[-L, L]^d$ (\eg mixing time \cite{BBY08}). However, understanding these aspects require knowledge about much finer structure of the random graph~$\dG^{d}_{\vp}$.

%%%%%%%%%%%%%%%%%%%%%%%%%%%%%%%%%%%%%%%%%%%
\subsection{First-passage Percolation}\label{subsec:fpp}

Parallel to the development of the percolation theory, there has always been interest in studying different aspects of {\it shortest paths} between two vertices of  deterministically or randomly weighted graphs. In this regard, another classical model, the standard {\it first-passage percolation} (FPP) model, has gained a lot of interest in the mathematical-physics literature since its introduction in~1965 \cite{HW65} and has developed into an independent field by now.~FPP was originally introduced for the graph $\dZ^d$ with nearest neighbor edges in the context of flow of fluids through a random porous medium. We refer to \cite{SW78, K86} for an account of earlier works and to \cite{GK12} for recent results. Later, FPP on  $\dZ^d$ has been used extensively as the basic model in a variety of fields including competing infections in epidemiology \cite{H05, GM05, B10}, growing interfaces in statistical-physics \cite{KS91}. Also, FPP on other large finite graphs (\eg the complete graph \cite{J99}, sparse locally tree-like random graphs \cite{B08, BHH10, HHM01}) have been used for modeling information spreading and flows through networks.

In this model, each edge $e$ of a ground graph is associated with its {\it passage time} (or weight), and the passage times are independent and have common distribution $F$ supported on $[0,\infty]$. The passage time of a finite path in the ground graph is the sum of passage times of the edges present in the path, and the {\it first-passage time} $T^F(\vx,\vy)$ between two vertices $\vx$ and $\vy$ of the ground graph is the minimum passage time of a finite path joining them. Note that $T^F(\cdot,\cdot)$ is always a (random) pseudo-metric, and it is a random metric (which is called the {\it first-passage metric} on $\dZ^d$ associated with $F$) if $F$ has no atom at 0. Moreover, $T^F(\cdot,\cdot)$ can also be interpreted as the time required to communicate between its two arguments. While percolation theory deals with issues like connectivity of distant points of some context-dependent space and properties of connected clusters, the main focus of research in FPP is to analyze

\begin{enumerate}[$\bullet$]
    \item the first-passage metric -- (a) scaling properties, (b) fluctuations, (c) scaling limits,
    \item the associated first-passage balls $\cB^F_t:=\{\vx: T^F(\vzero,\vx)\le  t \}$ -- (a) the time evolution, (b) existence of asymptotic shape, (c) analytic and geometric properties of the limiting shape.
\end{enumerate}

It is well known that in any direction $\vx\in\dZ^d$ the first-passage metric $T^F$ on $\dZ^d$ grows linearly with the Euclidean metric, \ie $T^F(\vzero, n\vx)/n$ has a positive and finite limit as $n\to\infty$, and $T^F(\vzero,n\vx)$ has sublinear fluctuation provided $F(0) < p_c(d)$ (the critical bond percolation probability for $\dZ^d$ with nearest-neighbor edges). In addition, under suitable moment condition on $F$ (see \cite{CD81}), $\cB^F_t$ grows linearly in $t$ and has a deterministic limiting shape, \ie $\cB^F_t \approx (t B) \cap \dZ^d$ as $t\to\infty$ for some nonrandom compact set $B\subseteq \dR^d$. Although many estimates and techniques are available to analyze the distribution of $T^F(\vzero, n\vx)$, its distributional convergence as $n\to\infty$ is almost completely open.

The behavior of FPP on general ground graphs is not universal. There are large finite ground graphs such that, even with mean one exponentially distributed edge-weights,  the ratio of the first-passage metric and the graph metric evaluated at a typical pair of vertices of the graph decays to 0 rapidly (in case of complete graph \cite{J99}) or slowly (in case of a family of sparse locally tree-like random graphs \cite{BHH10}) as the size of the graph grows to infinity. Here, we consider a long-range version of the FPP model on the $d$-dimensional lattice and analyze the scaling properties of the associated first-passage metric.

%%%%%%%%%%%%%%%%%%%%%%%%%%%%%%%%%%%%%%%%%%%
\subsection{Appearance of long-range first-passage percolation}

Although FPP was originally introduced and extensively studied in the nearest-neighbor settings, the long-range version of it (which we denote by LRFPP) naturally appears in many applications. For instance, theoretical biologists have used certain version of LRFPP for modeling biological invasion of species \cite{MMC11, CMM06, SRC11, FM04}. Along with many other factors they use dispersal kernels $r(\cdot)$ with heavy tails as part of their models for dispersal mechanism of biological objects (such as seeds, pollen, fungi etc.). However, most of their conclusions are based on simulations in two dimensional grid and non-rigorous heuristics. In \cite{CMM06}, followed by \cite{MMC11}, the authors have recognized two phases of spatio-temporal behavior, which they call long-distance dispersal and short-distance dispersal, based on whether the second moment of the dispersal kernel is infinite or finite. They argue that under finite second moment condition(short-distance dispersal regime) the growth behavior of the region reachable within time $t$ is same as that in nearest neighbor (or finite range) FPP. On the other hand, the authors in \cite{FM04} have recognized one additional phase, which they call medium-distance dispersal, but they haven't specified where the transitions between different phases occur. As we will prove here, the situation is much more delicate, and there are at least four distinct phases (with three critical points in between) depending on the heavy tail index of the dispersal kernel.

Aldous \cite{A10} has considered communication of continuously arriving information through a finite agent network in a certain game theoretic set-up.  In one of the cases, where the network topology is a two dimensional discrete torus and the communication cost between any two agents is a nondecreasing function of the Euclidean distance between them, the main technical tool to understand the time evolution of the fraction of informed agents is the analysis of the LRFPP model, which we propose here, on large two dimensional discrete torus. Aldous has proposed a simplified version of this LRFPP model, which he has named {\it short-long FPP}, in which agent network topology is a discrete torus, each pair of nearest-neighbor agents communicate at rate one and all other pairs of agents communicate at a rate which depends only on the size of the torus regardless of the distance between the agents. The continuous analogue of the short-long FPP model  has been analyzed rigorously on (two dimensional) real torus \cite{CD11} and on finite Riemannian manifolds \cite{BR12}. Our model is, in a sense, a generalization of the nearest-neighbor FPP and LRP.

In some sense, ours is not the first attempt to analyze LRFPP rigorously. Mollison \cite{M72} has considered similar models in the context of spatial propagation of simple epidemics in one dimension. He has proved linear growth in one dimension when the dispersal kernel has heavy tail index higher than $3$ ($=2\cdot1+1$).

In the physics literature, long range interactions for epidemic models have been proposed as more realistic descriptions in different non-equilibrium phenomenon compared with their short-range counterparts. Grassberger~\cite{GB86} introduced a variation of the epidemic processes with infection probability distributions decaying with the distance as a power-law. The model was analyzed nonrigorously in~\cite{JOWH98} using field-theoretic calculations and in~\cite{HH98} using numerical simulations. Both the articles predict $\ga=d+2$ as the phase transition point to get short-range behavior with linear growth, which is false in dimension $2$ and above by our result. The main issue is with large but finite cutoff, where one really gets $\ga=d+2$ as the phase transition point, with the diffusion coefficient of the  infection probability distributions changing from infinite to finite.

Here we will address the general case in all dimension.

%%%%%%%%%%%%%%%%%%%%%%%%%%%%%%%%%%%%%%%%%%%
\subsection{Our Model}\label{subset:model}

In this paper, we consider a long-range first-passage percolation (LRFPP) model on $\dZ^d$ for
$d\ge  1$. We will use $\lone{\cdot}$ to denote the $\ell_{1}$--norm on $\dZ^{d}$.  Let
$\sE:=\{\la\vx\vy\ra: \vx, \vy \in \dZ^d, \vx\ne\vy\}$
be the edge set for the infinite complete graph on $\dZ^d$. The length of an edge $e=\la\vx\vy\ra \in \sE$ is taken to be $\lone{e} := \lone{\vx-\vy}$.
Since all $\ell_{p}$-\text{norms}, $ p\in[1,\infty],$  are equivalent, one can use anyone of them, however we will stick to the $\ell_1$--norm for convenience.

For a given nonnegative {\it communication rate function} $r(\cdot)$ on $\dR_{+}$, $r(\lone{e})$ will be the {\it rate of communication} through the edge $e$.
To each $e \in \sE$ we also assign an independent random weight $\go_e$, where $\{\go_e\}_{e\in\sE}$ are \iid with common distribution $F$ supported on $[0, \infty]$. The random variable
\[
    W_e:=\frac{\go_e}{r(\lone{e})}
\]
represents the amount of time needed (\ie {\it passage time}) to pass through the edge $e$, and for a finite $\sE$-path $\pi$ (consisting of edges from $\sE$) we define the corresponding passage time for $\pi$ to be
\begin{align}
    W_\pi := \sum_{e \in \pi}W_e = \sum_{e \in \pi}\frac{\go_e}{r(\lone{e})}.
\end{align}

Based on these $W_\pi$, the first-passage time $T(\vx,\vy)$ to reach $\vx\in\dZ^d$ from $\vy\in\dZ^d$ associated with the communication rate function $r$ is defined to be the minimum passage time over all finite $\sE$-paths from $\vx$ to $\vy$. More precisely,
\begin{align}\label{def:trxy}
    T(\vx,\vy) :=  \inf \{W_\pi\mid \pi \in \cP_{\vx,\vy}\} \text{ for } \vx, \vy \in \dZ^d,
\end{align}
where $\cP_{\vx,\vy}$ is the set of all finite $\sE$-paths  from
$\vx$ to $\vy$.

Clearly, this LRFPP model is a {\it stochastic growth} model and $T(\cdot,\cdot)$ is a random metric (assuming $F$ has no atom at 0) on $\dZ^d$, which we will refer to as the LRFPP metric. The first natural question related to this metric is how does  the associated LRFPP ball of radius $t$,
\begin{align*}
    \cB_t := \{\vx\in\dZ^d\mid T(\vzero,\vx) \le  t\},
\end{align*}
and its \emph{diameter} (viewing $\cB_t$ as a subset of $\dZ^{d}$)
\begin{align*}
    D_t := \sup\{\lone{\vx-\vy}\mid \vx,\vy \in \cB_t\},\ t\ge  0.
\end{align*}
grow as $t$ tends to infinity. We will address these questions for certain cases of $r(\cdot)$ and $F$.

We will primarily be concerned with the case when $r(k)=k^{-\ga}L(k)$ for some
$\ga>0$ and for some slowly varying function $L$, and $F$ is exponential distribution with mean one or its positive power. Note that when ``$\alpha=\infty$'' we get back the standard (nearest-neighbor) FPP model with \iid edge-weights, which is also known as Richardson's model when the edge-weights have mean one exponential distribution.

\begin{rem} \label{monotone}
    Our approach of constructing the LRFPP model naturally incorporates the monotonicity property in $r(\cdot)$, in the sense that if $r(\cdot)$ and $r'(\cdot)$
    are two communication rate functions such that $r(k) \ge r'(k)$ for all $k\ge 1$, then $T(\vx,\vy) \le T'(\vx,\vy)$ for all $\vx$ and $\vy$, $\cB_t \supseteq \cB'_t$ and $D_t \ge D'_t$, where $T'(\cdot,\cdot), \cB'_{t},D'_{t}$ corresponds to the rate function $r'$.
\end{rem}

%%%%%%%%%%%%%%%%%%%%%%%%%%%%%%%%%%%%%%%%%%%
\subsection{LRFPP as a long-distance dispersal model}\label{subsec:disperse}

When $\gl:=\sum_{\vx\in\dZ^d}r(\lone{\vx})<\infty$, the above LRFPP model can also be viewed as a long-distance dispersal model in the context of information propagation, infection spreading (Susceptible-Infected/SI epidemic model) and biological invasion of species. Note that for $r(k)=k^{-\ga}$, we have $\gl < \infty$ if and only if $\ga> d$. To fix idea, suppose there is an agent at every vertex of $\dZ^d$ and the agents are either occupied (informed/infected) or vacant (uninformed/healthy). Occupied sites never become vacant. Initially the agent at the origin is occupied at time $0$. Whenever an agent becomes occupied, it starts communicating at rate $\gl$. When the agent at $\vx$ communicates, it chooses a site $\vy$ independently with probability $r(\lone{\vx-\vy})/\gl$ and makes it occupied. All agents act independently of each other.

Let $\hat\cB_t$ denote the set of occupied vertices at time $t$. Clearly $\hat\cB_0=\{\vzero\}$ and it is easy to see that if $F$ is exponential with mean one, then
\begin{align*}
    (\cB_t: t \ge  0) \equald (\hat\cB_t: t\ge  0).
\end{align*}
Thus, our results can also be interpreted as growth results for the associated long-range dispersal models.

%%%%%%%%%%%%%%%%%%%%%%%%%%%%%%%%%%%%%%%%%%%
\subsection{Main Results}\label{subsec:res}

Recall that $T(\cdot,\cdot)$ and $\cB_t$ denote the (random) LRFPP metric on $\dZ^d$ and the corresponding $t$-ball associated with communication rate function $r(\cdot)$ and exponentially distributed edge-weights (passage times) for the complete graph on $\dZ^d$. Throughout the article  $r(\cdot)$ is a non-increasing function and for convenience we will assume that it has the form
\begin{align}\label{def:r}
    r(k)=k^{-\ga}L(k), k\ge 1,
\end{align}
for some $\ga \in [0, \infty)$ and for some slowly varying (at infinity) function $L(\cdot)$ satisfying
$L(1)=1$. Recall that, a function $L(\cdot)$ is slowly varying at infinity if
for any $a\in(0,\infty)$ we have $\lim_{x\to\infty}L(ax)/L(x)=1$.

The first natural question is whether all vertices become occupied (\ie percolation occurs) at some finite time or not starting from a single occupied vertex (or equivalently from finitely many occupied vertices) at time 0. Even if percolation does not occur at some finite time, it is not at all obvious whether $|\cB_t| < \infty$ \as~for any $t<\infty$ or not.

Hereafter $|A|$ denotes the size for a set $A$. Our first result shows that $\cB_t$ covers the entire $\dZ^d$ instantaneously ({\it instantaneous percolation regime}) when the communication rate function $r(\cdot)$ satisfies \eqref{def:r} for any $\ga < d$, whereas $|\cB_t| < \infty$ \as for any $t<\infty$ when \eqref{def:r} holds for any $\ga>d$.

\begin{thm}[Instantaneous percolation regime]\label{thm:ip}
    For the communication rate function $r(\cdot)$, define $A$ to be the integral $A:=\int_{1}^{\infty}x^{d-1}r(x)dx$.
    \begin{enumeratei}
        \item\label{ip:item1}  If $A=\infty$, then $\pr(|\cB_t|= \infty)=1$ for any $t>0$. In particular, if $r(\cdot)$ satisfies \eqref{def:r} for some $\alpha<d$, then for any $t>0$,
        \[
            \pr(\cB_t = \dZ^d)=1.
        \]
        \item\label{ip:item2}  If $A<\infty$, then there exists a constant $c>0$ depending only on $A$ and $d$ such that
        \begin{align*}
            \E(|\cB_t|) \le e^{ct} \text{ for all $t\ge  0$. }
        \end{align*}
    \end{enumeratei}
\end{thm}

So for $\alpha>d$ the size of the occupied set grows at a certain finite rate depending on $\ga$ and $d$. Here we show that there are many different growth regions. The first is the {\it exponential growth regime}, which is observed when \eqref{def:r} holds with $\ga=d$ and any $L(\cdot)$ satisfying some additional restrictions. Note that Theorem~\ref{thm:ip}\eqref{ip:item2} ensures that if the size of $\cB_t$ is  finite for any $t<\infty$, then it can grow at most exponentially fast.

\begin{thm}[Exponential growth]\label{thm:expg}
    Let the communication rate function $r(\cdot)$ for the LRFPP model on the complete graph with vertex set $\dZ^d$ be nonnegative and non-increasing, and satisfy \eqref{def:r} with $\ga=d$ and some $L(\cdot)$ having the properties
    \begin{align}
        \int_{1}^{\infty} \frac{L(x)}{x}dx  &<\infty \text{ and }
        \int_{1}^{\infty} \frac{-\log L(x)}{x(\log x)^2}dx <\infty.\label{eq:expg}
    \end{align}
    Then there exist constants $0< c <C <\infty$ depending on $L(\cdot)$ such
    that
    \begin{align*}
        \lim_{\norm{\vx} \to \infty} \pr\left(c \le  \frac{T(\vzero,\vx)}{\log\norm{\vx}} \le  C \right) = 1.
    \end{align*}
    Moreover, there is a constant $a>0$ such that $\E\abs{\cB_t}\ge e^{at}$ for any $t\ge 0$.
\end{thm}

 Between the two properties of $L(\cdot)$ mentioned in
 \eqref{eq:expg}, the first one ensures that the growth of the LRFPP ball is finite at any finite time, whereas the second one enables us to construct a path between any two vertices of $\dZ^d$, which are located at large $\ell_{1}$-distance away from each other, such that the passage time of the path is logarithmic in the Euclidean distance between them. The second condition of \eqref{eq:expg} arises quite naturally and is somewhat optimal. Note that the same condition also arises in case of the LRP model on $\dZ^d$ corresponding to exponential growth (see~\cite{T10}).

 Next, we focus on the {\it stretched exponential growth regime}, which is observed when \eqref{def:r} holds for some $\ga \in (d, 2d)$.

\begin{thm}[Stretched exponential growth]\label{thm:seg}
   Let the communication rate function $r(\cdot)$ for the LRFPP model on the complete graph with vertex set $\dZ^d$ be nonnegative and non-increasing, and satisfy \eqref{def:r} for some $\ga \in (d,2d)$. Define
   \[ \gD(\ga, d) :=\frac{\log 2}{\log(2d/\ga)} \in (1, \infty).\]
   Then, for any $\eps>0$, we have
    \begin{align*}
        (1) & \lim_{\norm{\vx} \to \infty} \pr\left(\gD(\ga, d)-\eps\le\frac{\log T(\vzero,\vx)}{\log\log\norm{\vx}} \le  \gD(\ga, d) + \eps\right)=1,\text{ and}\\
        (2) & \lim_{t\to\infty}\pr\left(\left|\frac{\log\log D_t}{\log t} - 1/\gD(\ga, d)\right| \le \eps \right) =1.
    \end{align*}
\end{thm}

Note that as $\ga$ increases from $d$ to $2d$, the value of $1/\gD(\ga,d)$ strictly decreases from 1 to $0$. Thus, the family of the LRFPP models, which satisfies the hypothesis of Theorem \ref{thm:seg} exhibits all possible ``stretched exponential" growth behavior.

This together with the monotonicity property of the LRFPP model (see Remark
\ref{monotone}) indicates that when $\ga\ge 2d$, the growth rate of the occupied set is slower than any stretched  exponential.
Now we present some bounds for the LRFPP metric and diameter of the
associated LRFPP ball when \eqref{def:r} holds with $\ga=2d$ and $L\equiv 1$. These bounds capture the order of
magnitude for the growth of the LRFPP ball at this critical value of $\ga$.

\begin{thm}[Log correction for $\ga=2d$]\label{thm:crit2d}
    Let $r(k)=k^{-2d}, k\ge 1,$ be the communication rate function for the LRFPP model on the complete graph with vertex set $\dZ^d$. There exist constants $0< c <C <\infty$ depending only on $d$ such  that
    \begin{align*}
        (1) & \lim_{\norm{\vx} \to \infty} \pr\left(c \le  \frac{\log T(\vzero,\vx)}{\sqrt{\log\norm{\vx}}} \le  C \right) = 1,\text{ and}\\
        (2) & \lim_{t\to\infty} \pr\left(C^{-2} \le \frac{\log D_t}{(\log t)^2} \le c^{-2} \right) = 1.
    \end{align*}
\end{thm}

Next, we show that any communication rate function satisfying \eqref{def:r} for some $\ga \in (2d, 2d+1)$ corresponds to the {\it superlinear growth regime}, in which the occupied set of the LRFPP model grows faster than linear at certain polynomial rate.

\begin{thm}[Superlinear growth]\label{thm:slg}
    Let the communication rate function $r(\cdot)$ for the LRFPP model on the complete graph with vertex set $\dZ^d$ be nonnegative and non-increasing, and satisfy \eqref{def:r} for some $\ga \in (2d , 2d+1)$. Define
    \[ \Gamma(\ga, d) := \ga-2d.\]
    Then, for any $\eps>0$,
    \begin{align*}
        (1) & \lim_{\norm{\vx} \to \infty} \pr\left(\Gamma(\ga, d)-\eps\le  \frac{\log T(\vzero,\vx)}{\log\norm{\vx}} \le \Gamma(\ga, d)+\eps \right) = 1,\text{ and}\\
        (2) & \lim_{t\to\infty} \pr\left(\left|\frac{\log D_t}{\log  t} - 1/\Gamma(\ga, d)\right| \le \eps \right) = 1.
    \end{align*}
\end{thm}

Note that as $\ga$ increases from $2d$ to $2d+1$, the value of $1/\Gamma(\ga,d)$ strictly decreases from infinity to $1$. Thus, the family of the LRFPP models, which satisfies the hypothesis of Theorem \ref{thm:slg} exhibits all possible ``superlinear" growth behavior.

Finally we show that any communication rate function satisfying \eqref{def:r} for some $\ga>2d+1$ corresponds to {\it linear growth regime}, in which the growth of the occupied set in the LRFPP model is similar to that of the standard (nearest-neighbor) first-passage percolation model.

\begin{thm}[Linear growth]\label{thm:lg}
    Let the communication rate function $r(\cdot)$ for the LRFPP model on the complete graph with vertex set $\dZ^d$ be nonnegative and non-increasing, and satisfy \eqref{def:r} for some $\ga > 2d+1$. Then, for any $\vx\in\dZ^{d} \setminus \{\vzero\}$ there exists $\nu(\vx)>0$ such
    that for any $\eps>0$,
    \begin{align*}
        \lim_{n\to \infty} \pr\left((1-\eps)\nu(\vx) \le  n^{-1}T(\vzero,n\vx) \le  (1+\eps)\nu(\vx)\right) = 1.
    \end{align*}
        Moreover, $\nu(\cdot)$ can be extended to a function $\boldsymbol\nu: \dR^d\mapsto [0, \infty)$, for which $\boldsymbol\nu(\vy) \ne 0$ whenever
    $\vy \ne 0$ and
    \begin{align*}
        \pr\bigl( \{\vy \in \dR^d: \boldsymbol\nu(\vy)\le 1-\eps\} \subseteq t^{-1}\cB_{t}\subseteq &\{\vy \in \dR^d : \boldsymbol\nu(\vy)\le 1+\eps\}\\
        & \text{ for all sufficiently large } t\bigr)=1
    \end{align*}
for all $\eps>0$.
\end{thm}

Note that when the communication rate function $r(\cdot)$ satisfies \eqref{def:r} with $\ga=2d+1$, then comparing with the growth for other values of $\ga$ and using the monotonicity property of the LRFPP model (see Remark
\ref{monotone}) it is easy to see that $T(\vzero,\vx)$ grows like $\norm{\vx}^{1+o(1)}$ as $\norm{\vx}\to\infty$. However, our current techniques do not yield the exact growth rate in this case. Based on the results from \cite{gm08}, we believe the condition
$
    \textstyle \int_{1}^{\infty}\bigl(\int_{t}^{\infty}x^{d}r(x)dx\bigr)^{1/d}dt<\infty
$
is sufficient and the condition $\int_{1}^{\infty}x^{2d}r(x)dx<\infty$ is necessary to ensure linear growth for the LRFPP metric with communication rate function $r(\cdot)$. However, here we have not pursued the problem of finding necessary and sufficient conditions that will imply  linear growth.

\begin{rem}
If the communication rate function $r(\cdot)$ satisfies
\[ \lim_{k\to\infty} \frac{\log r(k)}{\log k} = -\infty,\]
then combining the monotonicity property of the LRFPP model (see Remark
\ref{monotone}) together with Proposition \ref{induction} below, it is easy to see that the conclusion of Proposition \ref{induction}, and hence that of Theorem \ref{thm:lg}, also hold for $r(\cdot)$.
\end{rem}

\begin{rem}\label{rem:pexp}
 Our proofs for the above theorems can be extended to the case where  the common distribution of $\{\go_e\}_{e\in\sE}$ is a positive  power $\gc$ of the exponential distribution. In that case, similar  phase transitions occur for the associated LRFPP metric, but the phase transition points are $d\gc, 2d\gc$ and $2d\gc+1$ instead  of $d, 2d$ and $2d+1$ respectively. Also, the corresponding growth exponents for the  stretched exponential and superlinear growth regimes are $\gD(\ga, d\gc)$ and $\Gamma(\ga, d\gc)$. %Also, the same result holds if we replace the square lattice by a Poisson Process of rate one in $\dR^{d}$, as
\end{rem}
%%%%%%%%%%%%%%%%%%%%%%%%%%%%%%%%%%%%%%%%%%%
\subsection{Heuristics behind the thresholds}\label{subsec:heu}

In this section, we provide an intuitive explanation for the existence of different phase transition points. Comparing with the `$\ga=\infty$ case' (the nearest-neighbor FPP model), it is easy to see that the growth of the LRFPP balls are always linear or faster than linear. For simplicity, assume that $r(k)=k^{-\ga}, k\ge 1,$ for some $\ga>0$.
Note that when $\ga<d$, $\sum_{\vx\in\dZ^d} r(\lone{\vx})=\infty$ and hence $\min_{\vx\in\dZ^d} W_{\la \vzero,\vx \ra}=0$ a.s. Thus $|\cB_t|=\infty$ a.s.~for any $t>0$.

Now suppose that the growth of the LRFPP ball is polynomial for some $\ga$ with growth exponent $\gb=\gb_{\ga}$ in the sense that the Euclidean diameter of the occupied set $\cB_t=\cB^{(\ga)}_{t}$ at time $t$ is of order $t^{\gb}$. Clearly we must have $\gb\ge 1$.
Then, the size of $\cB_t$ at time $t$ is of order $t^{d\gb}$, so the minimum weight among all edges which have one end in $\cB_t$ and have length more than $\ell$ is exponential with rate approximately of order $t^{d\gb}\ell^{-(\ga-d)}$.
Now note that if $\ell\gg t^{\gb}$, then this minimum weight edge must have weight more than $O(t)$ w.h.p., otherwise the growth of the LRFPP ball will be faster than $O(t^{\gb})$.
Thus, using the fact that $\pr(X\ge t)=e^{-\gl t}$ when $X$ is exponentially distributed with rate $\gl$,
 we must have $t\cdot t^{d\gb}\cdot t^{-(\gb+\eps)(\ga-d)}=o(1)$ as $t\to\infty$ for any $\eps>0$, which implies $\gb(\ga-2d)\ge 1$.

Clearly for  $\ga\le 2d$, the above heuristic calculation does not hold (as $\ga-2d\le 0$), which implies that the growth is faster than any polynomial.
In fact, comparing with a LRP model, in which an edge $e \in \sE$ is present with probability $1-\exp(-\lone{e}^{-\ga})$ it is easy to see \cite{B04a} that the growth is at least stretched exponential.
Thus, one expects a transition from stretched exponential to polynomial growth as $\ga$ changes from smaller than $2d$ to larger than $2d$.

Now, for $\ga\in (2d,2d+1]$ we have $(\ga-2d)^{-1}\ge 1$, so the growth exponent $\gb$ for the diameter of the LRFPP ball is bigger than $(\ga-2d)^{-1}$.
Also, for $\gb:=(\ga-2d)^{-1}$ intuitively the growth of the diameter cannot be faster than $t^{\gb+\eps}$ for any $\eps>0$, as eventually by time $t$ all  ``usable'' edges will have Euclidean length $\ll t^{\gb}$. Thus, for $\ga\in (2d,2d+1]$ the growth exponent for the diameter of the LRFPP ball must be $(\ga-2d)^{-1}$.

Now note that $(\ga-2d)^{-1}=1$ when $\ga=2d+1$, so by monotonicity one expects the growth exponent for the diameter to be $\le 1$ when $\ga>2d+1$. So, the linear growth dominates in this case. Moreover, it is easy to see that for any $\ell \gg t^{(d+1)/(\ga-d)}$ the minimum weight among all edges which have one end in $\cB_t$ (which has linear growth) and have Euclidean length $\ell$ is larger than $O(t)$ w.h.p., and thus up to time $O(t)$ none of the edges having Euclidean length more than $O(t^{(d+1)/(\ga-d)})$ will be used. Moreover, $\theta:=(d+1)/(\ga-d)\in(0,1)$ when $\ga>2d+1$. This idea will play a crucial role in proving linear growth. If we break the lattice into boxes of length $n^{\theta}$ and if the optimal path from $\vzero$ to $n\vx$ cannot jump over boxes, then we have a nearest-neighbor path over the boxes and if the path spends $\Theta(n^{\theta})$ time in most of the boxes, the total time is $\Theta(n)$. We will use a renormalization technique to use this idea in proving the linear growth.

%%%%%%%%%%%%%%%%%%%%%%%%%%%%%%%%%%%%%%%%%%%
\subsection{Discussion and Open Problems}\label{sec:disc}
As alluded earlier, the long-range percolation model on $\dZ^d$ is
not well understood when the associated LRP graph metric is expected
to scale polynomially with the Euclidean metric. The only available
result in this context \cite{DS13} is the existence of a scaling
exponent in one dimension ensuring polynomial scaling of the LRP
metric. Also, the linear growth of the LRP metric, when relevant, is
not fully established. In this article, we have been able to
elucidate those two growth regimes (polynomial and linear) in case
of a class of long-range first-passage percolation model, which can
be thought of as a continuous analogue of the LRP model, in addition
to identifying and analyzing other growth regimes for it. For our
model, we have proved linear growth for the associated LRFPP metric
along with a {\it shape theorem} for the growth set in case of
almost all candidate communication rate functions. We have also pinned
down the growth exponent for all communication rate functions which
correspond to polynomial growth for the occupied set.

In our LRFPP model, all edges-weights are exponentially distributed
(or some power of it). So, a natural question arises: what happens
if we replace exponential distribution by an arbitrary distribution
supported on $[0,\infty)$. In many places in this article we have
used properties of the exponential distribution to facilitate our
calculations. However, the crucial fact that will imply a similar
phase transition is that the distribution of $\go_{e}$'s satisfy
$\pr(\go_{e}\le x)=\Theta(x)$ for $x\ll 1$. In general, when
$\pr(\go_{e}\le x)=\Theta(x^{s})$ for $x\ll 1$ and for some real
number $s>0$,  the phase transition points will be $d/s, 2d/s,
2d/s+1$ respectively under appropriate moment conditions. In a sense,
the ``effective'' dimension becomes $d/s$ instead of $d$ in that case. Note that
when $(\go_e)^{s}$ has exponential distribution with rate one, it is
easy to see that $\pr(\go_{e}\le x)\approx x^{s}$ and one can go
through almost all the computations in this article to see the above
phenomenon (see Remark~\ref{rem:pexp}). The general case will be dealt with in a forthcoming
article.

\begin{figure}[htbfp]
    \centering
    \includegraphics[height=5cm,page=1]{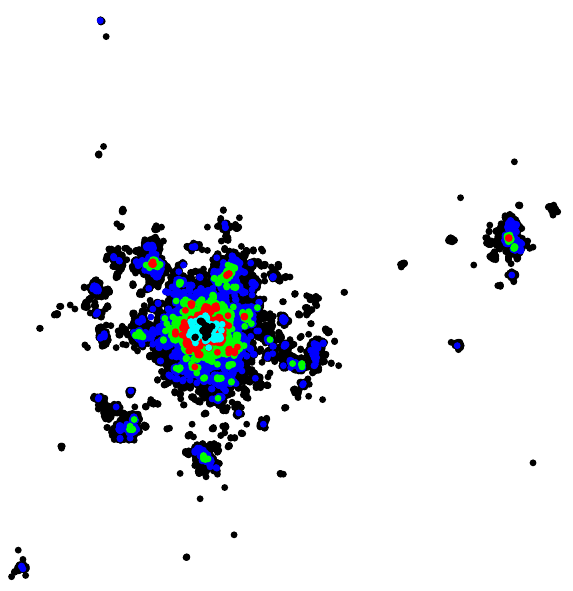}
    \caption{Simulated growth for $r(k)=k^{-3.5}$ in $d=2$ upto time $t=24$ at volume $25421$. Different colors show the growth pattern at $6$ equispaced time points.}
    \label{fig:sim3.5}
\end{figure}

\begin{figure}[hbtfp]
    \centering
     \includegraphics[height=5cm,page=2]{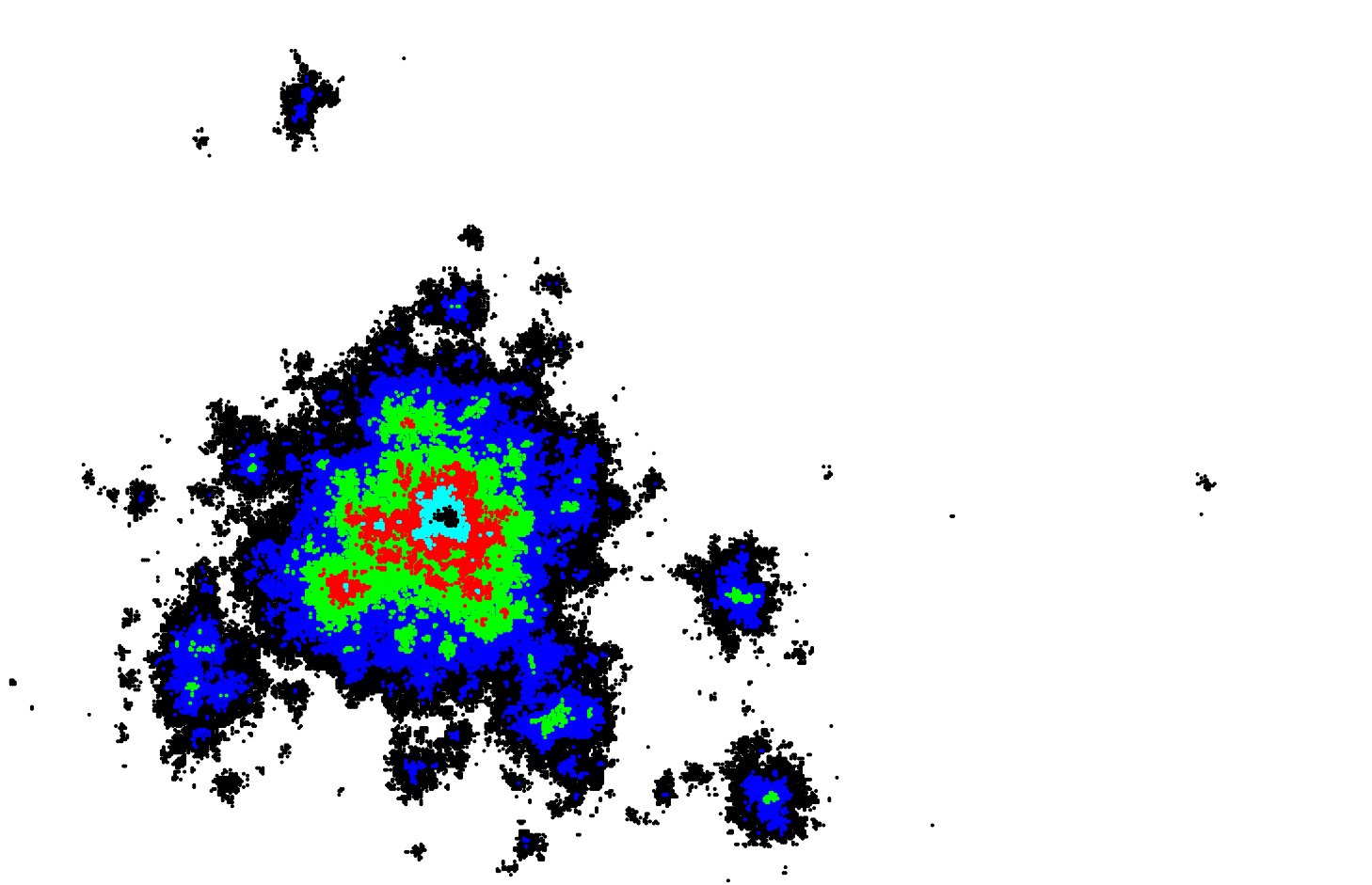}\hspace*{-2cm}\includegraphics[height=4cm,page=3]{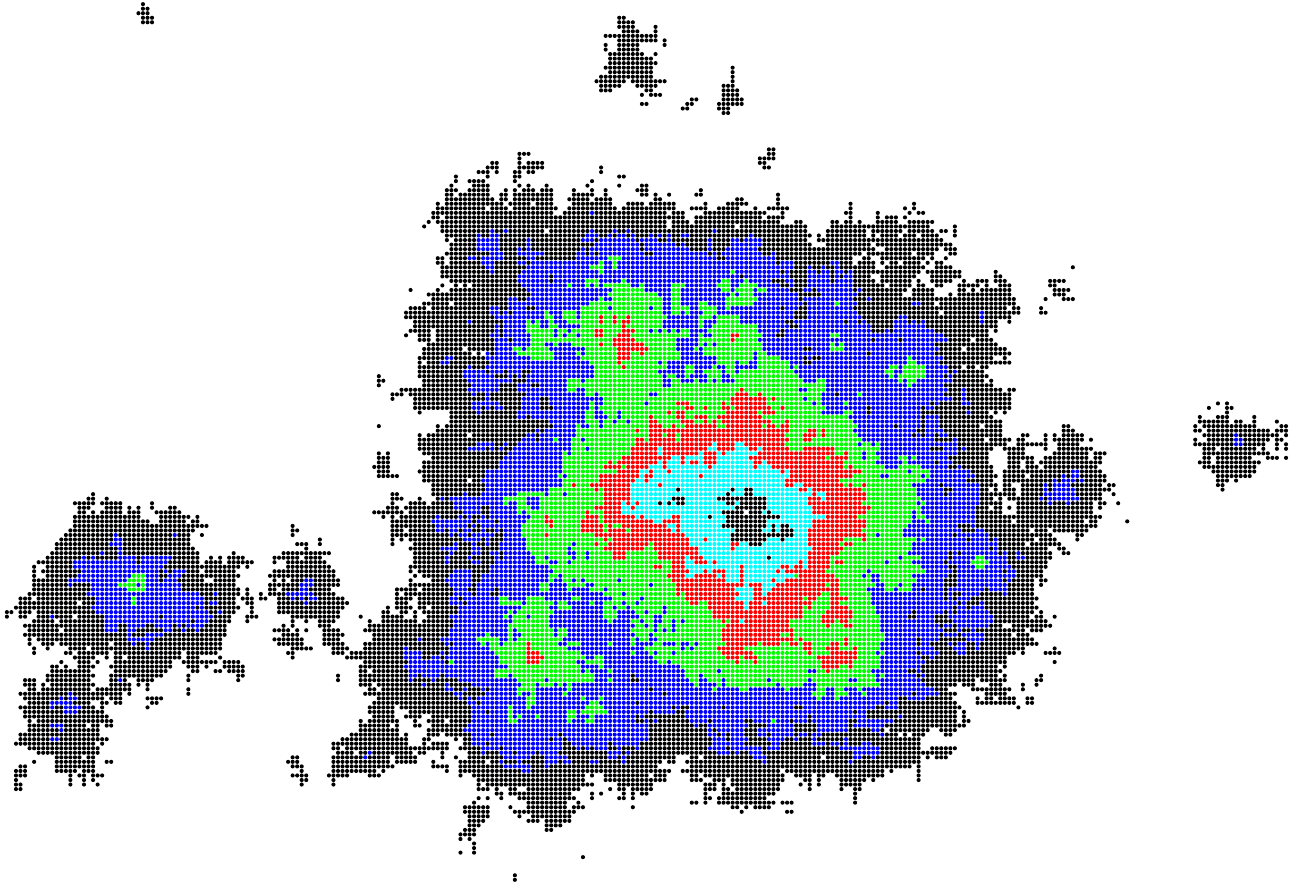}
    \caption{Simulated growth for $r(k)=k^{-4}$(left) and $r(k)=k^{-4.5}$(right) upto time $t=48$ and $t=60$ with volume $46113$ and $19635$, respectively. Different colors show the growth pattern at $6$ equispaced time points.}
    \label{fig:sim4}
\end{figure}

\begin{figure}[hbtfp]
    \centering
     \includegraphics[height=4cm,page=4]{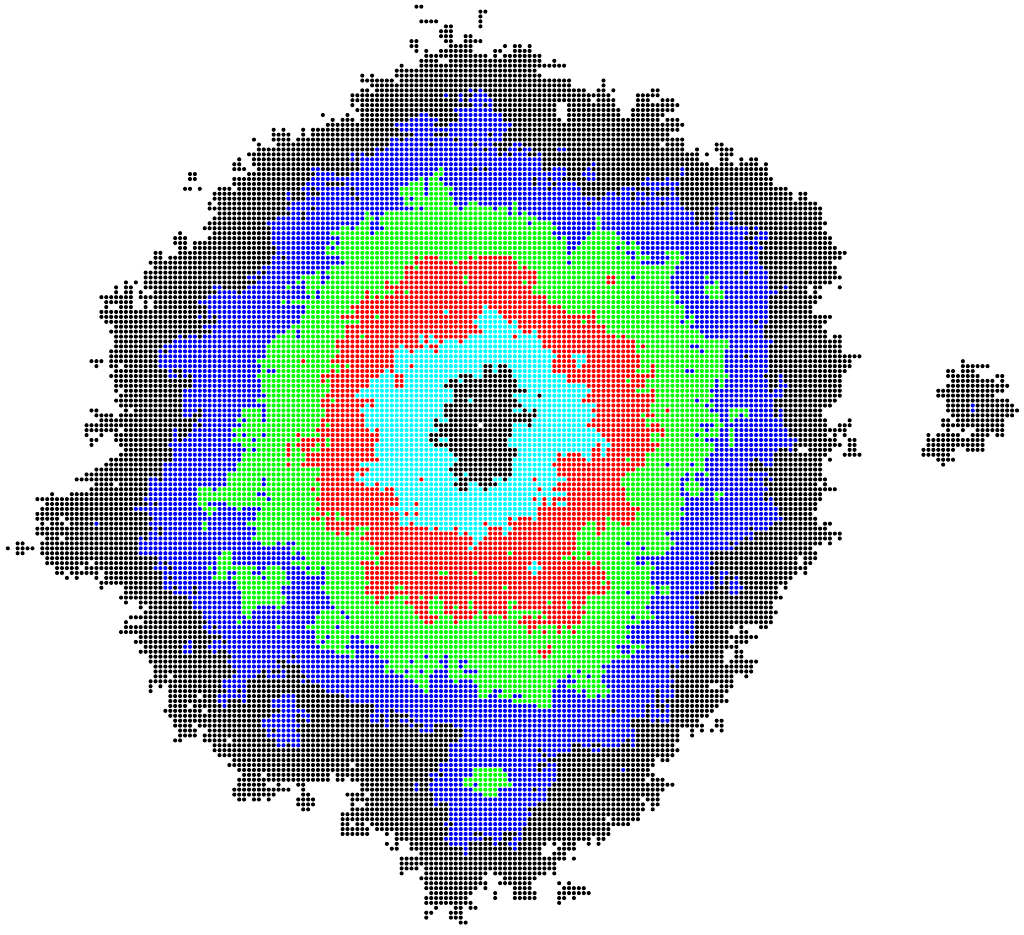}\hspace*{2cm}\includegraphics[height=4cm,page=5]{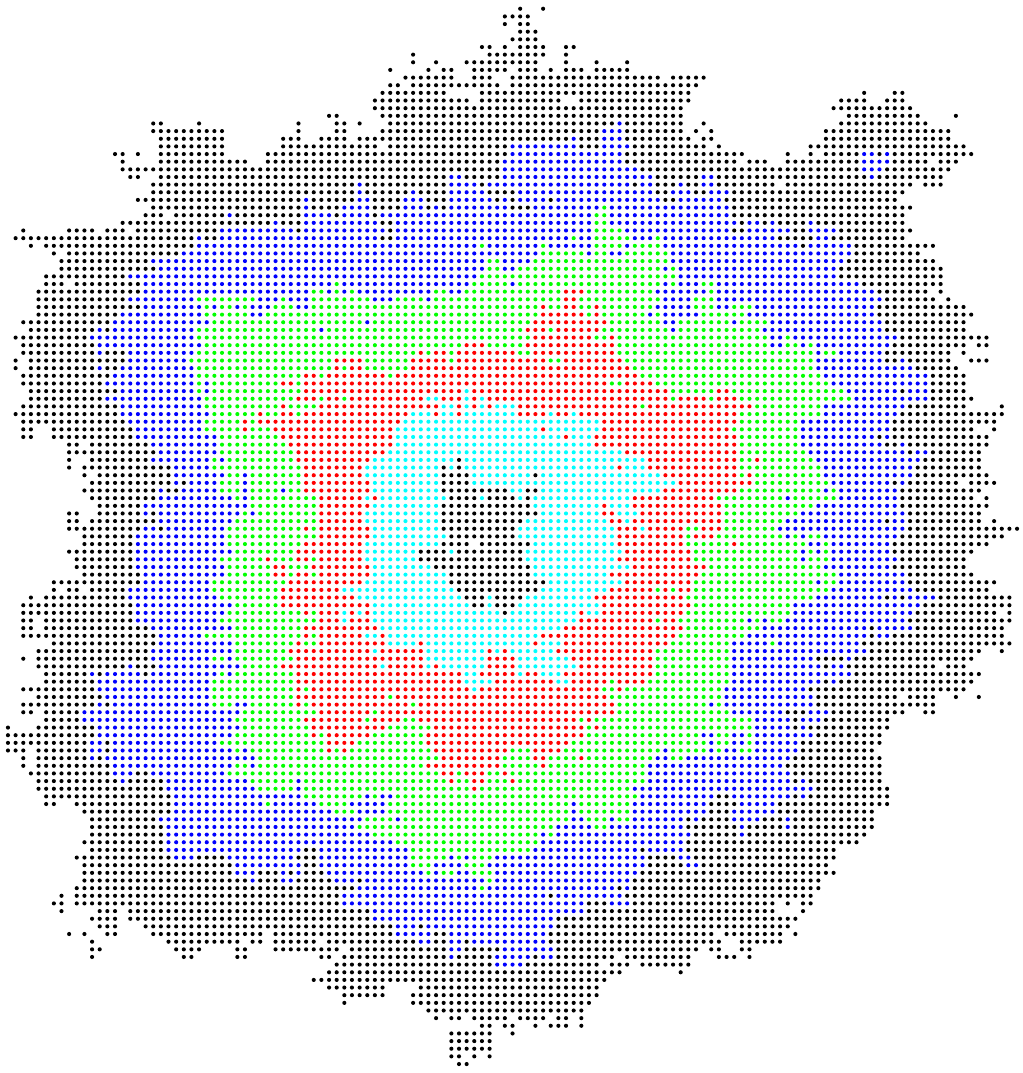}
    \caption{Simulated growth for $r(k)=k^{-5}$(left) and $r(k)=k^{-5.5}$(right) upto time $t=90$, with volume $19534$ and $12911$, respectively. Different colors show the growth pattern at $6$ equispaced time points.}
    \label{fig:sim5}
\end{figure}

Even for the model that we consider here, there are many fascinating phenomenon that we have not analyzed yet.  We mention some of them below. See Figures~\ref{fig:sim3.5}, \ref{fig:sim4} and \ref{fig:sim5} for simulated pictures of the random growth set in two dimension.

\begin{enumeratea}
    \item What is the limiting distribution of $T(\vzero,n\vx)$ as $n\to\infty$ when $\ga\in(d,2d)$? Heuristically the limit should be Gaussian at least when $\ga$ is close to $2d$.
    \item In the stretched exponential growth regime $\ga\in(d,2d)$, is it possible to formulate and analyze the  boundary behavior of the growth set?
    \item For $\ga\in(2d,2d+1)$, one should have a random `shape theorem'. More precisely, for any fixed direction $\vx\in\dR^{d}$ the ratio $T(\vzero,\lfloor n\vx\rfloor)/n^{\ga-2d}$ should converge to a random variable.
    \item How does $T(\vzero,n\vx)$ grow as $n\to\infty$ when $r(k)=k^{-(2d+1)}$? We believe the answer is $\Theta(n(\log n)^{-\theta})$ for some $\theta>0$.
    \item In the linear growth regime, does the fluctuation of the first-passage time have a phase transition too or the fluctuation is universal?
    \item From simulation results, it is obvious that for $\ga>2d$ there is a single large connected (in $\dZ^{d}$) cluster for the growth set, however for $\ga<2d$ there are many of them. Is it possible to analyze the number of ``big'' components in the growth set $\cB_{t}$?
\end{enumeratea}

%%%%%%%%%%%%%%%%%%%%%%%%%%%%%%%%%%%%%%%%%%%
\subsection{Organization of the paper}

The paper is organized as follows. In Section~\ref{sec:est} we set up our notations and prove the  technical estimates needed later in the proofs.  Section~\ref{sec:inst} contains the proof of Theorem~\ref{thm:ip} about the transition from instantaneous growth to subexponential growth. In Sections~\ref{sec:msa} and~\ref{sec:sbineq} we develop a Multi-scale analysis and Self bounding recursion for the expected volume that will be used crucially to find appropriate lower and upper bounds for the growth set at time $t$. Finally we prove the main Theorems~\ref{thm:seg}~--~\ref{thm:crit2d},~\ref{thm:slg} and \ref{thm:lg} in Sections~\ref{sec:seg},~\ref{sec:slg} and \ref{sec:lg} respectively.

%%%%%%%%%%%%%%%%%%%%%%%%%%%%%%%%%%%%%%%%%%%
\section{Notations and Estimates}
\label{sec:est}

Recall that $\dG^d=(\dZ^d, \sE)$ denotes the infinite complete graph on the vertex set $\dZ^d$ and edge set $\sE:=\{\la\vx\vy\ra: \vx, \vy \in \dZ^d, \vx \ne \vy\}$. Also $\{\go_e\}_{e\in\sE}$ is a collection of \iid exponentially distributed random variables with mean one, and the passage time for the edge $e=\la\vx\vy\ra \in \sE$ is $W_e=\go_e/r(\lone{e})$, where $\lone{e}=\lone{\vx-\vy}$. Recall that $r$ satisfies \eqref{def:r}, \ie $r$  is non-increasing and is of the form
\begin{align*}
    r(k)=k^{-\ga}L(k), k\ge 1
\end{align*}
for some $\ga>0$ and  a slowly varying function $L(\cdot)$ with $L(1)=1$.

For a finite $\sE$-path $\pi$, the passage time is defined as $W_\pi := \sum_{e\in\pi} W_e$ and the first-passage metric on $\dZ^d$ is
\begin{align*}
      & T(\vx,\vy)  := \inf_{\pi\in\cP_{\vx,\vy}} W_\pi,                                                                     \\
      & \text{where } \cP_{\vx,\vy}  := \{\la\vx_0\vx_1\ldots \vx_k\ra: \vx_0=\vx, \vx_k=\vy, \vx_i \ne \vx_{i-1}, i=1,2,\ldots,k\}.
\end{align*}
$\cB_t$ denotes the ball of radius $t$ around the origin for the random metric $T(\cdot,\cdot)$, and $D_t$ denotes the diameter of that ball.

The following tail estimate for sums of exponential random variables will be useful in our analysis.

\begin{lem}\label{lem:expldp}
    Let $X_{1},X_{2},\ldots$ be i.i.d.~exponential random variables with mean $1$. Let $\gl_{1},\gl_{2},\ldots$ be a sequence of positive real numbers. Then for any $t\ge  0$ and $k\ge  1$ we have
    \begin{align*}
        \frac{t^k}{\prod_{i=1}^k (k+\gl_{i} t)} \prod_{i=1}^k \gl_{i} \le
        \pr\left(\sum_{i=1}^k X_{i}/\gl_{i} \le  t\right) \le
        \left(\frac{et}{k}\right)^k \prod_{i=1}^k\gl_{i}.
    \end{align*}
    Moreover, if $\gl_{i}\ge  \gl$ for all $i$ and $\gL:=\sum_{i=1}^k1/\gl_{i}$, then for any $t\ge  \gL$ we have
    \begin{align*}
        \pr\left(\sum_{i=1}^k X_{i}/\gl_{i} \ge  t\right) \le  \exp\left(-\frac{\gl(t-\gL)^2}{2t}\right).
    \end{align*}
\end{lem}

\begin{proof}
Using exponential Markov inequality and the fact that
\begin{align}\label{eq:mgf}
    \E(e^{-\theta X_i}) = 1/(1+\theta), \text{ for } \theta>-1
\end{align}
we have
\begin{align*}
    \pr\left(\sum_{i=1}^k X_{i}/\gl_{i} \le  t\right) \le  e^{\theta t}
    \prod_{i=1}^k (1+\theta/\gl_{i})^{-1}\le  e^{\theta t}
    \prod_{i=1}^k \frac{\gl_i}{\theta}
\end{align*}
for all $\theta>0$. Taking $\theta=k/t$ we get the required upper bound for the lower tail of $\sum_{i=1}^k X_i/\gl_i$. For the lower bound we use independent events. Clearly
\begin{align*}
    \pr\left(\sum_{i=1}^k X_{i}/\gl_{i} \le  t\right)
      & \ge  \pr(X_{i}\le  \gl_{i}t/k \text{ for all } i) \\
      & = \prod_{i=1}^k (1-e^{-\gl_i t/k})
    \ge  \prod_{i=1}^k \frac{\gl_it}{k+\gl_it},
\end{align*}
where the last inequality follows from the fact that $1-e^{-x}\ge  x/(1+x)$ for all $x\ge  0$.

For the upper tail bound we use the Markov inequality and \eqref{eq:mgf} to have
\begin{align*}
    \pr\left(\sum_{i=1}^k X_{i}/\gl_{i} \ge  t\right) \le e^{-\theta\gl t}
    \prod_{i=1}^k (1-\theta\gl/\gl_{i})^{-1}
\end{align*}
for all $\theta\in[0,1)$. Also using the monotonicity of the function $-\log(1-x)/x$, we have
\begin{align*}
    \log \pr\left(\sum_{i=1}^k X_{i}/\gl_{i} \ge  t\right) \le  -\theta \gl t -\sum_{i=1}^k \frac{\gl}{\gl_i}\log(1-\theta)
    =-\gl(\theta t + \gL\log(1-\theta)).
\end{align*}
Taking $\theta=1-\gL/t$ and using the fact that $1-x+x\log x \ge  (1-x)^{2}/2$ for all $x\in[0,1]$ we finally have
\begin{align*}
    \log \pr\left(\sum_{i=1}^k X_{i}/\gl_{i} \ge  t\right) \le  -\frac{\gl t (1-\gL/t)^2}{2}= -\frac{\gl(t-\gL)^2}{2t}.
\end{align*}
This completes the proof.
\end{proof}

\begin{lem}\label{lem:jointprob}
    Let $X_{1},X_{2},\ldots$ be i.i.d.~exponential random variables with mean $1$. Let $\gl_{1},\gl_{2},\ldots$ be a sequence of positive real numbers. Then for any $t\ge  0,k> m\ge   0$, we have
    \begin{align*}
        \pr\left(\sum_{i=1}^k X_{i}/\gl_{i} \le  t, \sum_{i=1}^m
        X_{i}/\gl_{i} + \sum_{i=k+1}^{2k -m }X_{i}/\gl_{i}\le  t\right) \le
        \frac{(et)^{2k - m}}{(k - m)^{2k -2m}m^m} \prod_{i=1}^{2k - m}\gl_{i}.
    \end{align*}
\end{lem}

\begin{proof}[Proof of Lemma~\ref{lem:jointprob}]
It is easy to see that
\begin{align*}
      & \pr\left( \sum_{i=1}^k \frac{X_i}{\gl_i} \le  t, \sum_{i=1}^m \frac{X_i}{\gl_i} + \sum_{i=k+1}^{2k - m}\frac{X_i}{\gl_i} \le  t \right)                                      \\
      & \le   \pr\left(\sum_{i=1}^m \frac{X_i}{\gl_i} \le  t, \sum_{i=k+1}^{2k-m}\frac{X_i}{\gl_i}\le  t, \sum_{i=m+1}^k \frac{X_i}{\gl_i} \le  t\right)                           \\
      & = \pr\left(\sum_{i=1}^m \frac{X_i}{\gl_i} \le  t\right) \pr\left(\sum_{i=k+1}^{2k-m}\frac{X_i}{\gl_i}\le  t\right) \pr\left(\sum_{i=m+1}^k \frac{X_i}{\gl_i} \le  t\right).
\end{align*}
Applying Lemma \ref{lem:expldp} to bound each of the above terms we get the desired inequality.
\end{proof}

The inequalities in Lemma \ref{lem:expldp} and \ref{lem:jointprob} clearly suggests that the behavior of the tail probabilities for the passage time of a finite $\sE$-path $\pi$ depends on $\prod_{e\in\pi} r(\lone{e})$ (which corresponds to the term $\prod_i \gl_i$ in the two lemmas). So analyzing this quantity for certain collection of paths is important in order to understand the growth of the first-passage metric. Keeping that in mind, we now estimate the following

For any positive integer $k\ge  1$ and $\vx,\vy\in\dZ^{d}$, let
\begin{align}\label{cP_k}
    \begin{split}
    \cP_{k}(\vx,\vy) & \text{ be the set of all finite $\sE$-paths}              \\
                     & \text{ of length (no.~of edges) $k$ from $\vx$ to $\vy$}.
    \end{split}
\end{align}
We define
\begin{align}
    \sS^r_k(\vx,\vy):= \sum_{\pi\in \cP_{k}(\vx,\vy)}\prod_{e\in \pi}r(\lone{e}).
    \label{sS}
\end{align}
In order to estimate the growth of $\sS^r_k$, first we need the following bound.

\begin{lem}\label{lem:gsumbd}
    Let $r(\cdot), q(\cdot)$ be non-increasing functions on $\dN\to(0,\infty)$ satisfying
    \begin{align}\label{eq:rat}
        \sup_{x\ge 1}\frac{x\abs{r(x+1)-r(x)}}{r(x)}\le c,\ \sup_{x\ge 1}\frac{x\abs{q(x+1)-q(x)}}{q(x)}\le c,
    \end{align}
    for some constant $c>0$. Then, for any $\vx \in \dZ^d$ we have
    \begin{align*}
        \sum_{\vy\neq \vzero,\vx}  r(\lone{\vx-\vy}) q(\lone{\vy})
          & \le  a\biggl( r(\lone{\vx})\int_{1}^{\lone{\vx}}x^{d-1}q(x)dx \\
          & \quad + q(\lone{\vx})\int_{1}^{\lone{\vx}}x^{d-1}r(x)dx
        + \int_{\lone{\vx}}^{\infty} x^{d-1}r(x)q(x)dx
        \biggr)
    \end{align*}
    for some constant $a<\infty$ depending only on $c,d$.
\end{lem}

\begin{proof}
Let $m:=\frac{1}{4}\norm{\vx}_2$. Here $\norm{\vx}_{2}=(\sum x_i^2)^{1/2}$ is the $\ell_{2}$--norm and $\ballt{\vx,r}=\{\vy: \norm{\vx-\vy}_2\le t\}$ is the $\ell_{2}$--ball of radius $t$ centered at $\vx$. Note that $d^{-1/2}\lone{\vx}\le \norm{\vx}_{2}\le \lone{\vx}$.
Define the sets
\begin{align*}
    A_1 :=\ballt{\vzero, 3m} \setminus \{\vzero\},\
    A_2 :=\ballt{\vx, 3m} \setminus \{\vx\}
    \text{ and }
    A_3 & :=\ballt{\vx/2,\sqrt{5} m}^c.
\end{align*}

\begin{figure}[hbt]
    \centering
    \includegraphics[width=1.5in]{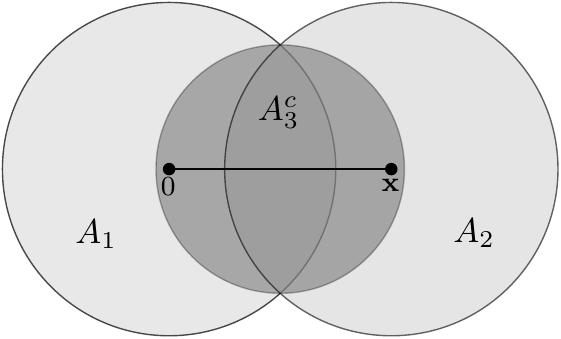}
    \caption{Decomposition of $\dZ^{d}\setminus\{\vzero,\vx\}$ into $A_{i},i=1,2,3$. }
\end{figure}

It is easy to see that $\bigcup_{i=1}^3 A_i = \dZ^d \setminus \{\vzero,\vx\}$, as the
distance between $\vx/2$ and any vertex outside $A_1\cup A_2$ is at least $\sqrt{(3/4)^2-(1/2)^2}\norm{\vx}_2=\sqrt{5} m$. Therefore we have
\[
    \sum_{\vy\neq \vzero,\vx}  r(\lone{\vx-\vy}) q(\lone{\vy})
    \le  \sum_{i=1}^{3}\sum_{\vy\in A_ i} r(\lone{\vx-\vy}) q(\lone{\vy}).
\]
Now $\vy\in A_1$ implies $\lone{\vx-\vy}\ge \norm{\vx-\vy}_2\ge m$.
Moreover the conditions \eqref{eq:rat} imply that $\sup_{k\ge
1}r(ak)/r(k)<\infty, \sup_{k\ge 1}q(ak)/q(k)<\infty$ for all $a>0$.
In particular, $r(\lone{\vx})$ and $r(\norm{\vx}_{2})$ are
equivalent upto constant multiplication.

Thus we have
\begin{align*}
    \sum_{\vy \in A_1} r(\lone{\vx-\vy}) q(\lone{\vy})
      & \le  r(m) \sum_{i=1}^{3m} i^{d-1} q(i)
    \le a r(\lone{\vx})\int_{1}^{\lone{\vx}}x^{d-1}q(x)dx
\end{align*}
for some constant $a>0$. Similarly, we have
\begin{align*}
    \sum_{\vy \in A_2} r(\lone{\vx-\vy}) q(\lone{\vy})
      & \le a q(\lone{\vx})\int_{1}^{\lone{\vx}}x^{d-1}r(x)dx.
\end{align*}
Finally, using triangle inequality we have
\[
    \norm{\vy}_2 \ge  \norm{\vy-\vx/2}_2 - 2m\text{ and  }\norm{\vx-\vy}_2\ge  \norm{\vy-\vx/2}_2 - 2m,
\]
and thus
\begin{align*}
    \sum_{\vy \in A_3} r(\lone{\vx-\vy}) q(\lone{\vy})
      & \le  a'\sum_{s\ge 5m^2} s^{d/2-1} r(\sqrt{s}-2m) q(\sqrt{s}-2m) \\
      & \le  a'' \int_{\lone{\vx}}^{\infty} x^{d-1}r(x)q(x)dx
\end{align*}
for some constant $a''$.
\end{proof}

\begin{cor} \label{cor:sumbd}
    For $\ga, \gb>0$ there exists constant $c>0$ depending on $\ga, \gb$ and $d$ such that for any $\vx \in \dZ^d$
    \begin{enumeratea}
    \item\label{smit1} $\sum_{\vy\neq \vzero,\vx}  \lone{\vy}^{-\gb} \lone{\vx-\vy}^{-\ga} \le  c \norm{\vx}^{d-\gb-\ga}$ if  $0 < \ga, \gb < d$ and $\ga+\gb > d$,

    \item\label{smit2} $\sum_{\vy\neq \vzero,\vx}  \lone{\vy}^{-\ga} \lone{\vx-\vy}^{-\ga} \le  c \norm{\vx}^{-\ga}$ if  $\ga > d$.
    \end{enumeratea}
\end{cor}

We now use Lemma \ref{lem:gsumbd} to estimate the growth of $\sS^r_k$. We use $\sS^{\ga}_{k}$ when $r(k)=k^{-\ga}$.

\begin{lem}\label{lem:sumprod}
    Let $r(\cdot)$ satisfy \ref{def:r}. For any fixed $k \ge 1$ and any $\vx,\vy\in\dZ^d$,
    \begin{enumeratea}
        \item\label{it:spa} $\ga\le  (k-1)d/k$ implies $\sS^\ga_k(\vx,\vy)=\infty$,
        \item\label{it:spb} $(k-1)d/k<\ga<d$ implies
        \[
            a^{k-1}\lone{\vx-\vy}^{(k-1)d-k\ga}\le  \sS^{\ga}_k(\vx,\vy)\le  b^{k-1}\lone{\vx-\vy}^{(k-1)d-k\ga}
        \]
        for some constant $a,b>0$ depending only on $\ga$ and $d$,
        \item\label{it:spc} $A:=\int_{1}^{\infty}x^{d-1}r(x)dx<\infty$ implies that
        \[
            a^{k-1}r(\lone{\vx-\vy})\le  \sS^r_k(\vx,\vy)\le  b^{k-1}r(\lone{\vx-\vy})
        \]
        for some constant $a,b>0$ depending only on $A$ and $d$.
    \end{enumeratea}
\end{lem}

\begin{proof}[Proof of Lemma~\ref{lem:sumprod}]

Let $\vz$ be a lattice point closest to $\vx/2$. Define $\ell_{i}:=2^{i}\norm{\vx}$ for $i\ge  0$ and consider the open annulus $\cA_{i}$ around $\vz$ of in-radius $\ell_{i}$ and out-radius $2\ell_{i}$. Clearly $|\cA_{i}| \ge  c_{d} \ell_{i}^{d}$ for some constant $c_{d}>0$.

Let $\cP_i, i\ge  0$ be the set of  all paths from $\vzero$ to $\vx$ with $k$ edges where all the vertices on the path, except the first and last one,  are in $\cA_{i}$.  Clearly $\cP_{i}$'s are disjoint and $\lone{e} \le 4\ell_i$ for every edge $e$ belonging to some $\pi\in\cP_i$, so the contribution of $\cP_{i}$ in $\sS_{k}^{\ga}(\vzero,\vx)$  is $\ge  a^{k}\ell_{i}^{d(k-1)-\ga k}$ for some constant $a>0$.\\

\eqref{it:spa} For $\ga \le (k-1)d/k$, the index of $\ell_i$ is nonnegative, so summing over $i$ we get $\sS_{k}^{\ga}(\vzero,\vx)=\infty$. This proves~\eqref{it:spa}.\\

\eqref{it:spb} Note that we have already proved that
\[
    \sS_{k}^{\ga}(\vzero,\vx)\ge  \sum_{\pi \in\cP_0}\prod_{e\in\pi} \lone{e}^{-\ga} \ge a^{k}\norm{\vx}^{(k-1)d-k\ga}.
\]
To show that this is the correct order for $\ga\in ((1-1/k)d, d)$, we will use induction to show that
\begin{align}\label{indhyp1}
    \sS^\ga_L(\vzero,\vx) \le c_1^{L-1}\norm{\vx}^{(L-1)d-L\ga}
    & \text{ for all $\vx \in \dZ^d$}                   \\
    & \text{ and for any $1\le  L < d/(d-\ga)$,}\notag
\end{align}
for some positive constant $c_1$. For $L=1$, it is trivial to see
that $\sS_{1}^{\ga}(\vzero,\vx) =\norm{\vx}^{-\ga}$ and
\eqref{indhyp1} holds. Assuming \eqref{indhyp1} holds for $L=l$ and
$l+1<d/(d-\ga)$, we have
\begin{align*}
    \sS_{l+1}^{\ga}(\vzero,\vx)
    & \le  \sum_{\vy\neq \vzero,\vx} \sS_{l}^\ga(\vzero,\vy) \lone{\vx-\vy}^{-\ga}
    \le    c_1^{l-1}\sum_{\vy\neq \vzero,\vx}  \lone{\vy}^{(l-1)d-l\ga} \lone{\vx-\vy}^{-\ga}.
\end{align*}
So applying Corollary~\ref{cor:sumbd} with $\gb=l\ga-(l-1)d$, we have
\begin{align*}
    \sS_{l+1}^{\ga}(\vzero,\vx) & \le c_1^l\norm{\vx}^{d-\ga-\gb} =
    c_1^l \norm{\vx}^{ld-(l+1)\ga},
\end{align*}
and thus \eqref{indhyp1} holds for $L=l+1$. This proves \eqref{it:spb}.\\

\eqref{it:spc} Now, we move to the proof of the case when $A:=\int_{1}^{\infty}x^{d-1}r(x)dx<\infty$. To see the lower bound for $\sS_k^r(\vzero,\vx)$, it is enough to consider a path that starting from $\vzero$ moves among the set $\{\vy:\lone{\vy}=1\}$  and finally jumps to $\vx$ at the $k$-th step. For the upper bound, we follow the induction argument which leads to the proof of \eqref{indhyp1} to prove
\begin{align*}
    \sS^r_k(\vzero,\vx)\le c^{k-1}r(\lone{\vx}) \text{ for all $\vx \in \dZ^d$ for any $k\ge 1$,}
\end{align*}
where $c$ is as in Corollary~\ref{cor:sumbd}\eqref{smit2}. The main step is to bound $\sum_{\vy} r(\lone{\vy}) r(\lone{\vx-\vy})$, for which we use Corollary~\ref{cor:sumbd}\eqref{smit2}.
\end{proof}

Lemma~\ref{lem:sumprod} together with Lemma~\ref{lem:expldp} gives an estimate for the first-passage time when $\ga > d$.

\begin{lem}\label{lem:fpest}
    Assume $A:=\int_{1}^{\infty}x^{d-1}r(x)dx<\infty$. There exists a constant $c=c(A,d)>0$ such that for any $\vx\in \dZ^{d}$ and $t>0$,
    \begin{align*}
        \pr(T(\vzero,\vx)\le  t) \le  (e^{ct}-1)r(\lone{\vx}).
    \end{align*}
\end{lem}

\begin{proof}
For $\cP_{k}(\vzero,\vx)$ as defined in \eqref{cP_k}  we use union bound to have
\begin{align*}
    \pr(T(\vzero,\vx) \le  t)
      & \le  \sum_{k=1}^{\infty}\sum_{\pi\in \cP_{k}(\vzero,\vx)}\pr(W_{\pi}\le  t).
\end{align*}
Applying Lemma~\ref{lem:expldp} to bound the summands of the above display and recalling the definition of $\sS^r_k$ from \eqref{sS}, we have
\begin{align*}
    \pr(T(\vzero,\vx) \le  t)
      & \le  \sum_{k=1}^{\infty}\left( \frac{et}{k}\right)^{k} \sum_{\pi\in \cP_{k}(\vzero,\vx)} \prod_{e\in \pi} r(\lone{e}) \\
      & =\sum_{k=1}^{\infty}\left( \frac{et}{k}\right)^{k} \sS_{k}^r(\vzero,\vx)
    \le   \sum_{k=1}^{\infty} \left( \frac{ebt}{k}\right)^{k} r(\lone{\vx})
\end{align*}
for some constant $b=b(A,d)>0$, where the last inequality follows by applying Lemma~\ref{lem:sumprod}\eqref{it:spc}. The rest of the proof follows easily as $\sum_{k=1}^{\infty}\left( ebt\over k\right)^{k} \le  e^{ebt}-1$.
\end{proof}

%%%%%%%%%%%%%%%%%%%%%%%%%%%%%%%%%%%%%%%%%%%
\section{Instantaneous Percolation Regime}\label{sec:inst}

\noindent{\bf Proof of Theorem~\ref{thm:ip}}.
\eqref{ip:item1} When $A:=\int_{1}^{\infty}x^{d-1}r(x)=\infty$, it is trivial to show that $|\cB_{t}|=\infty$ for any $t>0$. So we consider the case when $r(k)=k^{-\ga}$ with $\ga<d$. It suffices to show that $\pr(T(\vzero,\vx) > \eps)=0$ for any $\eps>0$ and $\vx\in\dZ^{d}$.
To prove this assertion we will define a sequence $\{\cP_j\}_{j\ge  0}$ of subsets of $\cP(\vzero,\vx)$, which is the set of finite $\sE$-paths joining $\vzero$ and $\vx$, such that whenever $j\ne j'$, any $\pi \in \cP_j$ and $\pi' \in \cP_{j'}$ are edge disjoint, and
\begin{align} \label{Tjineq}
    T_j:=\inf\{W^\ga_{\pi}: \pi \in \cP_j\} \text{ satisfies } \pr(T_{j} > \eps)\le  1-\gd
\end{align}
for some $\gd>0$,    which does not depend on $j$.
Clearly $\{T_{j}\}_{j\ge  0}$ will be a sequence of independent random variables, so that
\[
    \pr(T(0,\vx) > \eps) \le  \prod_{j\ge  0} \pr(T_{j} >\eps).
\]
The product term equals 0 by the property of $T_{j}$, and so the desired assertion will be proved.

In order to define $\{\cP_j\}$, fix an integer $k > d/(d-\ga)$ and for $j\ge  0$ let $\ell_{j}:=2^{j}(k-1)^{j}\norm{\vx}$. Let $\vz$ be one of the lattice points closest to $\vx/2$. Also let $B_{i}^{(j)}, 1\le  i\le  k-1,$ be the annulus centered at $\vz$ and having in-radius $(2i-1)\ell_{j}$ and out-radius $2i\ell_{j}$. With these ingredients, define
\begin{align*}
    \cP_{j} := \{\pi=\la \vx_0\vx_1\ldots \vx_k\ra: \vx_0=\vzero, \vx_k=\vx, \vx_i \in B_{i}^{(j)} \text{ for } i=1, 2, \ldots, k-1\}.
\end{align*}
It is easy to see that
\begin{align}
    (a)\quad & |\cP_j| =  |B_{1}^{(j)}|\cdot |B_{2}^{(j)}| \cdot \ldots \cdot|B_{k-1}^{(j)}| \notag                                       \\
    (b)\quad & c_{i} l_j^d \le  |B^{(j)}_i| \le  C_{i} l_j^d \text{ for some constants $c_i$ and $C_{i}$, and }  \label{B^j_iprop} \\
    (c)\quad & l_j \le  \lone{e} \le  4(k-1) l_j \text{ for all $e$ belonging to some }  \pi \in \cP_j.\notag
\end{align}
In order to obtain \eqref{Tjineq} we use a standard second moment argument involving $N_j:=|\{\pi\in\cP_j: W_{\pi} \le  \eps\}|$ to have
\begin{align}\label{Tjineq0}
    \pr(T_j \le  \eps) = \pr(N_j \ge  1) \ge  \frac{(\E (N_j))^2}{\E( N_j^2)}.
\end{align}
Now using the first inequality of Lemma \ref{lem:expldp}
\begin{align*}
    \E(N_j) = \sum_{\pi \in \cP_j} \pr\left( \sum_{e\in\pi} \lone{e}^\ga \omega_e \le  \eps\right) \ge  \sum_{\pi \in \cP_j} \left(\frac{\eps}{k+\eps}\right)^k \prod_{e \in \pi} \lone{e}^{-\ga}.
\end{align*}
Combining the last inequality with \eqref{B^j_iprop},
\begin{align}\label{Tjineq1}
    \E(N_j) \ge  \left(\frac{\eps}{k+\eps}\right)^k (4(k-1)l_j)^{-\ga k} |\cP_j| \ge  A(k,\eps) \cdot l_j^{d(k-1)-\ga k}
\end{align}
for some constant $A(k,\eps) > 0$. On the other hand, noting that for the paths $\pi, \pi' \in \cP_j$ either $\pi=\pi'$ or $|\pi \cap \pi'| \le k-2$,
\begin{align*}
    \E(N_j^2) & = \sum_{\pi, \pi' \in \cP_j} \pr\left( \sum_{e\in\pi} \lone{e}^\ga \omega_e \le  \eps , \sum_{e\in\pi'} \lone{e}^\ga \omega_e \le  \eps\right)                                           \\
              & = \sum_{\pi \in \cP_j} \pr\left( \sum_{e\in\pi} \lone{e}^\ga \omega_e \le  \eps \right)                                                                                                   \\
              & \qquad + \sum_{m=0}^{k-2} \sum_{\pi, \pi' \in \cP_j: |\pi\cap\pi'|=m} \pr\left( \sum_{e\in\pi} \lone{e}^\ga \omega_e \le  \eps , \sum_{e\in\pi'} \lone{e}^\ga \omega_e \le  \eps\right).
\end{align*}
Using Lemma \ref{lem:expldp} and \ref{lem:jointprob} to bound the summands of the first and second term respectively in the right hand side of the above display,
\begin{align}\label{Tjineq2}
\begin{split}
    \E(N_j^2) &\le \sum_{\pi \in \cP_j} c(k,\eps) \prod_{e\in\pi} \lone{e}^{-\ga} \\
    &\qquad+
    \sum_{m=0}^{k-2} \sum_{\pi, \pi' \in \cP_j: |\pi\cap\pi'|=m}
    c(k,m,\eps) \prod_{e \in \pi \cup \pi'} \lone{e}^{-\ga}.
    \end{split}
\end{align}
Now \eqref{B^j_iprop} suggests that $|\cP_j| \le  \prod_{i=1}^k C_i l_j^{d(k-1)}$ and for any fixed $\pi \in \cP_j$ and $0\le  m\le  k-2$,
\begin{align*}
    |\{\pi' \in \cP_j: |\pi \cap \pi'|=m\}| \le  \prod_{i=1}^k C_{i} l_j^{d(k-1-m)},
\end{align*}
as $\pi \cap \pi'| = m$ implies that there are at most $k-1-m$ end points of edges present in $\pi'$ but absent in $\pi$. So the number of summands in the inner sum for the second term in \eqref{Tjineq2} is at most
$(\prod_{i=1}^k C_i)^2 l_j^{d(2k-2-m)}$. From \eqref{B^j_iprop} we also have that the product term in the first summand of \eqref{Tjineq2} is at most $\prod_{i=1}^{k-1} c_i^{-\ga} l_j^{-\ga k}$ and that in the second summand is at most $\prod_{i=1}^{k-1} c_i^{-2\ga} l_j^{-\ga(2k-m)}$. Hence, using the fact that $\ga<d$
\begin{align}\label{Tjineq3}
\begin{split}
    \E(N_j^2)
    & \le  \prod_{i=1}^{k-1} C_i l_j^{d(k-1)} \cdot c(k,\eps) \prod_{i=1}^{k-1} c_i^{-\ga} l_j^{-\ga k}\\
    & \qquad + \sum_{m=0}^{k-2} \prod_{i=1}^{k-1} C_i^2 l_j^{d(2k-2-m)} \cdot c(k,m,\eps) \prod_{i=1}^{k-1} c_i^{-2\ga} l_j^{-\ga(2k-m)}\\
              & \le  A'(k,\eps) (l_j^{2d(k-1)-2k\ga}+l_j^{d(k-1)-k\ga})
\end{split}
\end{align}
for some constant $A'(k,\eps)>0$. Plugging the estimates of \eqref{Tjineq1} and \eqref{Tjineq3} in \eqref{Tjineq0} and noting that $d(k-1)-k\ga>0$ by our choice of $k$ we finally have
$\pr(T_j \le  \eps) \ge  A(k,\eps)^2/(2A'(k,\eps))=: \gd$. This completes the argument. \\

\noindent\eqref{ip:item2}
We have
\begin{align*}
    \E(|\cB_t|)
      & =\sum_{\vx\in\dZ^d}\pr(T(\vzero,\vx)\le  t)                                                                   \\
      & \le  1+\sum_{\vx\in\dZ^d,\vx\neq\vzero} r(\lone{\vx})(e^{ct}-1)= 1+C_{\ga}(e^{ct}-1)\le e^{c(1\vee C_{\ga}) t}
\end{align*}
where the first inequality follows from Lemma~\ref{lem:fpest} and the second from the fact that $$\sum_{\vx\in\dZ^d,\vx\neq\vzero}r(\lone{\vx})=C_{\ga}<\infty.$$ This completes the proof of the Theorem.\qed

%%%%%%%%%%%%%%%%%%%%%%%%%%%%%%%%%%%%%%%%%%%
\section{Multi scale analysis}\label{sec:msa}

In this section, our goal is to find suitable upper bound for the first-passage time $T(\vzero,\vx)$ in terms of $\lone{\vx}$ when $\ga\in (d, 2d+1)$. For simplicity, we will only consider the case when $r(k)=k^{-\ga}$ for $k\ge 1$.

\begin{prop} \label{prop:Tgaupbd}
    Assume that $r(k)=k^{-\ga},k\ge 1$ with $\ga \in (d,2d+1)$. Define $\gD(\ga) := 1/\log_2(2d/\ga)$ for $\ga\in(d,2d)$. For any $t>0$, there exist constants $c, C>0$ depending only on $\ga, d$ such that
    \begin{enumeratea}
    \item\label{it:upbda} $\displaystyle\pr\bigl(T(\vzero,\vx) \ge (1+t) c\norm{\vx}^{\ga-2d}\bigr)
    \le\exp\left(-\frac{Ct^2}{1+t}\right)$ for $\ga \in (2d, 2d+1)$

    \item\label{it:upbdb} $\displaystyle\pr\bigl(T(\vzero,\vx) \ge c(1+t) \exp\left(2\sqrt{2d\log 2\log\norm{\vx}}\right)\bigr)
    \le\exp\left(-\frac{t^2}{1+t}\exp(C\sqrt{\log n})\right)$ for $\ga=2d$ and

    \item\label{it:upbdc} $\displaystyle\pr\bigl(T(\vzero,\vx) \ge  (1+t) c(\log\norm{\vx})^{\gD(\ga)}\bigr)
    \le  \exp\left(-C\frac{t^2}{t+1} (\log\norm{\vx})^{\gD(\ga)}\right)$ for $\ga \in (d, 2d)$.
    \end{enumeratea}
\end{prop}

In order to prove Proposition \ref{prop:Tgaupbd}, we look at a general ansatz for the optimal path that will give us an appropriate upper bound
for the minimum time to reach a point from the origin.
The idea is to get hold of the minimum among all functions $f:\dR_{+}\to\dR_{+}$
such that the longest edge in the optimal path joining any two points separated by Euclidean distance $n$ from each other connects
the Euclidean balls of radius $f(n)$ around those two points with high probability as $n$ increases to infinity. Identifying this will
enable us to understand the structure of some near-optimal paths, and hence to obtain upper bound for the minimum time to communicate
between two points.

Let $\ball{\vy,k}$ denote the $\ell_\infty$--ball of radius
$k$ around $\vy$, so the volume of $\ball{\vy,k}$ is at least $c
k^d$ for some constant $c$. Fix a point $\vx$ with  $\norm{\vx}=n$.
Obviously
\begin{align*}
    \pr(T(\vzero,\vx) \ge  t) \le  \pr(W_{\pi(\vx)} \ge  t)
\end{align*}
for any (possibly random) path $\pi(\vx)$ joining $\vzero$ and
$\vx$. We will work with some particular choices of $\pi(\vx)$, for
which first we need to introduce some notations.

Fix a function $f:\dR_{+}\to\dR_{+}$ such that $f(x)<x/2$ for all
$x\ge  1$, and let $f_{0}=n$ and $f_{k}=f(f_{k-1})$ inductively for
$1\le  k\le  K:=\max\{k: f_k\ge  1\}$. Define
\begin{align}\label{eq:mscale}
    \vu_0 := \vzero, \vu_1 := \vx, B_0 := \ball{\vu_0 , f_1}, B_1 := \ball{\vu_1 ,
    f_1},
\end{align}
and let $\vu_{01} \in B_0$ and $\vu_{10} \in B_1$ be random vertices such that the edge $\la\vu_{01}\vu_{10}\ra$ has minimum passage time
among all edges connecting the two Euclidean balls $B_0$ and $B_1$, \ie
\begin{align*}
    \la\vu_{01}\vu_{10}\ra := \argmin \left\{W_{\la\vu\vv\ra}: \vu \in B_0, \vv \in B_1 \right\}.
\end{align*}
In general, for $i\ge  0$ and $\gs \in \{0,1\}^i$ we identify $\vu_{\gs 0}$ with $\vu_{\gs 00}$ and $\vu_{\gs 1}$ with $\vu_{\gs 11}$; then we inductively define
\begin{align*}
    B_{\gs j} := \ball{\vu_{\gs j} , f_{i+1}} \text{ for }
    j \in \{0, 1\},
\end{align*}
and let $\vu_{\gs 01} \in B_{\gs 0}$ and $\vu_{\gs 10} \in B_{\gs 1}$ be random points such that
\begin{align*}
    \la\vu_{\gs 01}\vu_{\gs 10}\ra := \argmin \left\{W_{\la\vu\vv\ra}: \vu \in B_{\gs 0}, \vv \in B_{\gs 1}\right\}.
\end{align*}
We denote the length of $\gs$ by $\lone{\gs}$, that is
$\lone{\gs}:=i$ for $\gs \in \{0, 1\}^i$.

Now we define a collection of finite $\sE$-paths
$\{\hat\pi_k\}_{k=1}^K$ joining $\vzero$ and $\vx$ as follows. Since
$\vu_{\gs 00}, \vu_{\gs 01} \in B_{\gs 0}$ and $\lone{\vu_{\gs 00} -
\vu_{\gs 01}} \le f_{\lone{\gs}+1}$, there are nearest-neighbor
paths of length at most $f_{\lone{\gs}+1}$ joining $\vu_{\gs 00}$
and $\vu_{\gs 01}$ and staying inside $B_{\gs 0}$. Choose one such
path $\pi_{\gs 0}$. Similarly, choose one nearest-neighbor path
$\pi_{\gs 1}$ of length at most $f_{\lone{\gs}+1}$ joining $\vu_{\gs
10}$ and $\vu_{\gs 11}$ and staying inside $B_{\gs 1}$. Using the
edges $\{\la\vu_{\gs 01} \vu_{\gs 10}\ra: \gs \in \cup_{i \ge 0}
\{0,1\}^i\}$ and the path segments $\{(\pi_{\gs 0}, \pi_{\gs 1}):
\gs \in \cup_{i \ge 0} \{0,1\}^i\}$ as ingredients  define the paths
$\{\hat\pi_k\}_{1\le k\le K}$ by
\begin{align*}
    \hat\pi_k := \bigcup_{\gs \in \{0, 1\}^{k-1}} (\pi_{\gs 0} \cup \pi_{\gs1}) \bigcup_{i=1}^k \bigcup_{\gs \in \{0, 1\}^{i-1}} \la\vu_{\gs 01} \vu_{\gs10}\ra.
\end{align*}

In words, $\hat\pi_k$ consists of the nearest-neighbor path segments
$\pi_{\gs 0}$ (which stays inside the balls $B_{\gs 0}$ and connects
$\vu_{\gs 00}$ and $\vu_{\gs 01}$) and $\pi_{\gs 1}$ (which stays
inside the balls $B_{\gs 1}$ and connects $\vu_{\gs 10}$ and
$\vu_{\gs 11}$) for $\gs \in \{0, 1\}^{k-1}$ (there are $2^k$ such
path segments, all having length at most $f_k$) and the edges
$\la\vu_{\gs 01} \vu_{\gs10}\ra$ connecting the balls $B_{\gs 0}$
and $B_{\gs 1}$ for $\gs\in \cup_{i=0}^{k-1} \{0, 1\}^i$. See
Figure~\ref{fig:path} for a pictorial description of the path
$\hat\pi_3$ in the Ansatz.
\begin{figure}[htbp]
   \centering
   \includegraphics[width=3.5in]{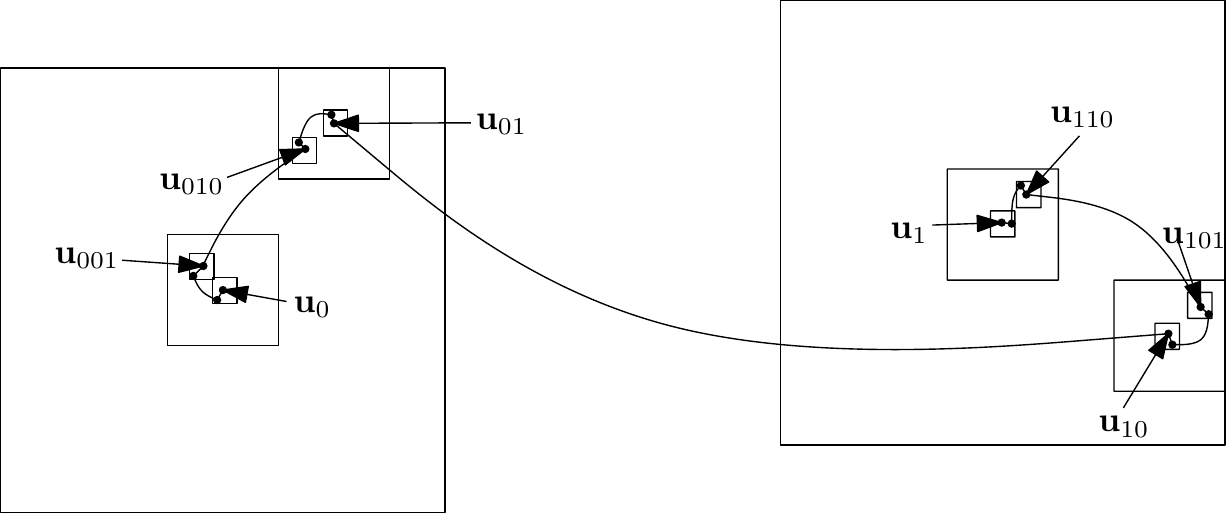}
   \caption{Pictorial description of the paths $\hat\pi_k$ in the Ansatz for $k=3$. Vertices in the small boxes are connected by nearest-neighbor paths.}
   \label{fig:path}
\end{figure}

Having defined the paths $\{\hat\pi_k\}$, we now estimate their
passage times. Basic properties of exponential distribution and the
fact that $\lone{\vu-\vv} \le f_{\lone{\gs}}+2f_{\lone{\gs}+1}$ for
$\vu\in B_{\gs 0}$ and $\vv\in B_{\gs1}$ imply that the passage time
for $\la\vu_{\gs 01}\vu_{\gs 10}\ra$ is exponentially distributed
with rate
\begin{align} \label{rateineq}
    \sum_{\vu\in B_{\gs 0}, \vv\in B_{\gs 1}} \lone{\vv-\vu}^{-\ga} \ge  c
    (f_{\lone{\gs}}+2f_{\lone{\gs}+1})^{-\ga} f_{\lone{\gs}+1}^{2d}.
\end{align}
Also the passage times of $\la\vu_{\gs 01}\vu_{\gs 10}\ra, \gs
\in\cup_{i=0}^{K-1} \{0, 1\}^i,$ are independent, as their
definition involves minimum over disjoint sets of edges.

Combining the last observation with the inequality in
\eqref{rateineq} and the fact that the passage times for the
nearest-neighbor edges are \iid and exponentially distributed with
mean one, it is easy to see that $W_{\hat\pi_k}$ is stochastically
dominated by
\begin{align*}
      & \sum_{\gs\in \{0,1\}^k} \sum_{i=1}^{f_k} X_{\gs,i} +
    \sum_{i=1}^{k} \sum_{\gs\in \{0,1\}^{i-1}}c(f_{i-1}+2f_{i})^{\ga}f_{i}^{-2d}X_{\gs},
\end{align*}
where $\{X_\gs\}$ and $\{X_{\gs,i}\}$ are \iid and exponentially distributed with mean one. Now the second assertion of Lemma~\ref{lem:expldp} with
\begin{align}\label{gL}
\begin{split}
    \gL = \gL_{f,k}
    & := c\sum_{i=1}^{k} 2^{i-1}(f_{i-1}+2f_{i})^{\ga}f_{i}^{-2d} + 2^k f_{k} \\
    \text{ and } \gl = \gl_{f,k}
    & := \left[1+\max_{1\le  i\le k}(f_{i-1}+2f_{i})^{\ga}f_{i}^{-2d}\right]^{-1}
\end{split}
\end{align}
together with the fact that
$T(\vzero, \vx) \le W_{\hat\pi_k}$ implies
\begin{align} \label{Tgabd}
    \pr(T(\vzero,\vx)\ge  (t+1)\gL_{f,k})\le  \exp\left(- \frac{t^2}{2(t+1)} \gL_{f,k} \gl_{f,k} \right)
\end{align}
for any $t \ge 0$ and for any $f:\dR_+ \to \dR_+$ such that $f(x) \le x/2$. We need to minimize the value of $\gL_{f,k}$ over the choices of the function $f$ and $1 \le  k \le K$.

\subsection{Case 1: $\ga \in (2d , 2d+1)$} (Proof of Proposition \ref{prop:Tgaupbd}\eqref{it:upbda}).

In this case, we consider the collection of functions $\{f^a(\cdot):
a>2\}$, where $f^a(x)=x/a$. The optimal choice of $a$ will be
specified later. We will use $a$ in the superscript to denote the dependence on $a$. In that case, $f^a_k=n/a^k$ and $K^a=\lfloor \log
n/\log a\rfloor$.  To understand the order of magnitude of
$\gL_{f^a,k}$ note that
\begin{align}
    \gL_{f^a,k}
        & = c\sum_{i=1}^{k} 2^{i-1}(a+2)^{\ga} a^{(2d-\ga)i} n^{\ga-2d} + 2^k n a^{-k} \notag \\
        & =
    \begin{cases}
        c(n/a)^{\ga-2d}(a+2)^{\ga} \cdot \frac{1-[2a^{2d-\ga}]^k}{1-2a^{2d-\ga}} + 2^k n a^{-k} & \text{ if $2a^{2d-\ga} \ne 1$}\\
        c(n/a)^{\ga-2d}(a+2)^{\ga} \cdot k + 2^k n a^{-k}                                       & \text{ if $2a^{2d-\ga} = 1$}
    \end{cases}\label{gL_kbd} \\
        & =
    \begin{cases}
        c_1 n^{\ga-2d}(1-[2a^{2d-\ga}]^k) + (2/a)^k n                                           & \text{ if $a^{\ga-2d}>2$}\\
        c_1 n^{\ga-2d}k + (2/a)^k n                                                             & \text{ if $a^{\ga-2d}=2$}\\
        c_1 n^{\ga-2d}[2a^{2d-\ga}]^k + (2/a)^k n                                               & \text{ if $a^{\ga-2d}<2$}
    \end{cases} \notag
\end{align}
for some constant $c_1>0$ depending on $\ga, d$ and $a$. To minimize the last expressions with respect to $k$ note that the functions
\begin{align*}
    y \mapsto
    \begin{cases}
        c_1 n^{\ga-2d}(1-[2a^{2d-\ga}]^y) + (2/a)^y n & \text{ if $a^{\ga-2d}>2$} \\
        c_1 n^{\ga-2d}y + (2/a)^y n                   & \text{ if $a^{\ga-2d}=2$} \\
        c_1 n^{\ga-2d}[2a^{2d-\ga}]^y + (2/a)^y n     & \text{ if $a^{\ga-2d}<2$}
    \end{cases}
\end{align*}
are minimized when $y=\log n/\log a +c_2$  for some constant $c_2.$ Now note that $n(2/a)^k = O(n^{\log 2/\log a})$ when $k=\log n/\log a +c$ for some constant $c$.
Keeping in mind that $k$ can be at most $K^a=\lfloor\log n/\log a\rfloor$, let $k^a=\lfloor\log n/\log a+\min\{c_2,0\}\rfloor$. Considering the dominating terms, we have
\begin{align*}
    \min_{k\le  K^a} \gL_{f^a,k} = \gL_{f^a,k^a} =
    \begin{cases}
        O(n^{\ga-2d})                          & \text{ if $a^{\ga-2d}>2$}  \\
        O(n^{\ga-2d}\log n) \gg O(n^{\ga-2d})  & \text{ if $a^{\ga-2d}=2$}  \\
        O(n^{\log 2/\log a}) \gg O(n^{\ga-2d}) & \text{ if $a^{\ga-2d}<2$.}
    \end{cases}
\end{align*}
So our choice of $a$ should satisfy $a^{\ga-2d}>2$. Therefore, \eqref{gL_kbd} implies
\begin{align*}
    \min_{a, k \le  K^a} & \gL_{f^a, k} = \gL_{f^{a_0}, k^{a_0}} = (A_{0}+o(1)) n^{\ga-2d}
\end{align*}
for
\begin{align*}
    a_0   & :=\argmin\{(a+2)^\ga/(a^{\ga-2d}-2): a^{\ga-2d} > 2\} \\
    \text{ and }
    A_{0} & :=(a_0+2)^\ga/(a_0^{\ga-2d}-2).
\end{align*}

Also, it is easy to see that $\gl_{f^{a_0}, k^{a_0}}=[1+n^{\ga-2d}(a_0+2)^\ga/a_0^{\ga-2d}]^{-1}$, so $\gL_{f^{a_0}, k^{a_0}}$ $\gl_{f^{a_0}, k^{a_0}} = C_2+o(1)$ for some constant $C_2>0$.
Therefore, replacing $(f,k)$ by $(f^{a_0},k^{a_0})$ in \eqref{Tgabd} and recalling that $\norm{\vx}=n$ we see that if $\ga \in (2d,2d+1)$, then there are constants $c(\ga), C(\ga)>0$ such that the desired bound holds.\qed

\subsection{Case 2: $\ga =2d$}(Proof of Proposition \ref{prop:Tgaupbd}\eqref{it:upbdb}).

In this case, we consider the sequences $\{a_n\}$, which satisfy $a_n \gg 1$ and $\log a_n \ll \log n$, and (following the notations of \eqref{gL}) define
\begin{align*}
    \gL(\{a_n\},k)              & := c\sum_{i=1}^{k} 2^{i-1}\left(\frac{n}{(a_n)^{i-1}}+2\frac{n}{(a_n)^{i}}\right)^{2d} \left( \frac{n}{(a_n)^{i}} \right)^{-2d} + 2^k \frac{n}{(a_n)^{k}} \\
    \text{ and } \gl(\{a_n\},k) & := \left[1+\max_{1\le  i\le
    k}\left(\frac{n}{(a_n)^{i-1}} + 2\frac{n}{(a_n)^{i}}\right)^{2d} \left(\frac{n}{(a_n)^{i}}\right)^{-2d}\right]^{-1}. \notag
\end{align*}
The particular choice of $\{a_n\}$ will be specified later.

In this case, $k$ can be at most $K^{a_n}=\lfloor \log n/\log (a_n)\rfloor$. Now note that
\begin{align*}
    \gL(\{a_n\},k) & = c\sum_{i=1}^{k} 2^{i-1}(a_n+2)^{2d}  + 2^k n a_n^{-k}
    = c(1+o(1)) 2^k a_n^{2d} + 2^k n a_n^{-k}.
\end{align*}
The two summands in
the last expression will be of same order if $a_n^{k+2d}=n$. Replacing $a_n$ by $n^{1/(k+2d)}$ the right hand side of the last display equals
\[
    c (1+o(1)) \exp\left(k\log 2 +\frac{2d}{k+2d}\log n\right),
\]
which is minimized when $k+2d \approx \sqrt{2d\log_2 n}$. So we choose $k_0:= \lfloor\sqrt{2d\log_2 n}-2d\rfloor$ and $a^0_n := n^{1/(k_0+2d)}$ and hence
\[
    \gL(\{a^0_n\}, k_0) = c(1+o(1)) \exp\left(2\sqrt{2d\log 2\log n}\right).
\]
Also, it can be easily checked that $k_0 \le K^{a^0_n}$ and
\[
    \gl(\{a^0_n\}, k_0) = (1+o(1)) (a^0_n)^{-2d} = (1+o(1))\exp\left(-\sqrt{2d\log 2\log n}\right),
\]
which makes $\gL(\{a^0_n\},k_0) \gl(\{a^0_n\},k_0)=(1+o(1))\exp(C\sqrt{\log n})$ for some $C>0$. Therefore, replacing $\gL_{f, k}$ and $\gl_{f, k}$ by
$\gL(\{a^0_n\}, k_0)$ and $\gl(\{a^0_n\}, k_0)$ respectively in~\eqref{Tgabd} and recalling that $\norm{\vx}=n$ we see that if $\ga =
2d$, then there are constants $c(d), C(d)>0$ such that the desired bound holds.\qed

\subsection{Case 3: $\ga \in (d , 2d)$} (Proof of Proposition~\ref{prop:Tgaupbd}\eqref{it:upbdc}).

In this case, we consider the collection of functions $\{f^\gc: \gc
\in (0,1)\}$, where $f^\gc(x)=x^\gc$. The optimal choice of $\gc$
will be specified later. We will use $\gc$ in the superscript to denote the dependence on $\gc$. In that case, $f^\gc_k=n^{\gc^k}$ so that
$K^\gc=\lfloor \log \log n/\log (1/\gc) \rfloor$. In order to
understand the order of magnitude of of $\gL_{f^\gc,k}$, first note
that $k\le  K^\gc$ implies that
\begin{align}
    \sum_{i=1}^k 2^{i-1} \le &2^k \le  (\log
    n)^{\log2/\log(1/\gc)} \notag\\
    \text{ and } \gc^k \ge (\log n)^{-1},
    &\text{ and hence } n^{-a\gc^k} \ge e^{-a}. \label{Kgc}
\end{align}
The definition of $\gL$ in \eqref{gL} suggests
\[
    \gL_{f^\gc,k} = \sum_{i=1}^{k} 2^{i-1}\left(1+2n^{-(1-\gc)\gc^{i-1}}\right)^{\ga} n^{-(2d\gc-\ga) \gc^{i-1}} + 2^k
    n^{\gc^k}.
\]
For $\gc<\ga/2d$, it is easy to see using \eqref{Kgc} that
\begin{align*}
    \gL_{f^\gc,k} - n^{(\ga-2d\gc)}
      & \le n^{(\ga-2d\gc)}\left[o(1) + n^{(\ga-2d\gc)(\gc-1)} \sum_{i=2}^{k} 2^{i-1} c_\gc \right] \\
      & = n^{(\ga-2d\gc)} o(1) + 2^k n^{\gc^k}.
\end{align*}
For $\gc \ge \ga/2d$, in order to understand the order of magnitude
of $\gL_{f^\gc,k}$ note that
\begin{align*}
      & 2^{-(k-1)} n^{(2d\gc-\ga) \gc^{k-1}} \gL_{f^\gc,k}
    - \left(1+2n^{-(1-\gc)\gc^{k-1}}\right)^{\ga}\\
      & = \sum_{i=1}^{k-1} 2^{-(k-i)} (1+c_\gc)^\ga
    n^{-(2d\gc-\ga)(\gc^{i-1}-\gc^{k-1})}   \le \sum_{i=0}^\infty 2^{-i}.
\end{align*}

So
\begin{align}
    \gL_{f^\gc,k}                                               & = \begin{cases}
    n^{\ga-2d\gc}(1+o(1)) +  2^k n^{\gc^k}                      & \text{ if $\gc < \ga/2d$}    \\
    C(\gc) 2^{k-1} n^{-(2d\gc-\ga) \gc^{k-1}}  +  2^k n^{\gc^k} & \text{ if } \gc \ge  \ga/2d
    \end{cases}
    \label{gL_kbd1}
\end{align}
For fixed $\gc$, the dominating terms in the last expressions are
minimized when $\gc^k=c(\gc)/\log n$ for some constant $c(\gc)$. Now
note that $2^k n^{\gc^k} = O((\log n)^{\log 2/\log(1/\gc)})$ when
$\gc^k=c(\gc)/\log n$. Keeping in mind that $k$ can be at most
$$K^\gc=\lfloor\log\log n/\log(1/\gc)\rfloor,$$ we choose
$k^\gc=\lfloor\log\log n/\log(1/\gc)+\min\{c(\gc),0\}\rfloor$, and
considering the dominating terms
\begin{align*}
    \min_{k\le  K^\gc} \gL_{f^\gc,k} \approx \gL_{f^\gc,k^\gc} =
    \begin{cases}
    O(n^{\ga-2d\gc})                 & \text{ if $\gc < \ga/2d$, }   \\
    O((\log n)^{\log 2/\log(1/\gc)}) & \text{ if $\gc\ge  \ga/2d$.}
    \end{cases}
\end{align*}
Clearly the optimal choice of $\gc$ to minimize the order of magnitude for the above expresion is $\ga/2d$. Therefore,
\eqref{gL_kbd1} will suggest
\begin{align*}
    \min_{\gc, k \le  K^\gc} \gL_{f^\gc, k} &\approx \gL_{f^{\ga/2d}, k^{\ga/2d}} = c(\ga) (\log n)^{\log 2/\log(2d/\ga)}\\
     &\qquad\text{ for some constant $c(\ga)>0$}.
\end{align*}

Also, by the definition of $\gl$ in \eqref{gL}
\begin{align*}
    \gL_{f^{\ga/2d}, k^{\ga/2d}} \gl_{f^{\ga/2d},k^{\ga/2d}}
      & \ge \frac{\sum_{i=1}^{k^{\ga/2d}}
    2^{i-1}}{1+\left(1+n^{-(1-\ga/2d)(\ga/2d)^{k^{\ga/2d}-1}}\right)^\ga}\\
      & =C(\ga) 2^{k^{\ga/2d}} = C(\ga) (\log n)^{\log 2/\log(2d/\ga)}.
\end{align*}
Plugging in the above values of $\gL$ and $\gl$ in \eqref{Tgabd}
and recalling that $\norm{\vx}=n$, we get the desired result.\qed

%%%%%%%%%%%%%%%%%%%%%%%%%%%%%%%%%%%%%%%%%%%
\section{Self-bounding Inequality for Expected Growth}\label{sec:sbineq}

In this section, we will prove an  inequality for the expected volume of the random growth set $\cB_t$ when $\ga>d$. This will lead to a lower bound for the first-passage time $T(\vzero,\vx)$ later and is inspired by one of the arguments presented in Trapman~\cite{T10}.

For simplicity we will work with the case $r(k)=k^{-\ga}, k\ge 1$ for $L(t)\equiv 1$ with fixed $\alpha>d$. For general $L(\cdot)$ the proof is similar as for $a>1$ and some slowly varying function $L(\cdot)$, we have $\sum_{k\ge n} k^{-a}L(k)=n^{1-a}\hat{L}(n)$ for another slowly varying function $\hat{L}(\cdot)$ and thus the exponents remains same if the change the slowly varying function.

Define
\begin{align} \label{gga}
    g(t):=\E|\cB_t|\text{ for } t\ge  0.
\end{align}
Theorem \ref{thm:ip} suggests that $g(t)\le  e^{ct}$ for some constant $c$ depending only on $\ga$ and $d$. We will improve upon this bound and will eventually obtain a much better one. For that we need to define
\begin{align*}
    f(k,t):=\sup_{\norm{\vx}=k} \pr(T(0,\vx)\le  t) \in [0,1]
\end{align*}
for $k,t>0$.

The following lemma proves that $f(k,t) \le k^{-\ga}h(t)$ for a suitable choice of $h(\cdot)$. Thus the contribution of the two arguments of $f(\cdot,\cdot)$ can be separated, which will be helpful in the analysis of this function.

\begin{lem}\label{lem:ht}
    For any fixed $\ga>d$ and $g$ as in~\eqref{gga}, there exist constants $c,\gd>0$ depending only on $\ga$ and $d$ such that $f(k,t) \le ck^{-\ga}h(t)$,  where
    \[
        h(t) :=  t^\ga \int_0^t g(t-y)(g(y)-1) dy +e^{-\delta t}.
    \]
\end{lem}

It is easy to see that
\begin{align*}
    g(t) = \sum_{\vx\in\dZ^d} \pr(T(0,\vx)\le  t) \le  \sum_{k=0}^{\infty} v_{d}(k)
    f(k,t),
\end{align*}
where
$v_d(k)=|\{\vx\in\dZ^d: \lone{\vx}=k\}|\le  c_d k^{d-1}$ for some constant $c_{d}>0$. So, Lemma \ref{lem:ht} together with the fact that $f(k,t) \le 1$ suggests that for any $R \ge 1$
\begin{align*}
    g(t)
      & \le  1+\sum_{k=1}^{\infty} c_d k^{d-1} f(k,t)                      \\
      & \le 1+ \sum_{k=1}^R c_d k^{d-1} + \sum_{k>R} c_d k^{d-1}k^{-\ga}h(t)
    \le  1+ c'_d R^d + c'_d h(t)R^{d-\ga}.
\end{align*}
Taking $R = c h(t)^{1/\ga}$ and simplifying we see that
\begin{align}\label{eq:rec}
    (g(t)-1)^{\ga/d} \le  c h(t).
\end{align}
for some constant $c=c(d,\ga)>0$ and $h(\cdot)$ as in Lemma~\ref{lem:ht}.

Lemma~\ref{lem:ht} together with \eqref{eq:rec} gives rise to a recursive inequality involving $g(\cdot)$. Solving this inequality we get an improved bound for $g(\cdot)$, which leads to the following bound for the $\ga$-th first-passage time $T(\vzero,\vx)$.

\begin{prop} \label{Tgalwbd}
    For $\ga>d$ there are constants $c, C>0$ depending only on $\ga$ and $d$ such that for $\gc =\log_2(2d/\ga)$
    \begin{align*}
        \log \pr(T(\vzero,\vx) \le  t) \le
        \begin{cases}
        c (\log(1+t))^{1-\gc}t^\gc(1+o(1))- \ga \log\norm{\vx}  + c         & \text{ if }\ga \in (d , 2d) \\
        \frac{4d+2}{\log 2}(\log(1+t))^2 (1+o(1))- \ga \log \norm{\vx} +c           & \text{ if } \ga=2d          \\
        \ga\left(\frac{1+\ga}{\ga-2d} \log(1+t) (1+o(1)) - \log\norm{\vx}\right) +c & \text{ if }\ga >2d.
        \end{cases}
    \end{align*}
\end{prop}

One of the main ingredients in the proof of Proposition
\ref{Tgalwbd} is the following solution of a self-bounding
inequality for positive functions.

\begin{thm}\label{thm:gtbd}
    Let $g(t):[0,\infty)\to\dR$ be a function satisfying
    \begin{align} \label{gineq}
        1 & \le  g(t) \le  e^{\gl t} \text{ and } g(t)^{1/\theta} \le
        c\left(1+t^{\gb-1}\int_{0}^{t}g(y)g(t-y)dy\right)
    \end{align}
    for all $t\ge 0$ for some constants $\gl>0, \theta\in (0,1),
    \gb\ge 0$ and $ c\ge  1$. Then, there exist a constant $c_\theta>1$
    such that $g(t)\le G(t)$ for all $t\ge 0$, where
    \[
        \log G(t)=
        \begin{cases}
            c_\theta(2\gl t)^{\log_{2}(2\theta)} \left( \log(1+t^{\gb})\right)^{\log_2(1/\theta)}(1+o(1)) & \text{ if } \theta >1/2 \\
            \frac{1}{\gb\log 2}(\log(1+t^\gb))^2 (1+o(1))                                                 & \text{ if } \theta=1/2  \\
            \frac{\theta}{1-2\theta}\log(1+t^{\gb}) (1+o(1))                                              & \text{ if }
            \theta<1/2.
        \end{cases}
    \]
\end{thm}

We will present the proof of Theorem \ref{thm:gtbd}, followed by that of Proposition~\ref{Tgalwbd} and Lemma~\ref{lem:ht} respectively.

\begin{proof}[Proof of Theorem~\ref{thm:gtbd}]
Given \eqref{gineq}, we claim that
\begin{align}\label{eq:ind}
    \log g(t)\le
    \exp\left(\frac{\theta((2\theta)^k-1)}{2\theta-1} \log\left(c(1+t^{\gb})\right) + \gl \theta^{k} t\right)
\end{align}
for all $t\ge 0$ for all $k\ge 0$. When $2\theta=1$, $((2\theta)^k-1)/(2\theta-1)$ is interpreted as $k$.
We prove \eqref{eq:ind} using induction on $k$.

The case $k=0$ follows readily from our hypothesis.
Now assume that \eqref{eq:ind} holds for $k=m$. This together with
the fact that $t\mapsto t^\gb$ is increasing in $t$ implies
\[
    \int_0^t g(t-y) g(y) \; dy \le \left[(c(1+t^\gb))^{\theta((2\theta)^m-1)/(2\theta-1)}
    \right]^2\exp(\gl \theta^m t) \int_0^t\; dy.
\]
Combining this with the inequality in \eqref{gineq} suggest that for all $t\ge 0$
\begin{align*}
    g(t) & \le
    c^{\theta}\left(1+t^{\gb-1}\int_{0}^{t}g(y)g(t-y)
    \;dy\right)^{\theta}
    \\
         & \le  c^{\theta}\left(1+t^{\gb}
    \left(c(1+t^{\gb})\right)^{2\theta((2\theta)^m-1)/(2\theta-1)}
    \exp(\gl \theta^{m} t)  \right)^{\theta}.
\end{align*}
It is easy to see that the factor multiplied with $t^\gb$ in the above expression is $\ge
1$, so the above implies
\[
    g(t) \le c^\theta (1+t^\gb)^\theta \left[\left(c(1+t^{\gb})\right)^{2\theta((2\theta)^m-1)/(2\theta-1)}
    \exp(\gl \theta^{m} t) \right]^\theta.
\]
Simplifying the expression in the right hand side we conclude
\[
    g(t) \le  \left(c(1+t^{\gb})\right)^{\theta((2\theta)^{m+1}-1)/(2\theta-1)}
    \exp(\gl \theta^{m+1} t),
\]
and thus \eqref{eq:ind} is true for $k=m+1$. This proves the claim \eqref{eq:ind}.

Having proved \eqref{eq:ind}, we will put suitable values of $k$ there to get the desired result.\\

\noindent
{\bf Case 1}. Suppose $\theta>1/2$. The optimal value of $k$ should be such that
the two terms containing $k$ in the right hand side of \eqref{eq:ind} are approximately equal. If we equate them, then we have
\begin{align*}
    \frac{2^k\theta^{k+1}}{2\theta-1}\log\left( c(1+t^{\gb})\right) = \gl \theta^{k}t
    \quad\text{ or }\quad
    k = \log_2\frac{(2-1/\theta)\gl t}{\log\left(c(1+t^{\gb})\right)}.
\end{align*}
So plugging in $k_0:=\lfloor\log_2[(2-1/\theta)\gl t/\log(
c(1+t^{\gb}))] \rfloor$ in \eqref{eq:ind} we have
\begin{align*}
    \log g(t) & \le 2\gl t\theta^{k_0} \le  2\gl t \theta^{-1} \left(\frac{(2-1/\theta)\gl t}{\log\left( c(1+t^{\gb})\right)}\right)^{\log_{2}\theta}  - \frac{\theta}{2\theta-1}\log\left( c(1+t^{\gb})\right) \\
              & \le   c_\theta(2\gl t)^{\log_{2}(2\theta)} \left( \log\left(
    c(1+t^{\gb})\right)\right)^{\log_2(1/\theta)},
\end{align*}
where $\log_\theta c_\theta = -1+\log_2(1-1/(2\theta))$.\\

\noindent{\bf Case 2}. Suppose $\theta<1/2$. Then letting $k$ go to infinity in \eqref{eq:ind}, we have
\[
    g(t)\le \left(c(1+t^{\gb})\right)^{\frac{\theta}{1-2\theta}}.
\]

\noindent{\bf Case 3}. Finally we consider the case $\theta=1/2$. Here we have
\[
    \log g(t)\le  \frac k2 \log(c(1+t^{\gb})) + \gl t 2^{-k} \le \frac k2 \log(c(1+t^{\gb})) + \gl (c(1+t^{\gb}))^{1/\gb} 2^{-k} .
\]
Similar to our approach in case 1, we will use a value of $k$ for
which the two summands in the right hand side are approximately equal. We see that the summands are
equal if $k 2^k=2\gl(c(1+t^\gb))^{1/\gb}/\log(c(1+t^\gb))$. In order to capture the dominating term, it is enough to
choose $k_0=\lfloor\log_2[\gl(c(1+t^\gb))^{1/\gb}]\rfloor$ to have
\[
    \log g(t) \le 2 \cdot \frac {k_0}{2} \log(c(1+t^{\gb}))
    =\frac{1}{\gb\log 2} (\log(1+t^{\gb}))^2(1+o(t)).
\]
This completes the proof.
\end{proof}

\begin{proof}[Proof of Proposition~\ref{Tgalwbd}]
Let $\ga>d$ be fixed. From Lemma~\ref{lem:ht} and equation \eqref{eq:rec} we have
\begin{align}
    (g(t)-1)^{\ga/d} & \le  c \biggl( t^\ga \int_0^t g(t-y)(g(y)-1) \;dy + e^{-\delta t} \biggr) \label{eq:rec2s}
\end{align}
for all $t \ge 0$ for some constant $c>0$ depending only on $\ga$
and $d$. Let $\theta :=d/\ga\in (0,1)$. Combining the previous
inequality with the fact that $g(t)^{1/\theta}\le
2^{1/\theta-1}(1+(g(t)-1)^{1/\theta})$ (by H\"older inequality), and
noting that $0\le g(y)-1<g(y)$ and $e^{-\gd t} \le 1$ we have
\[
    g(t)^{1/\theta} \le  C\biggl(1 +  t^\ga \int_0^t g(t-y)g(y) \;dy\biggr)
\]
for all $t\ge 0$ for some constant $C>1$ depending on $\ga$ and $d$. From Theorem \ref{thm:ip}\eqref{ip:item2} we also have $g(t)\le  e^{\gl t}$ for all $t\ge  0$ for some constant $\gl>0$ depending on $\ga$ and $d$. Therefore, we can apply Theorem~\ref{thm:gtbd} with $\gb=1+\ga$ and use the inequality $1+t^\gb \le (1+t)^\gb$ to have
$g(t)\le G(t)$, where
\begin{align}\label{eq:G(t)}
    \log G(t) =
    \begin{cases}
    c t^{\log_{2}(2d/\ga)} (\log(1+t))^{\log_2(\ga/d)} (1+o(1)) & \text{ if } \ga \in (d,2d) \\
    \frac{2d+1}{\log 2}(\log(1+t))^2 (1+o(1))                           & \text{ if } \ga=2d         \\
    \frac{(1+\ga)d}{\ga-2d}\log(1+t)(1+o(1))                            & \text{ if } \ga >2d
    \end{cases}
\end{align}
for some constant $c$ depending only on $\ga$ and $d$.

Now we use the definition of $f(\cdot,\cdot)$ and Lemma \ref{lem:ht} to have
\begin{align*}
    \log \pr(T(\vzero,\vx) \le  t) & \le  \log f(\norm{\vx},t)                                                                           \\
                                        & \le  c -\ga\log \norm{\vx} + \log \bigl(t^\ga \int_0^t g(t-y)(g(y)-1) \;dy +e^{-\delta t}\bigr) \\
                                        & \le  c - \ga\log \norm{\vx} + \frac{\ga}{d}\log G(t),
\end{align*}
where $G(\cdot)$ is specified in \eqref{eq:G(t)}. Here we used the fact that $G(\cdot)$ satisfies the inequality \eqref{eq:rec2s} as equality. Plugging in the expression for $\log G(t)$ we get the required result.
\end{proof}

\begin{proof}[Proof of Lemma~\ref{lem:ht}]
Fix $k,t>0$ and $\vx\in\dZ^{d}$ with $\lone{\vx}=k$. We begin by estimating $\pr(T(\vzero,\vx)\le  t)$ and then take supremum over all $\vx$ with $\lone{\vx}=k$. Let $N(\vx)$ be the number of edges in the optimal path joining $\vzero$ and $\vx$.  Breaking in terms of the magnitude of $N(\vx)$ we have
\begin{align}\label{breakup2}
    \pr(T(\vzero,\vx) \le  t) \le  \pr(T(\vzero,\vx) \le  t, N(\vx) > at) + \pr(T(\vzero,\vx) \le  t, N(\vx) \le  at)
\end{align}
for any $a>0$. We first show that for $b$ as in Lemma~\ref{lem:sumprod}\eqref{it:spb} and any $a>e\cdot b$, the first term in the right hand side of \eqref{breakup2} satisfies
\begin{align*}
    \pr(T(\vzero,\vx)\le  t, N(\vx)>at) & \le  \frac{a}{b(a-b)}e^{-a t\log(a/eb) }
    \lone{\vx}^{-\ga}.
\end{align*}
Let $N_k$ be the number of self-avoiding paths between $\vzero$ and $\vx$  which have  $k$ many edges and passage time at most $t$. Using union bound and then Markov inequality
\begin{align}
    \pr(T(\vzero,\vx)\le  t, N(\vx)>at)
      & \le  \sum_{k=at}^\infty \pr(N_k \ge  1)
      \le  \sum_{k=at}^\infty \E( N_k) \notag\\
      & \le  \sum_{k=at}^\infty \sum_{\substack{\vx_0=\vzero, \vx_k=\vx \\ \vx_1, \ldots, \vx_{k-1}\in\dZ^d}} \pr\left(\sum_{i=1}^k W_{\la\vx_{i-1}\vx_i\ra} \le  t\right).
    \label{manyedgebd}
\end{align}
In order to estimate the summands we invoke Lemma \ref{lem:expldp} to have for any $\theta>0$
\[
    \pr\left(\sum_{i=1}^k W_{\la\vx_{i-1}\vx_i\ra} \le  t\right) \le  e^{\theta t}\theta^{-k} \prod_{i=1}^k \lone{\vx_i-\vx_{i-1}}^{-\ga}.
\]
Using this bound for the summands in \eqref{manyedgebd} and applying Lemma \ref{lem:sumprod} for the inner sum, we
see that
\begin{align*}
    \pr(T(\vzero,\vx)\le  t, N(\vx)>at)
      & \le  \sum_{k=at}^\infty e^{\theta t}\theta^{-k}b^{ k-1} \lone{\vx}^{-\ga} \\
      & =b^{-1} e^{\theta t } (b/\theta)^{at}  \lone{\vx}^{-\ga}/(1-b/\theta)
\end{align*}
for all $\theta>b$. Taking $\theta=a>eb$ we have
\begin{align}\label{term2bd}
    \pr(T(\vzero,\vx)\le  t, N(\vx)>at)
      & \le  \frac{a}{b(a-b)}e^{-a t\log(a/eb) }  \norm{\vx}^{-\ga}.
\end{align}

To bound the second term in the right hand side of \eqref{breakup2}, first note that if a vertex self-avoiding path between $\vzero$ and $\vx$ of length (\ie number of edges) at most $at$ exists, then it will contain at least one edge shared between two vertices at distance $\ge  \lceil\norm{\vx}/at\rceil$ from each other. Then by Markov inequality we have
\begin{align*}
    \pr(T(\vzero,\vx) &\le t, N(\vx) \le  at)\\
    &
    \le  \sum_{\substack{\vx_1, \vx_2 \in \dZ^d \\
    \lone{\vx_1-\vx_2} \ge  \lceil\norm{\vx}/at\rceil}} \pr(T(\vzero,\vx_1)+W_{\la\vx_1\vx_2\ra}+T(\vx_2,\vx) \le  t).
\end{align*}
Recalling that the density of $W_{\la\vx_1\vx_2\ra}$ is at most $\lone{\vx_1-\vx_2}^{-\ga}\le  \lone{\vx}^{-\ga} (at)^\ga$ whenever $\lone{\vx_1-\vx_2}\ge  \lceil\lone{\vx}/at\rceil$, the right hand side of the last display is
\begin{align*}
    \le  \sum_{\vx_1, \vx_2 \in \dZ^d} \int_0^t d(\pr(T(\vzero,\vx_1) \le  s)) \int_0^{t-s}\pr(T(\vx_2,\vx)\le  y) \norm{\vx}^{-\ga}(at)^\ga \; dy.
\end{align*}
Taking the sum inside the integral the above equals
\[
\lone{\vx}^{-\ga}(at)^\ga \int_0^t \;dg(s) \int_0^{t-s} g(y) \;dy = \norm{\vx}^{-\ga}(at)^\ga\int_0^t g(y)\int_0^{t-y} dg(s) \;dy
\]
after changing the order of integration. Hence, we conclude
\begin{align}\label{term1bd}
    \pr(T(\vzero,\vx) \le  t, N(\vx) \le  at) \le  \norm{\vx}^{-\ga}(at)^\ga \int_0^t g(y)(g(t-y)-1)\; dy.
\end{align}
Combining \eqref{breakup2}, \eqref{term2bd} and \eqref{term1bd}, with $a>eb$ and $\gd=a\log(a/eb)$ we finally have
\begin{align*}
    \pr(T(\vzero,\vx)\le  t) \le  c \norm{\vx}^{-\ga}\left( t^\ga \int_0^t g(t-y)(g(y)-1) \;dy +e^{-\delta t}\right)
\end{align*}
for some constant $c=c(\ga,d)>0$ for all $\vx\in\dZ^d\setminus\{\vzero\},t>0$.
\end{proof}

%%%%%%%%%%%%%%%%%%%%%%%%%%%%%%%%%%%%%%%%%%%
\section{Stretched Exponential and Exponential Growth Regimes}\label{sec:seg}
In this section, we will put the necessary pieces together and complete the proofs of Theorems~\ref{thm:seg},~\ref{thm:expg} and ~\ref{thm:crit2d}. As before we will work in the case when $L(k)\equiv 1$, so that $r(k)=k^{-\ga}$ for $k\ge 1$. Proof for the general $L(\cdot)$ is similar as explained in Section~\ref{sec:sbineq}.

%%%%%%%%%%%%%%%%%%%%%%%%%%%%%%%%%%%%%
\subsection{Proof of Theorem \ref{thm:seg}}
The probability estimate in Proposition \ref{prop:Tgaupbd}\eqref{it:upbdc} suggests that if $c>0$ is large enough, then
\begin{align}\label{timebd1}
    \lim_{\norm{\vx}\to\infty} \pr\left(T(\vzero, \vx) \ge
    c(\log\norm{\vx})^{\gD}\right) = 0.
\end{align}
Now for any $\eps>0$ and any $\vx$ satisfying $\norm{\vx} =\lfloor\exp(t^{1/\gD-\eps})\rfloor$,
\begin{align*}
    \pr\left(\log D_t < t^{1/\gD-\eps}\right)
      & \le \pr\left(T(\vzero,\vx) > t\right) \\
      & \le \pr\left(T(\vzero,\vx) >
    (\log\norm{\vx})^{\gD/(1-\eps\gD)}\right),
\end{align*}
so using \eqref{timebd1}
$\lim_{t\to\infty} \pr(\log D_t < t^{1/\gD-\eps})=0$ for any $\eps>0$.

On the other hand, the probability estimate in Proposition \ref{Tgalwbd} suggests that
\[
    \pr\left(T(\vzero,\vx) \le (\log\norm{\vx})^{\gD-\eps}\right)
    \le \exp\left(-\ga\log\norm{\vx}(1-\varphi(\norm{\vx}))\right),
\]
where $\varphi$ is such that $\lim_{l\to\infty} \varphi(l)=0$. Hence,
\[
    \lim_{\norm{\vx} \to \infty}\pr\left(T(\vzero,\vx) \le (\log\norm{\vx})^{\gD-\eps}\right)
    =0,
\]
and using union bound
\begin{align*}
    \pr\left(\log D_t > t^{1/\gD+\eps}\right)
      & \le  \sum_{\vx: \norm{\vx} \ge \exp(t^{1/\gD+\eps})}
    \pr\left(T(\vzero,\vx) \le t\right) \\
      & \le  \sum_{k: k \ge \exp(t^{1/\gD+\eps})}
    c_d  k^{d-1} \exp\left(- \ga \log k (1-\varphi(k))\right).
\end{align*}
By the property of $\varphi$, if $t$ is large enough, then the above
is upper bounded by
$$
     \sum_{k \ge \exp(t^{1/\gD+\eps})} c_d  k^{d-1-(\ga-\gd)}
$$
for any given  $\gd>0$. Since $\ga > d$, the above series is convergent for small enough $\gd$, which implies $\lim_{t\to\infty} \pr(\log D_t > t^{1/\gD+\eps}) = 0$. This completes the proof of the Theorem.\qed

%%%%%%%%%%%%%%%%%%%%%%%%%%%%%%%%%%%%%
\subsection{Proof of Theorem~\ref{thm:crit2d}}
The probability estimate in Proposition \ref{prop:Tgaupbd}\eqref{it:upbdb} suggests that if $c>0$ is large enough, then
\[
    \lim_{\norm{\vx}\to\infty} \pr\left(T(\vzero, \vx) \ge
    c\exp\left(2\sqrt{2d\log 2\log\norm{\vx}}\right)\right)=0.
\]
Now for any $\eps>0, c<\infty$ and for any
$\vx$ satisfying $$\norm{\vx}=\lfloor\exp((1-\eps) (\log t)^2/(8d\log
2))\rfloor,$$ we have
\begin{align*}
    \pr\Bigl(\log D_t &< \left(\frac{1-\eps}{8d\log 2} (\log t)^2\right)\Bigr)\\
      & \le \pr\left(T(\vzero,\vx) > t \right)           \\
      & \le \pr\left(T(\vzero,\vx) > c\exp(2\sqrt{2d\log
    2\log\norm{\vx}})\right)
\end{align*}
provided $t$ is large enough. This together with the bound of
Proposition~\ref{prop:Tgaupbd}\eqref{it:upbdb} gives $\lim_{t\to\infty} \pr(\log
D_t < (1-\eps)(\log t)^2/(8d\log 2))=0$.

On the other hand, the probability estimate of Proposition~\ref{Tgalwbd} suggests that
\begin{align*}
    \pr\Bigl(T(\vzero,\vx) &\le \exp\bigl(\sqrt{(d-\eps)\log 2\log\norm{\vx}/(4d+2)}\bigr)\Bigr)
   \\
   &\qquad\qquad\qquad\qquad\qquad \le c\exp\left(-(d+\eps)\log\norm{\vx}\right),
\end{align*}
and so if we let $C(d,\eps)=\frac{4d+2}{(d-\eps)\log 2}$ and use union
bound, then
\begin{align*}
    \pr\bigl(\log D_t &> C(d,\eps) (\log t)^2\bigr)\\
      & \le  \sum_{\vx: \norm{\vx} \ge \exp(C(d,\eps)(\log t)^2)}\pr\left(T(\vzero,\vx) \le \exp(\sqrt{\log\norm{\vx}/C(d,\eps)}\right) \\
      & \le  \sum_{k \ge \exp(C(d,\eps)(\log t)^2)} c_d k^{d-1} \exp\left(- (d+\eps) \log k \right) \to 0
\end{align*}
as $t\to\infty$, as the above series is convergent. This completes the proof of the Theorem.\qed

%%%%%%%%%%%%%%%%%%%%%%%%%%%%%%%%%%%%%
\subsection{Proof of Theorem~\ref{thm:expg}}
As in the proof of Theorem~\ref{thm:seg} and~\ref{thm:crit2d}, we will use Lemma~\ref{lem:ht} to find a lower bound for $T(\vzero,\vx)$ and multi-scale analysis to find a matching upper bound. Using Theorem~\ref{thm:ip}\eqref{ip:item2} under the assumption that $r(k)=k^{-d}L(k)$, $k\ge 1$ with $\int_{1}^{\infty}x^{-1}L(x)dx $ $<\infty$ we have
\[
g(t):=\E\abs{\cB_{t}}\le e^{\gl t}, t\ge 1
\]
for some $\gl<\infty$. Combining with the result from Lemma~\ref{lem:ht} we thus have
$$
    \pr(T(0,\vx)\le  t)  \le cr(\lone{\vx})h(t)
$$
where
\[
    h(t) :=  t^d \int_0^t g(t-y)(g(y)-1) dy +e^{-\delta t}\le e^{bt}
\]
for some $b>0$. Hence, for any $\eps\in(0,1)$, we have
\[
    \pr(T(0,\vx)\le  (d -\eps)/b\log\lone{\vx})  \le c L(\lone{\vx})\lone{\vx}^{-\eps}
\]
for all $\vx\in\dZ^d\setminus\{\vzero\}$.

For the upper bound on $T(\vzero,\vx)$ we will use multi-scale analysis to construct a path that achieves the $\log\lone{\vx}$ lower bound. We define the function
\[
f(x) = \frac{\sqrt{x}}{L(x)^{1/d}(\log x)^{1/2d}},\ x\ge 2.
\]
The choice of this function is not arbitrary and is almost optimal as seen from the arguments below.

Using this function, as done in~\eqref{eq:mscale}, we construct a path $\pi(\vx)$ from $\vzero$ to $\vx$. where $f_{0}=n$ and $f_{i}=f(f_{i-1})$ for $i=1,2,\ldots$. Using~\eqref{gL} we have
\begin{align}
\begin{split}
    \gL
    :=\ &c\sum_{i=1}^k 2^{i-1} f_{i-1}^d f_{i}^{-2d}/L(f_{i-1}) + 2^k f_{k}\\
    =\ & c\sum_{i=1}^k 2^{i-1} L(f_{i-1})\log f_{i-1} + 2^k f_{k}\\
    \text{ and }
    \gl
     :=\ & \left[1+\max_{1\le  i\le k} f_{i-1}^d f_{i}^{-2d}/L(f_{i-1})\right]^{-1}\\
      =\ & \left[1+\max_{1\le  i\le k} L(f_{i-1})\log f_{i-1}\right]^{-1}=\Theta(1).
\end{split}
\end{align}
Now it is easy to see that $n^{1/2^i}\le f_i$ for all $i\ge 0$. If we can show that
\begin{align}\label{eq:ts}
f_{i}\le n^{c/2^i}, i\ge 0
\end{align}
for some $c\in[1,\infty)$, then we have $2^i\approx \log n/\log f_{i-1}, i\ge 1$ and hence
\[
    \gL \approx \left(c\sum_{i=1}^k  L(f_{i-1}) + 1/\log f_{k}\right)\log n = \Theta(\log n)
\]
as $\int_{A}^{A^c}x^{-1}dx=\log c$ and $\int_{1}^{\infty}x^{-1}L(x)dx<\infty$. Then we are done by~\eqref{Tgabd}.  We claim that \eqref{eq:ts} holds when
\[
    \int_{1}^{\infty} \frac{-\log L(x)}{x(\log x)^2}dx <\infty.
\]
We have
\begin{align*}
\frac{\log f_i}{(1/2)\log f_{i-1}} = 1 - \frac{2\log L(f_{i-1}) + \log\log f_{i-1}}{d\log f_{i-1}}.
\end{align*}
Thus~\eqref{eq:ts} holds when
\[
\sum_{i=1}^{k} \frac{- \log L(f_{i-1}) + \log\log f_{i-1}}{\log f_{i-1}}
\]
is bounded by a finite constant independent of $n$. Note that $2^k\approx \log n$.
Now
\[
\int_{A}^{A^c} \frac{dx}{x(\log x)^2} = \frac{1-c^{-1}}{\log A}
\]
and the proof follows by induction on $i$. We leave the exact calculation to the interested reader. The lower bound on $\E\abs{\cB_{t}}$ follows by comparison to LRP as given in \cite[Theorem~1.1(b)]{T10}. \qed

%%%%%%%%%%%%%%%%%%%%%%%%%%%%%%%%%%%%%
\section{Super-linear Growth Regime}\label{sec:slg}
\noindent{\bf Proof of Theorem~\ref{thm:slg}}: As before, for
simplicity we will restrict ourselves to the rate function
$r(k)=k^{-\ga}, k\ge 1$ where $\ga\in(2d,2d+1)$.  It follows easily
from Proposition~\ref{prop:Tgaupbd}\eqref{it:upbda} that there is a
constant $c>0$ such that
 $\pr(T(\vzero,\vx) \ge t\norm{\vx}^{\ga-2d})\le e^{-ct}$ for all $t$ large enough. This in turn implies that
\[
    \lim_{t \to \infty} \pr(\log D_t \le (1/(\ga-2d)-\eps)\log t)=0.
\]

For the other direction we will prove using an induction argument that there is a constant $C\ge 1$ (to be chosen later) and a recursively defined sequence $(\gc_k,k\ge 0)$  (see~\eqref{gc_m} for the precise definition) satisfying
 \begin{align} \label{gammaprop}
 \begin{split}
 &(a)\quad (\gc_k)_{k\ge0} \text{ is decreasing } \\
 &(b)\quad \gc_k > \frac{1}{\ga-2d} \text{ for all } k \text{ and } \\
 & (c)\quad \gc_k \to \frac{1}{\ga-2d} + \eps/2,
  \end{split}
 \end{align}
 so that
 \begin{align}\label{recursion}
    \lim_{t \to \infty} \pr(\cB_t \subseteq B(\vzero, Ct^{\gc_k})) = 1.
 \end{align}
 Then, choosing $k$ large enough so that $\gc_k
 < 1/(\ga-2d)+\eps$ and applying \eqref{recursion} the proof of the
 theorem will be complete.

 To emphasize the dependence on $\ga$, if
 necessary, we will use the notations $T^{(\ga)}(\cdot,\cdot)$, $\cB^{(\ga)}_t$,
 $D^{(\ga)}_t$ instead of $T(\cdot,\cdot), \cB_t, D_t$, respectively
 when the rate function is $k^{-\ga}, k\ge 1$.

In order to initiate the induction argument for \eqref{recursion} we will use the probability estimate of Proposition \ref{Tgalwbd} for $\ga \in (2d, 2d+1)$. Keeping that in mind, we choose and fix any $\gc_0$ satisfying
\begin{align} \label{gc_0}
    \gc_0 > \frac{\ga(1+\ga)/(\ga-2d)}{\ga-d-1}.
\end{align}
Now note that if $\cB_t \not\subseteq \ball{\vzero, t^{\gc_0}}$, then there is at least one $\vx$ with $\norm{\vx} \ge t^{\gc_0}$ such that $T(\vzero,\vx) \le t$, so using union bound
\[
    \pr(\cB_t \not\subseteq \ball{\vzero, t^{\gc_0}}) \le \sum_{k\ge t^{\gc_0}} \sum_{\vx: \norm{\vx}=k} \pr(T(\vzero,\vx) \le t).
\]
Applying Proposition \ref{Tgalwbd} for $\ga \in (2d, 2d+1)$ to bound the summands of the above display and noting that $|\{\vx: \norm{\vx} =k\}| \le ck^{d-1}$,
\begin{align*}
    \pr(\cB_t \not\subseteq \ball{\vzero, t^{\gc_0}}) & \le  \sum_{k: k\ge t^{\gc_0}}  ck^{-2}\exp\left( \frac{\ga(1+\ga)}{\ga-2d} \log(1+t) - (\ga-d-1) \log k\right) \\
                                                    & \le c \exp\left( \frac{\ga(1+\ga)}{\ga-2d} \log\frac{1+t}{t} \right)\sum_{k: k\ge t^{\gc_0}} k^{-2}  \to 0
\end{align*}
as $t \to \infty$. The last inequality follows from the bound of $\gc_0(\ga-d-1)>\frac{\ga(1+\ga)}{\ga-2d}$ in~\eqref{gc_0} and the fact that $\log k \ge \gamma_0\log t$ for all the summands. Thus, \eqref{recursion} holds for $k=1$.

Now suppose $\gc_m$ has been defined and \eqref{recursion} holds for $k=m$. In order to choose $\gc_{m+1} < \gc_m$ so that  \eqref{recursion} holds for $k=m+1$, first we estimate the length of the longest edge used in the LRFPP process by time $t$ under our induction hypothesis.
Observe that on the event $\{\cB_t \subseteq \ball{\vzero, Ct^{\gc_m}}\}$, for any $\gd>0$
\begin{align*}
      & \min\{W_{\la\vx,\vy\ra}: \vx \in \cB_t , \vy \not \in \ball{\vx, t^\gd}\} \text{ stochastically dominates } \\
      & \qquad \min\{W_{\la\vx,\vy\ra}: \vx \in \ball{\vzero, Ct^{\gc_m}} , \vy \not \in \ball{\vx, t^\gd}\},
\end{align*}
which in turn stochastically dominates an exponential random variable with rate
\begin{align*}
\sum_{\vx \in \ball{\vzero, Ct^{\gc_m}}} \sum_{\vy \not\in \ball{\vx, t^{\gd}}} \lone{\vx-\vy}^{-\ga}
    &\le  |\ball{\vzero, Ct^{\gc_m}}| \sum_{k\ge t^\gd} \sum_{\vu: \lone{\vu}=k} \lone{\vu}^{-\ga}\\
    & \le  ct^{\gc_m d} \sum_{k\ge t^\gd} k^{d-1-\ga} \le ct^{\gc_md-\gd(\ga-d)}
\end{align*}
for some constant $c>0$. Therefore, using the inequality $1-e^{-x} \le x$ we have
\begin{align}
      & \pr\left(\left\{\min\{W_{\la\vx,\vy\ra}: \vx \in \cB_t , \vy \not \in \ball{\vx, t^\gd}\} \le t\right\} \cap \{\cB_t \subseteq \ball{\vzero, Ct^{\gc_m}}\} \right) \notag \\
      &\qquad \le 1-\exp\left(-c t^{1+\gc_md-\gd(\ga-d)}\right)
    \le ct^{1+\gc_md-\gd(\ga-d)}. \label{longedgeest}
\end{align}
Now if $\cB_t \subseteq \ball{\vzero, Ct^{\gc_m}}$ and $W_{\la\vx,\vy\ra} > t$ for all $\vx\in\cB_t$ and $\vy \not \in \ball{\vx, t^\gd}$, then all the edges belonging to the optimal path joining $\vzero$ and $\vx \in \cB_t$ must have Euclidean length at most $t^\gd$. Hence for any $\gb>\ga$,  $\vx \in \cB_t$ implies
\begin{align*}
    t \ge T^{(\ga)}(\vzero,\vx)
      & = \inf_{\pi \in \cP_{\vzero,\vx}: \lone{e} \le t^\gd \forall e \in \pi} \sum_{e \in \pi} \lone{e}^\ga \go_e                       \\
      & \ge  \inf_{\pi \in \cP_{\vzero,\vx}: \lone{e} \le t^\gd \forall e \in \pi} \sum_{e \in \pi} \lone{e}^\gb t^{-\gd(\gb-\ga)} \go_e \\
      & \ge  t^{-\gd(\gb-\ga)} \inf_{\pi \in \cP_{\vzero,\vx}} \sum_{e \in \pi} \lone{e}^\gb \go_e
       = t^{-\gd(\gb-\ga)} T^{(\gb)}(\vzero,\vx),
\end{align*}
which in turn implies $\vx \in \cB^{(\gb)}_{t^{1+\gd(\gb-\ga)}}$. This shows that if we let
 \[ \gc_{m+1} := 1+\gd(\gb-\ga),\]
 then $\cB_t \subseteq \ball{\vzero, Ct^{\gc_{m+1}}}$
 on the event
\begin{align} \label{event}
    \{\cB_t \subseteq \ball{\vzero, Ct^{\gc_m}}\}
    & \cap \{W_{\la\vx,\vy\ra} > t\ \forall\ \vx \in \cB_t \text{ and } \vy \not\in \ball{\vx, t^\gd}\}  \\
    & \cap \{\cB^{(\gb)}_{t^{1+\gd(\gb-\ga)}} \subseteq \ball{\vzero, Ct^{1+\gd(\gb-\ga)}}\} \notag
\end{align}
for any $\gb > \ga$.

Now we choose $\gb>2d+1$ and $\gd>(1+\gc_md)/(\ga-d)$ so that
\begin{align} \label{gc_m}
    \gc_{m+1} & := 1+\gd(\gb-\ga)\notag                                                                         \\
              & = 1+\frac{1+\gc_md}{\ga-d}(2d+1-\ga)+\frac{\eps}{2} \frac{(d+1)(\ga-2d)}{\ga-d}.
\end{align}
We also choose $C\ge 1$ such that $\lim_{t \to \infty} \pr(\cB^{(\gb)}_t \subseteq \ball{\vzero, Ct})=1$. Proposition \ref{induction} guaranties the existence of such a $C$. Simple algebraic manipulation confirms that $\{\gc_k\}_{k\ge 0}$, as defined in the last display, satisfies \eqref{gammaprop}.
Combining \eqref{longedgeest}, our induction hypothesis that~\eqref{recursion} holds for $k=m$ and our choices of $\gd$ and $C$, the limit of the probability of the event in \eqref{event} is 1, and hence
\[
    \lim_{t \to \infty} \pr(\cB_t \subseteq \ball{\vzero, Ct^{\gc_{m+1}}}) =1,
\]
which completes the proof of the induction argument.\qed

%%%%%%%%%%%%%%%%%%%%%%%%%%%%%%%%%%%%%%%%%%%
\section{Linear Growth Regime}\label{sec:lg}

In this section, we consider the case of fixed $\ga>2d+1$. Let
$\ball{\vu,r}$ denote the Euclidean $\ell_\infty$--ball of radius $r$
around $\vu\in \dR^d$. In order to establish linear growth for the
LRFPP balls, we will show that for any fixed $\eta \in (0,2)$,
\begin{align*}
    \inf_{{\vz, \vw \in \ball{\vzero, n}, \linf{\vz-\vw} \ge \eta n}} \frac{T(\vz,\vw)}{\linf{\vz-\vw}} \ge c_{\eta}>0
\end{align*}
with high probability, where the constant $c_{\eta}$ does not depend
on $n$. For that we need the following lemma.

\begin{lem} \label{nobigjump}
    If $\ga > 2d+1, \theta \in (0, 1)$ and $c_1, c_2$ are any positive
    constants, then
    \[
        \pr\left(W_{\la\vz\vw\ra} \ge c_1 n \text{ for all } \vz \in \ball{\vzero,  c_2n} \text{ and }  \vw \not\in \ball{\vz, n^\theta}\right) \ge 1- Cn^{d+1-\theta(\ga-d)}
    \]
    for some constant $C>0$.
\end{lem}

\begin{proof}
Note that for any fixed $\vz \in \dZ^d$, the random variable $\min\{W_{\la\vz\vw\ra}\mid \vw \not\in \ball{\vz, n^\theta}\}$ is exponentially distributed with rate
\[
    \sum_{\vu: \linf\vu \ge n^\theta} \lone{\vu}^{-\ga} \le \sum_{k \ge n^\theta} c k^{d-1} \cdot k^{-\ga} \le
    cn^{\theta(d-\ga)},
\]
and hence the random variable
$\min\{W_{\la\vz\vw\ra}\mid \vz\in \ball{\vzero, c_2 n}, \vw \not\in \ball{\vz, n^\theta}\}$
stochastically dominates an exponential distribution with rate
$|\ball{\vzero,c_2 n}| \cdot cn^{\theta(d-\ga)}$. Therefore, using the inequality $1-e^{-x} \le x$ and the fact that $|\ball{\vzero,c_2 n}| \le cn^d$ we have
\begin{align*}
      & \pr\left(\min\{W_{\la\vz\vw\ra}: \vz\in \ball{\vzero, c_2 n}, \vw \not\in \ball{\vz, n^\theta}\} \le c_1 n\right) \\
      &\qquad \le 1-\exp\left(-|\ball{\vzero,c_2 n}| \cdot cn^{\theta(d-\ga)} \cdot c_1 n\right) \le C n^{\theta(d-\ga)+d+1}
\end{align*}
for some constant $C>1$.
\end{proof}

Lemma \ref{nobigjump} ensures that if $\theta<1$ is sufficiently close to 1, then with high probability none of the edges, which have $\ell_\infty$--length more than $O(n^\theta)$ and have at least one end in an $\ell_\infty$--ball of radius $O(n)$ around $\vzero$, will be a part of the optimal paths which start from $\vzero$ and have passage time $O(n)$. This observation plays a crucial role in proving linear growth for the LRFPP balls when $\ga>2d+1$. We now use this key observation to produce a linear lower bound for $T(\vx,\vy)$.

\begin{prop} \label{induction}
    For any $\ga>2d+1$ and $\eta\in(0,2)$, if $\eps >0$ is small enough, there exist constants $c(\eta), C>0$ such that
    \[
        \pr\left(\inf_{\vz, \vw \in \ball{\vzero, n}: \linf{\vz-\vw} \ge \eta n} \frac{T(\vz,\vw)}{\linf{\vz-\vw}} \ge c(\eta)\right)\ge 1-Cn^{-(\ga-2d-1-\eps)}.
    \]
\end{prop}

\begin{proof}
Using an induction argument we will prove that there are constants $\eps, \kappa > 0$ small enough, $\gd, \theta \in (0, 1), c_\eta:=5\eta+3d+6$ and $\ell \in
\dN$ large enough such that if $\ell_m:=\ell^{1/\theta^m}$ for $m\ge 0$
and if
\begin{align} \label{def:sA}
    \sA_k := \left\{\inf_{\vz, \vw \in \ball{\vzero, \ell_k^{1+\kappa}}: \linf{\vz-\vw} \ge \eta \ell_k}
    \frac{T(\vz,\vw)}{\linf{\vz-\vw}} < \ell^{-\gd} \prod_{i=1}^{k-1}\left(1-c_\eta\ell_i^{-\kappa}\right)
    \right\},
\end{align}
then for all $k \ge 0$,
\begin{align} \label{indhyp}
    \pr(\sA_k) \le C(\eta) \ell_k^{-(\ga-2d-1-2\eps)}
\end{align}
for some constant $C$. The choices for all the parameters will be
specified as we proceed through the proof. Once we prove \eqref{indhyp}, the proposition will follow by taking $k$ such that $\ell^{1/\theta^k}=m$ and $c(\eta):=\ell^{-\gd}\prod_{i=1}^\infty(1-c_\eta\ell_i^{-\kappa})$.

To prove \eqref{indhyp} for $k=0$ we use union bound and Proposition~\ref{Tgalwbd} with $T(\vzero,\vx)$ replaced by $T(\vz,\vw)$ and $t$ replaced by $\lone{\vz-\vw} \ell^{-\gd}$ to have
\begin{align*}
    \pr\left(\sA_0\right)
      & \le  c \sum_{\vz, \vw \in \ball{\vzero, \ell^{1+\kappa}}: \linf{\vz-\vw} \ge \eta\ell}
     \left(\frac{(\lone{\vz-\vw}\ell^{-\gd})^{(1+\ga)/(\ga-2d)}}{\lone{\vz-\vw}}\right)^\ga\\
      & \le  c(\eta) (2\ell^{1+\kappa})^{2d}\cdot
    \left(\frac{\ell^{(1+\kappa-\gd)(1+\ga)/(\ga-2d)}}{\ell}\right)^\ga,
\end{align*}
as there are at most $(2\ell^{1+\kappa})^{2d}$ terms in the above sum. A simple arithmetic shows that
the exponent of $\ell$ in the right hand side of the last display is $(2d + 1 + \eps -\ga)$ if we take
\begin{align}\label{gd}
    \gd := 1+\kappa-\frac{(1+\eps-2d\kappa)(\ga-2d)}{\ga(\ga+1)}.
\end{align}
Thus \eqref{indhyp} is established for $k=0$.

Now suppose \eqref{indhyp} holds for $k=m$, we will show that it holds for $k=m+1$ as well.
Fix any $\vx_1, \vy_1 \in \ball{\vzero, \ell_{m+1}^{1+\kappa}}$ such that $\linf{\vx_1 - \vy_1} \ge \eta\ell_{m+1}$
and let the optimal $\sE$-path joining $\vx_1$ and $\vy_1$ be $\pi \in \sP_{\vx_1,\vy_1}$. We will bound $W_\pi=T(\vx_1,\vy_1)$ from below. It is easy to see that if
\begin{align*}
    F_{m+1} & := \left\{\la\vz\vw\ra: \vz \in \ball{\vzero, 4\ell_{m+1}^{1+\kappa}}, \vw \not\in \ball{\vz,\ell_m}\right\},\\
    \text{and} \quad
    H_{m+1} & := \left\{\min_{e \in F_{m+1}} W_e \ge 2\ell_{m+1}^{1+\kappa}\right\},
\end{align*}
then $\pi \cap F_{m+1}=\emptyset$ on the event $H_{m+1}$, as $\linf{\vx-\vy} \le
2\ell_{m+1}^{1+\kappa}$.

For the remainder of the argument we will assume that  $H_{m+1}$ occurs. In that case, all the edges that belong to $\pi$ and have at least one of their endpoints in $\ball{\vzero, 4\ell_{m+1}^{1+\kappa}}$ must have $\ell_\infty$--length smaller than $\ell_m$. Our plan is to divide the $\ell_\infty$--ball $\ball{\vzero, 4\ell_{m+1}^{1+\kappa}}$ into smaller disjoint $\ell_\infty$--balls having radius $\ell_m^{1+\kappa}$ and study the contributions of the segments of $\pi$ restricted to those smaller balls to $W_\pi$.
In order to do so, let
\begin{align*}
    O_m & :=  \ell_m^{1+\kappa} \cdot \left\{-4(\ell_{m+1}/\ell_m)^{1+\kappa}+2i-1: 1\le i \le  4(\ell_{m+1}/\ell_m)^{1+\kappa}\right\}, \\
    E_m & :=  \ell_m^{1+\kappa} \cdot \left\{-4(\ell_{m+1}/\ell_m)^{1+\kappa}+2i: 0 \le i \le 4(\ell_{m+1}/\ell_m)^{1+\kappa}\right\},
\end{align*}
and based on these we define the index sets
\begin{align*}
    I_m &:= (O_m)^d, \quad \tI^0_m := (E_m)^d,\\
    \tI^k_m &:= (O_m)^{k-1}\times E_m \times (O_m)^{d-k}, 1\le k\le d.
\end{align*}
\begin{figure}[htbp] %  figure placement: here, top, bottom, or page
   \centering
   \includegraphics[width=2.5in]{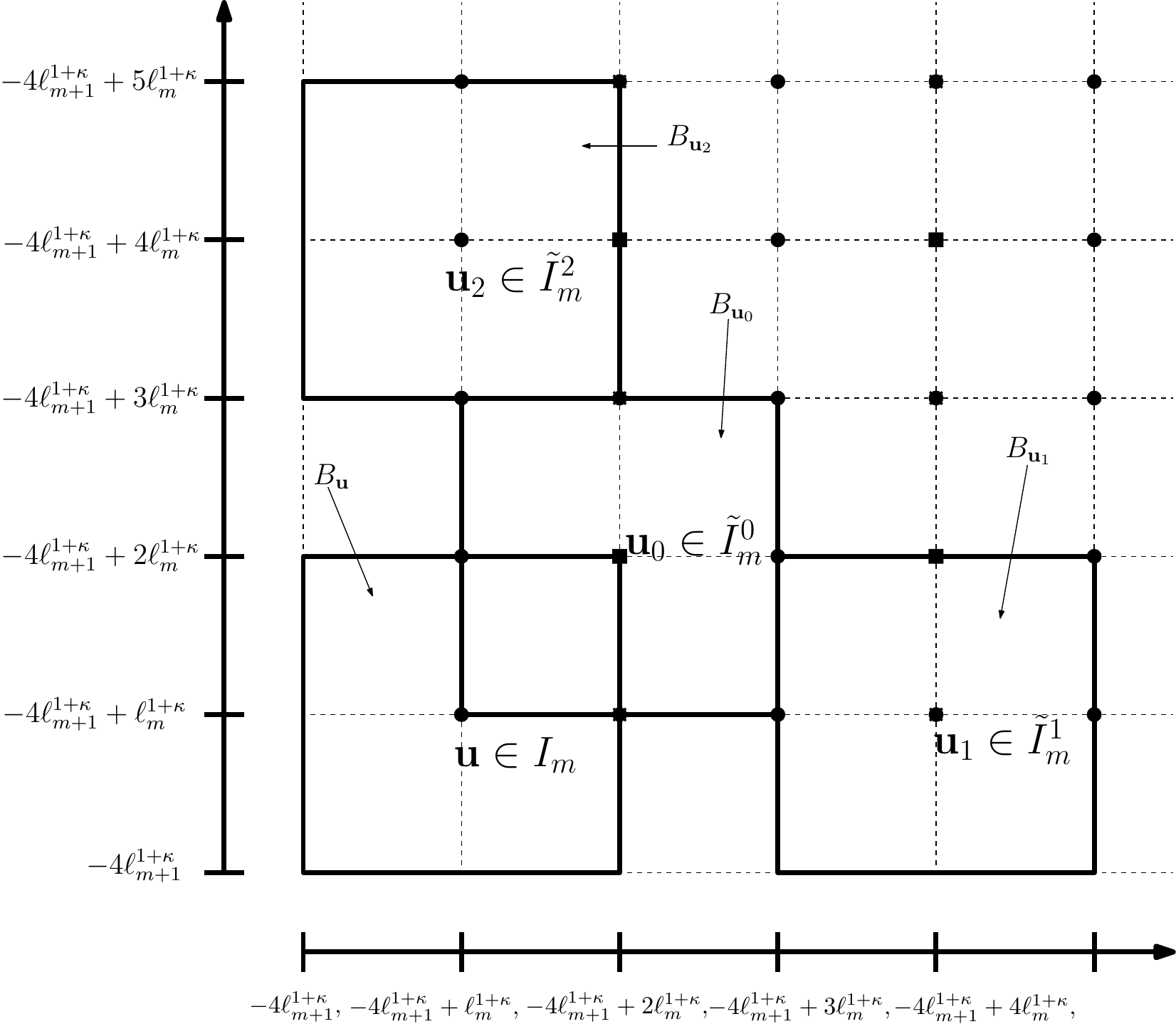}
   \caption{Pictorial description of the balls in $\sB, \sB^k, 0\le k \le 2,$ in case of $d=2$.
    Here the balls $B_\vu=\ball{\vu, \ell_m^{1+\kappa}} \in \sB$ and $B_{\vu_k}=\ball{\vu_k, \ell_m^{1+\kappa}} \in \sB^k, k=0,1,2$.
    Boundary of any ball is in the interior of another ball.}
   \label{fig:block}
\end{figure}
We also define the corresponding collections of disjoint balls
$\sB:=\{B_\vu=\ball{\vu, \ell_m^{1+\kappa}}: \vu \in I_m\}$ and
$\tilde\sB^k:=\{\tB^k_\vv=\ball{\vv, \ell_m^{1+\kappa}}: \vv \in
\tI^k_m\}$ for $0\le k\le d$ (See Figure~\ref{fig:block} for a
pictorial description of one element from each of the collections
$\sB, \sB^k, 0\le k\le 2$ in case of $d=2$). It is easy to see that
each of the collections of disjoint balls $\sB, \tilde\sB^k, 0\le
k\le d,$ covers $\ball{\vzero, 4\ell_{m+1}^{1+\kappa}}$. But the
reason behind considering more than one such collection is to make
sure that the path $\pi$ spends enough time going through the bulk
of one ball or the other, rather than staying close to their
boundaries. By the choice of the collection of smaller balls, any
segment of $\pi$ which stays close to the boundary of the balls in
$\sB$ must pass through one of the balls in $\tilde \sB^k$ for some
$0\le k\le d$.

If $\vu, \hat\vu \in I_m$ and $\linf{\vu-\hat\vu} =2\ell_m^{1+\kappa}$, we say that $B_\vu$ and $B_{\hat\vu}$ are neighboring balls. Similarly, $\tB^k_\vv$ and $\tB^k_{\hat\vv}$ will
be called neighboring balls if $\vv, \hat\vv \in \tI^k_m$ and $\linf{\vv-\hat\vv} = 2\ell_m^{1+\kappa}$. Assuming $H_{m+1}$ occurs, each of the edges of $\pi$, which have at least one endpoint in $\ball{\vzero,4\ell_{m+1}^{1+\kappa}}$, can either stay within one ball
$B_\vu$ (resp.~$\tB^k_\vv$) or go from a ball $B_\vu$ (resp.~$\tB^k_\vv$) to one of its neighboring balls $B_{\hat\vu}$ (resp.~$\tB^k_{\hat\vv}$). Now there are two possibilities for $\pi$; either it goes out of the ball $\ball{\vzero,4\ell_{m+1}^{1+\kappa}}$ at some point, or the entire path remains inside $\ball{\vzero, 4\ell_{m+1}^{1+\kappa}}$.
In the first case, in view of the last observation, if $\la\vz\vw\ra \in \pi$ is the first edge while traversing along $\pi$ from $\vx_1$ to $\vy_1$ such that $\vz \in \ball{\vzero,4\ell_{m+1}^{1+\kappa}}$ and
$\vw \not \in \ball{\vzero,4\ell_{m+1}^{1+\kappa}}$, then $\linf{\vz-\vw} \le \ell_m$ and hence
\[
    \linf{\vx_1-\vz} \ge \linf{\vx_1-\vw} - \linf{\vz-\vw} \ge
    (3\ell_{m+1}^{1+\kappa}-\ell_m) \ge \linf{\vx_1-\vy_1},
\]
as $\linf{\vx_1-\vy_1} \le 2\ell_{m+1}^{1+\kappa}$. In the second case, we obviously have
$\linf{\vx_1-\vy_1} \ge \eta\ell_{m+1}$.
So, in both cases $\pi$ must have a segment $\tilde{\pi}=\la\vx\cdots \vy\ra$ which stays within $\ball{\vzero, 4 \ell_{m+1}^{1+\kappa}}$ and satisfies $\linf{\tpi} \ge \linf{\pi} \ge \eta\ell_{m+1}$.

For a ball $B \in \sB \cup \cup_{k=0}^d\tilde\sB^k$ and $\vz, \vw \in B$ let
$T_B(\vz,\vw)$ denotes the minimum passage time over all paths which join $\vz$ and $\vw$ and stay within $B$.
We say that such a ball $B$ is good (resp.~bad) if
\begin{align}\label{badball}
    \inf_{\vz, \vw \in B: \linf{\vz-\vw} \ge \eta \ell_m}
    \frac{T_B(\vz,\vw)}{\linf{\vz-\vw}} \ge \text{ (resp.~$<$) } \ell^{-\gd} \prod_{k=1}^{m-1}\left(1 - c_\eta
    \ell_k^{-\kappa}\right).
\end{align}
Recalling the definition of $\linf{\pi}$ for an $\sE$-path $\pi$, we will see that

\begin{lem} \label{piprop}
    If the events $H_{m+1}$ and
    \begin{align*}
        L_{m+1} := &\{\text{numbers of bad balls in $\sB$ and } \tilde\sB^k, 0\le k\le d,\\
        &\ \ \text{are at most } \linf{\vx-\vy} \ell_m^{-(1+2\kappa)}\}
    \end{align*}
    occurs, then the path $\tpi=\la\vx \cdots \vy\ra$ obtained as above contains disjoint segments $\{\tpi_i=\la\vz_i \cdots \vw_i\ra\}_{i\ge 1}$ such that
    \begin{enumerate}[$\bullet$]
        \item each $\tpi_i$ stays within some $B_i \in \sB \cup \cup_{k=0}^d \tilde\sB^k$ such that $B_i$ is good according to \eqref{badball}.
        \item  $\linf{\tpi_i} \ge \eta\ell_m$ for all $i$
        \item $\sum_{i\ge 1} \linf{\tpi_i} \ge \linf{\tpi} \left(1-c_\eta\ell_m^{-\kappa}\right)$ for $c_\eta > 0$ defined at the beginning of the proof of Proposition~\ref{induction}.
    \end{enumerate}
\end{lem}

We postpone the proof of Lemma~\ref{piprop}, first we will see that this lemma provides a lower bound for $W_\pi$ by bounding $\sum_{j\ge 1} T_{B_j}(\vz_j,\vw_j)$, which will enable us to conclude
\begin{align}\label{AHL}
    \sA_{m+1}^c \supseteq H_{m+1} \cap L_{m+1}.
\end{align}
By the properties of the path segments $\{\tpi_i=\la\vz_i \cdots
\vw_i\ra\}$ in Lemma \ref{piprop} and the definition of good balls
in \eqref{badball}, it is easy to see that
\begin{align*}
    W_\pi \ge  W_{\tpi} \ge \sum_{j\ge 1} T_{B_j}(\vz_j,\vw_j)
    &\ge \ell^{-\gd} \prod_{k=1}^{m-1}(1-c_\eta\ell_k^{-\kappa})\sum_{j\ge 1} \linf{\vz_j-\vw_j} \\
    & \ge  \ell^{-\gd} \prod_{k=1}^{m}(1-c_\eta\ell_k^{-\kappa}) \linf{\vx-\vy}
\end{align*}
on the event $H_{m+1} \cap L_{m+1}$. Since $\linf{\vx-\vy} \ge \linf{\vx_1-\vy_1}$ and $\vx_1, \vy_1 \in \ball{\vzero, \ell_{m+1}^{1+\kappa}}$ were arbitrary vertices
satisfying $\linf{\vx_1-\vy_1} \ge \eta\ell_{m+1}$, the above inequality justifies \eqref{AHL}.
Thus, in order to complete the induction argument it remains to
estimate $\pr(H_{m+1})$ and $\pr(L_{m+1})$. On one hand, in view of Lemma \ref{nobigjump},
\begin{align}\label{H_mbd}
    \pr(H_{m+1}) \ge 1 - C
    \ell_{m+1}^{(1+\kappa)(d+1)-\theta(\ga-d)}.
\end{align}
On the other hand, our induction hypothesis \eqref{indhyp} for $k=m$ suggests that $B \in \sB$ is bad with probability $\le C\ell_m^{-(\ga-2d-1-2\eps)}$, as $T_B \ge T$ for every argument. Also, it is easy to see that if $\{B_i\}$ is a collection of pairwise disjoint balls, then the events $\{B_i \text{ is good}\}$ are independent. Since $|I_m| = 4^d \ell_m^{(1/\theta-1)(1+\kappa)d}$, the expected number of bad balls in $\sB$ is
\[
    \le 4^d \ell_m^{d(1+\kappa)(1/\theta-1)} \cdot C\ell_m^{-(\ga-2d-1-2\eps)}.
\]
If we choose $\eps, \theta>0$ such that
\begin{align}\label{epstheta}
    \eps \left[2 + \frac{d}{\ga-d-\eps}\right] \le \ga -2d-1 \text{ and } \theta = 1 - \frac{\eps}{\ga-d},
\end{align}
for small enough $\kappa>0$  the exponent of $\ell_m$ in the upper bound for the expected number of bad balls among $\{B_\vu: \vu \in I_m\}$ is
\[
    d(1+\kappa)(1/\theta-1)-(\ga-2d-1-2\eps) < 1/\theta-1 - 2\kappa.
\]
So, if we take $N=4^d(\ell_{m+1}/\ell_m)^{d(1+\kappa)}$ and $p = \eta\ell_m^{1/\theta - 1 - 2\kappa}/3N$, the number of bad balls in $\sB$ is stochastically dominated by the Binomial$(N, p)$ distribution. Now using standard large deviation argument for Binomial distribution and recalling that $\linf{\vx-\vy} \ge \eta \ell_{m+1}$, we have
\begin{align*}
    \pr\bigl(\sum_{u \in I_m} \mathbf 1_{\{\text{$B_\vu$ is bad}\}} > \linf{\vx-\vy} \ell_m^{-(1+2\kappa)} \bigr)
      & \le \pr(\text{Binomial}(N,p) > 3Np)                                \\
      & \le \exp(-Np(3\log 3-2))                                           \\
      & = \exp\left(-\eta(\log 3-2/3)\ell_m^{(1/\theta-1-2\kappa)} \right).
\end{align*}
The same estimate holds for the number of bad balls among
$\{\tB^k_\vv: \vv \in \tI^k_m\}$ for each $0\le k\le d$. Thus by union bound,
\begin{align}\label{good ball bd2}
    \pr(L_{m+1}^c) \le 3\exp\left(-\eta(\log 3-2/3)\ell_m^{(1/\theta-1-2\kappa)} \right).
\end{align}
Combining \eqref{AHL}, \eqref{H_mbd} and \eqref{good ball bd2} if
$\kappa>0$ is small enough, then
\[
    \pr(\sA_{m+1}) \le C\ell_{m+1}^{(1+\kappa)(d+1)-\theta(\ga-d)}
    \le C \ell_{m+1}^{-(\ga-2d-1-2\eps)}
\]
by the choice of $\theta$.  This proves \eqref{indhyp} for $k=m+1$.
\end{proof}

In order to complete the proof of Proposition \ref{induction}, it remains to prove Lemma \ref{piprop} which is now presented.

\begin{proof}[Proof of Lemma \ref{piprop}.]
Recall that $\tpi=\la\vx \cdots \vy\ra \subseteq \ball{\vzero, 4\ell_{m+1}^{1+\kappa}}$ and $\linf{\tpi} \ge \eta\ell_{m+1}$. Let $\{B_{\hat\vu_i}\}_{i\ge 0}, \hat\vu_i \in I_m,$ be a sequence of balls such that $B_{\hat\vu_i}$ is a neighbor of $B_{\hat\vu_{i-1}}$
for $i\ge 1$, $\vx \in B_{\hat\vu_0}$ and for $i\ge 1$ the path $\tpi$ enters $B_{\hat\vu_i}$ after exiting from
$B_{\hat\vu_{i-1}}$. We can think of $\{B_{\hat\vu_i}\}_{i\ge 0}$
as a trajectory of a nearest-neighbor random walk which takes its values in $\sB$.
If a ball $B_\vu$ is bad according to \eqref{badball} and if it appears more than once in the sequence $\{B_{\hat\vu_j}\}$, then we remove the loop created by the associated random walk at $B_\vu$. So a bad ball can appear at most once in this sequence. Abusing notations, we write $\{B_{\hat\vu_i}\}$ for the loop-erased sequence.
Suppose $\tpi$ enters $B_{\hat\vu_i}$ through $\hat\vx_i$ and leaves it through $\hat\vy_i$. So, $\la\hat\vx_i \cdots \hat\vy_i\ra$ is a segment of $\tpi$ staying within $B_{\hat\vu_i}$ whenever $B_{\hat\vu_i}$ is good, otherwise $\la\hat\vx_i \cdots \hat\vy_i\ra$ may not be a segment of $\tpi$.
Also let $\{\hat\vu_{i_j}\}_{j=1}^J$ be the subsequence of $\{\hat\vu_i\}_{i\ge 0}$ such that for each $j$ the portion of $\tpi$ within $B_{\hat\vu_{i_j}}$ does not stay within $\ball{\hat\vx_{i_j}, \eta\ell_m}$.
To simplify notation, we will write $(\bar\vu_j, \bar\vx_j, \bar\vy_j)$ instead of $(\hat\vu_{i_j}, \hat\vx_{i_j}, \hat\vy_{i_j})$, and $\vy_0\equiv \vx$ and $\vx_{J+1}\equiv\vy$.

We say that the segment $\la\bar\vx_j \cdots \bar\vy_j\ra$ is {\it admissible} (resp.~{\it inadmissible}) if $B_{\bar\vu_j}$ is good (resp.~bad).
In the same spirit, the segment $\la\bar\vy_j \cdots \bar\vx_ {j+1}\ra$
will be called {\it inadmissible} if it contains a segment of the form $\la\bar\vx_k \cdots \bar\vy_k\ra$ which is not admissible, otherwise $\la\bar\vy_j \cdots \bar\vx_{j+1}\ra$ will be called {\it admissible}.

An admissible segment $\la\bar\vx_j \cdots \bar\vy_j\ra$ will be called {\it short} if
the entire segment resides within $\ball{\bar\vx_j, \frac 14
    \ell_m^{1+\kappa}}$, otherwise we say that $\la\bar\vx_j \cdots
\bar\vy_j\ra$ is {\it long}. In the same spirit, an admissible segment
$\la\bar\vy_j \cdots \bar\vx_{j+1}\ra$ will be called {\it short} if $\linf{\bar\vy_j-\bar\vx_{j+1}} < 2\eta\ell_m$, otherwise we call it {\it
    long}. Instead of the entire sequences $\{\bar\vx_j\}_{j=1}^{J+1}$ and $\{\bar\vy_j\}_{j=0}^J$, we need to consider the following subsequences $\{\bar\vx_{j_l}\}_{l=1}^{L+1}$ and $\{\bar\vy_{j_l}\}_{l=0}^L$.
We define $\{j_l\}$ along with disjoint sets $\cL_i \subseteq \dN, -1\le i\le 5,$ inductively starting with $j_0 := 0$ and $\cL_i=\emptyset$. Having defined $j_l$,
\begin{enumeraten}
    \item  scan through the segments $\la\bar\vy_{j_l} \cdots \bar\vx_{j_l+1}\ra$, $\la\bar\vx_{j_l+1} \cdots \bar\vy_{j_l+1}\ra$, $\la\bar\vy_{j_l+1} \cdots \bar\vx_{j_l+2}\ra, \ldots$ sequentially,
    \item if $\la\bar\vx_{j_l+k} \cdots \bar\vy_{j_l+k}\ra$ is
    inadmissible for some $k\ge 1$ and all previous segments are
    short, then we let $j_{l+1}:=j_l+k$. If $\linf{\bar\vy_{j_l} - \bar\vx_{j_{l+1}}} \ge \eta\ell_m$,
    then we include $l$ in $\cL_{-1}$, otherwise we include $l$ in $\cL_0$. Then we go back to step (1) with $l$ replaced by $l+1$.
    \item if $\la\bar\vy_{j_l+k-1} \cdots \bar\vx_{j_l+k}\ra$ is the
    first long segment for some $k\ge 1$, then we let
    $j_{l+1} := j_l+k$ and $j_{\tilde  l} := j_l+k-1$. In addition,
    \begin{enumeratea}
        \item if $\linf{\bar\vy_{j_l}-\bar\vy_{j_{\tilde  l}}} \ge \eta \ell_m$, then we include $l$ in $\cL_1$; otherwise we include $l$ in $\cL_2$.
        \item if $j_{l+1}=J+1$, then we let $L=l$ and stop, otherwise $\bar\vy_{j_{l+1}}$ is defined and we go back to step (1) with $l$ replaced by $l+1$.
    \end{enumeratea}
    \item if $\la\bar\vx_{j_l+k} \cdots \bar\vy_{j_l+k}\ra$ is the
    first long segment for some $k\ge 1$, then we let
    $j_{l+1}:=j_l+k$. Here also we include $l$ in $L_3$ or $L_4$ depending on whether $\linf{\bar\vy_{j_l}-\bar\vx_{j_{l+1}}} \ge \eta\ell_m$ or not, and we go back to step (1) with $l$ replaced by $l+1$.
    \item if for some $k \ge 1$ all the segments $\la\bar\vy_{j_l+k'-1} \cdots \bar\vx_{j_l+k'}\ra, 1\le k' \le k,$ and $\la\bar\vx_{j_l+k'} \cdots \bar\vy_{j_l+k'}\ra$, $1\le k' < k,$ are short and $j_l+k=J+1$, then we let $L=l$ and $j_{L+1}=j_l+k$.
    If $\linf{\bar\vy_{j_L}-\bar\vx_{j_{L+1}}} \ge \eta\ell_m$, then we include $L$ in $\cL_2$, otherwise we let $\cL_5:=\{L\}$.
\end{enumeraten}
To simplify notations, we write $(\vx_l, \vy_l, \vu_l)$ instead of
$(\bar\vx_{j_l}, \bar\vy_{j_l}, \bar\vu_{j_l})$ and $\tilde l$ instead of $j_{\tilde  l}$.

Having defined  $\{\vx_l\}_{l=1}^{L+1}$ and $\{\vy_l\}_{l=0}^L$, note that if $l \in \cL_4$,
then $\la\vx_{l+1} \cdots \vy_{l+1}\ra$ is a segment of $\tpi$ and it has a subsegment of the form
$\la\vx_{l+1} \cdots \tilde\vx_{l+1}\ra$ such that $\linf{\vx_{l+1} - \tilde\vx_{l+1}} \ge \frac 14
\ell_m^{1+\kappa}$. So if $|\cL_4| \ge \linf{\vx - \vy}/\frac 14\ell_m^{1+\kappa}$, then
\[
    \sum_{l\in\cL_4} \linf{\la\vx_{l+1} \cdots \tilde\vx_{l+1}\ra}
    \ge  \linf{\vx - \vy},
\]
and hence the segments $\{\la\vx_{l+1} \cdots \tilde\vx_{l+1}\ra: l\in \cL_4\}$ fulfil the criteria of Lemma \ref{piprop}. Otherwise if $|\cL_4| < \linf{\vx - \vy}/\frac 14 \ell_m^{1+\kappa}$, then using triangle inequality and noting that $\vy_{0}=\vx, \vx_{L+1}=\vy$,
\begin{align*}
    \sum_{l \not\in \cL_{-1}\cup\cL_0}  \linf{\vx_{l} - \vy_{l}}
                                       & \ge  \linf{\vx-\vy} - \sum_{i=-1}^5 \sum_{l \in \cL_i} \linf{\vy_{l}-\vx_{l+1}} - \sum_{l \in \cL_{-1}\cup\cL_0} \linf{\vx_{l} - \vy_{l}}  \\
                                       & \ge  \linf{\vx-\vy} - \sum_{l\in \cL_1} \left(\linf{\vy_{l} - \vy_{\tilde l}} + \linf{\vy_{\tilde l} - \vx_{l+1}}\right)                   \\
                                       & \qquad  - \eta\ell_m (1+|\cL_0|+4\linf{\vx - \vy}/\ell_m^{1+\kappa}) \\
                                       & \qquad - \sum_{l\in \cL_{-1} \cup \cL_2 \cup \cL_3} \linf{\vy_{l} - \vx_{l+1}} - \sum_{l \in \cL_{-1}\cup\cL_0} \linf{\vx_{l} - \vy_{l}}.
\end{align*}
It is easy to see that $|\cL_{-1}|+|\cL_0| \le \linf{\vx-\vy} \ell_m^{-(1+2\kappa)}$ on the event $L_{m+1}$, and hence
\[
    \sum_{l \in \cL_{-1}\cup\cL_0} \linf{\vx_{l} - \vy_{l}} \le \linf{\vx-\vy} \ell_m^{-(1+2\kappa)}\cdot 2\ell_m^{1+\kappa}.
\]
Also recall that $\linf{\vx-\vy} \ge \eta\ell_{m+1} \ge \eta\ell_m^{1+\kappa}$ if $\kappa>0$ is small enough.
So if we define
\[
    \Pi := \{\la\vy_{l} \cdots \vx_{l+1}\ra: l \in \cL_{-1} \cup \cL_2 \cup \cL_3\}\cup \{\la\vy_{l} \cdots \vy_{\tilde l}\ra, \la\vy_{\tilde l} \cdots \vx_{l+1}\ra: l \in \cL_1\},
\]
then the inequality in the previous display reduces to
\begin{align}\label{trian1}
    \sum_{l \not\in \cL_{-1}\cup\cL_0} \linf{\vx_{l} - \vy_{l}}
    \ge \linf{\vx-\vy} [1-(3+5\eta)\ell_m^{-\kappa}] - \sum_{\hat\pi
    \in \Pi} \linf{\hat\pi}.
\end{align}
Now we focus on the segments in $\Pi$. We claim that each segment $\hat\pi \in \Pi$
\begin{enumeratea}
    \item satisfies $\linf{\hat\pi} \ge \eta\ell_m$, and
    \item (assuming that the event $H_{m+1}$ occurs) stays within $\ell_\infty$-distance $\frac 12\ell_m^{1+\kappa}$ from the boundary of some of the balls in $\sB$.
\end{enumeratea}
To see that (a) holds note that by the definition of $\cL_2$,
\[
    \linf{\vy_{l} - \vx_{l+1}} \ge \linf{\vy_{\tilde l} - \vx_{l+1}} - \linf{\vy_{l} - \vy_{\tilde l}} \ge 2\eta\ell_m - \eta\ell_m = \eta\ell_m.
\]
The facts that $\linf{\vy_{l} - \vx_{l+1}} \ge \eta\ell_m$ for $l \in \cL_3 \cup \cL_{-1}$ and $\linf{\vy_{l} - \vy_{\tilde l}}$, $\linf{\vy_{\tilde l} - \vx_{l+1}}$  $\ge \eta\ell_m$ for $l \in \cL_1$ follows trivially from the definition of $\cL_i$. To see that (b) holds observe that if $H_{m+1}$ occurs, then by the definition of $\{\vx_j\}$ and $\{\vy_j\}$ each $\vx_j$ and $\vy_j$ stays within $\ell_\infty$-distance $\ell_m$ from the boundary of some ball $B_\vu$, and by the definition of $\Pi$ any segment of the form $\la\vx_j \cdots \vy_j\ra$, which is a part of $\hat\pi \in \Pi$, must lie within $\ball{\vx_j, \frac 14 \ell_m^{1+\kappa}}$.

It is easy to see that by properties (a) and (b) of the segments in $\Pi$, any $\hat\pi \in \Pi$ should consist of path segments $\{\hat\pi_{i}\}_{i \ge 1}$ such that each $\hat\pi_i$ satisfies $\linf{\hat\pi_i} \ge \eta\ell_m$ and stays within one of the balls $\tB^k_\vv$ for $\vv \in \tI^k_m$ and $0\le k\le d$. But we need to discard those segments which belong to bad balls. In order to do so, for each $\hat\pi \in \Pi$ we determine the associated loop-erased sequence of balls $\{B_j\}$, as we did in the beginning of the proof. Then, segregating the portion of $\hat\pi$ within the bad balls among $\{B_j\}$ $\hat\pi$ can be written as $\hat\pi \equiv \hat\pi^I_1 \hat\pi^A_1\hat\pi^I_2 \ldots$ such that $\{\hat\pi^I_i\}_{i\ge 1}$ are inadmissible segments, whereas $\{\hat\pi^A_i\}_{i\ge 1}$ are admissible ones. Separating the segments $\{\hat\pi^A_i: \linf{\hat\pi^A_i} \ge \eta\ell_m\}$ from the rest and using triangle inequality,
\begin{align*}
      & \sum_{\hat\pi \in \Pi} \sum_{i\ge 1: \linf{\hat\pi^A_i} \ge \eta\ell_m} \linf{\hat\pi^A_i}                                                                         \\
      & \qquad \ge \sum_{\hat\pi \in \Pi} \linf{\hat\pi} - |\{\tB^k_\vv:  \vv\in\tI^k_m, 0\le k\le d, \tB^k_\vv \text{ is bad }\}|\cdot (2\ell_m^{1+\kappa}+2\eta\ell_m).
\end{align*}
So, on the event $L_{m+1}$ the above inequality reduces to
\begin{align*}
    & \sum_{\hat\pi \in \Pi} \sum_{i\ge 1: \linf{\hat\pi^A_i} \ge \eta\ell_m} \linf{\hat\pi^A_i}
    \ge \sum_{\hat\pi \in \Pi} \linf{\hat\pi} - (d+1)\linf{\vx-\vy}\ell_m^{-(1+2\kappa)}\cdot 3\ell_M^{1+\kappa}.
\end{align*}
Combining this inequality with \eqref{trian1},
\begin{align}\label{trian2}
    \sum_{l\not\in\cL_{-1}\cup\cL_0} \linf{\vx_{l} - \vy_{l}} + \sum_{\hat\pi \in \Pi} \sum_{i\ge 1: \linf{\hat\pi^A_i} \ge \eta\ell_m} \linf{\hat\pi^A_i}
    \ge \linf{\vx-\vy}(1-c_\eta\ell_m^{-\kappa}).
\end{align}
Now by the definition of $\{\vx_j\}$ and $\{\vy_j\}$ it is clear that for
$l\not\in\cL_{-1}\cup\cL_0$ either $\linf{\vx_l-\vy_l} \ge \eta\ell_m$ or $\la\vx_l \cdots \vy_l\ra$
has a subsegment of the form $\la\vx_l \cdots \tilde\vx_l\ra$ such that $\linf{\vx_l-\tilde\vx_l} \ge \eta\ell_m$.
So if we define the subsegments $\{\tpi_l\}_{l\not\in\cL_{-1}\cup\cL_0}$ of $\tpi$ by
\[
    \tpi_l :=
    \begin{cases}
        \la\vx_{l} \cdots \tilde\vx_{l}\ra & \text{ if } \linf{\vx_{l}-\vy_{l}} \le \eta\ell_m  \\
        \la\vx_{l} \cdots \vy_{l}\ra       & \text{ if } \linf{\vx_{l}-\vy_{l}}     > \eta\ell_m
    \end{cases},
\]
then clearly $\linf{\tpi_l} \ge \eta\ell_m$ and $\linf{\tpi_l} \ge \linf{\vx_{l}-\vy_{l}}$.
Combining this with \eqref{trian2} we see that
\[
    \{\hat\pi^A_i: i\ge 1, \linf{\hat\pi^A_i} \ge \eta\ell_m, \hat\pi \in \Pi\} \cup \{\tpi_l: l\not\in\cL_{-1} \cup \cL_0\}
\]
fulfill the requirement of this Lemma.
\end{proof}

Before proceeding further let us mention an immediate corollary of Proposition~\ref{induction}.
\begin{cor} \label{boxbound}
    For $\ga > 2d+1$  and $\vx \in \dZ^d$ such that $\linf{\vx}=n$, then with probability $1-o(1)$ the optimal path joining $\vzero$ and $\vx$ stays within $\ball{\vzero, Cn}$ for some large constant $C$.
\end{cor}
\begin{proof}
Using Lemma \ref{compare FPP}\eqref{it:sac} below $\pr(T(\vzero,\vx) > c\linf{\vx}) =o(1)$.
If $T(\vzero,\vx) \le cn$, Lemma~\ref{nobigjump} suggests that there exists $\theta \in (0,1)$ such that for any constant $C<\infty$ the optimal path does not contain an edge having length more than $n^\theta$ and one end in $\ball{\vzero,Cn}$. So if the optimal path goes out of $\ball{\vzero, Cn}$ through $\vy \in \ball{\vzero, Cn}$ for the first time as we traverse along the path starting from $\vzero$, then with probability $1-o(1)$, $\linf{\vy} \ge Cn-n^\theta$ and $T(\vzero,\vy) \le cn$, which event again has probability $1-o(1)$ if we choose $C$ large enough.
\end{proof}

Proposition~\ref{induction} ensures that if $\ga>2d+1$, then with high probability the first-passage metric $T(\vx,\vy)$ grows at least linearly in $\lone{\vx-\vy}$. For the other direction we have the following Lemma.

\begin{lem} \label{compare FPP}
    Let $\vx\in \dZ^d$ and $\ga>0$. Then
    \begin{enumeratea}
        \item\label{it:saa} $T(\vzero, \vx)$ is stochastically dominated by $\sum_{i=1}^{\lone{\vx}} E_i$, where $\{E_i\}$'s are \iid and the common  distribution is exponential with mean one.
        \item\label{it:sab} $\E(T(\vzero, \vx)) \le \lone{\vx}$.
        \item\label{it:sac} for any $\gl> 1$,
        $\pr(T(\vzero, \vx) \ge \gl \lone{\vx})\le \exp(-(\gl\log \gl - \gl +1) \lone{\vx})$.
    \end{enumeratea}
\end{lem}

\begin{proof}
Note that (b) follows from (a) trivially. (c) follows from (a) by using standard large deviation argument for exponential distribution.

To see that (a) holds note that, if $\pi \in \cP_{\vzero,\vx}$ consists of $\lone{\vx}$ many nearest edges, then for any $\ga$, $W_\pi \eqd \sum_{i=1}^{\lone{\vx}} E_i$ and $T(\vzero,\vx)$ is stochastically dominated by $W_\pi$.
\end{proof}

Combining Proposition \ref{induction} and Lemma \ref{compare FPP} with Liggett's subadditive ergodic theorem \cite{L85} and a standard `Denouement' argument described in~\cite[page 17]{D88}  we get the ``shape result" for LRFPP with $\ga>2d+1$.

\begin{proof}[Proof of Theorem \ref{thm:lg}]
For $\vx \in \dZ^d$ and $m, n\in \dZ$ such that $m<n$, let $X_{m,n}:=T(m\vx, n\vx)$. Then, from the definition of $T(\cdot,\cdot)$ it is straight forward to check that
\begin{enumeratei}
\item $X_{0,n} \le X_{0,m} + X_{m,n} \text{ whenever } 0 < m < n $
\item the joint distribution of $\{X_{m,m+k}, k\ge 1\}$ does not depend on $m$
\item for each $k \ge 1, \{X_{nk, (n+1)k}, n \in \dZ\}$  is a stationary process.
\end{enumeratei}
Also, using Lemma~\ref{compare FPP}\eqref{it:sab}, $\E(X_{0,n}) \le n\lone{\vx}$. So, applying Liggett's subadditive ergodic theorem (see~\cite[Theorem 1.10]{L85})
\begin{align}
      & \text{if } \mu(\vx) := \inf_n \frac{1}{n} \E(T(\vzero, n\vx)), \text{ then } \lim_{n \to \infty} \frac{1}{n}T(\vzero, n\vx) = \mu(\vx) \text{ a.s.~provided} \label{dirlim} \\
      & \text{for each } k \ge 1, \{Y^k_n := X_{nk, (n+1)k}\}_{n \in \dZ} \text{ is an ergodic process.} \label{erg}
\end{align}
We postpone the argument for \eqref{erg} towards the end of the proof, now we will see the consequence of \eqref{dirlim}. First note that if $c>0$ is chosen small enough, then by Proposition~\ref{induction}, we have $\pr(T(\vzero, n\vx) \ge cn) \to 1$ for any $\vx \in \dZ^d \setminus \{\vzero\}$ as $n \to \infty$, so
\[
    \E(T(\vzero, n\vx))\ge cn \pr(T(\vzero, n\vx) \ge cn) \ge c/2
\]
for all $n$ large enough, which ensures $\mu(\vx) >0$ for any $\vx \in \dZ^d \setminus \{\vzero\}$.

We extend the definition of $\mu(\cdot)$ to whole of $\dR^d$ using standard procedure, which we mention here for the sake of completeness. For $\vy \in \dR^d$, let
\[
    T(\vzero, \vy) = \min_{\vx \in \dZ^d: \linf{\vy-\vx} \le 1/2} T(\vzero,\vx).
\]
In view of \eqref{dirlim}, if $\vy \in \dQ^d$, then for any $m$ such that $m\vy \in \dZ^d$
\[
    \mu(\vy) := \lim_{n \to \infty} \frac 1n T(\vzero,n\vy) = \frac{1}{m}
    \mu(m\vy) \text{a.s.}
\]
Finally, using subadditivity and Lemma~\ref{compare FPP}\eqref{it:saa} it is easy to see that if for any $\vx, \vy$, $|T(\vzero, \vx) - T(\vzero,\vy)|$ is stochastically dominated by a sum of
$\lceil n\linf{\vx-\vy}\rceil$ many \iid mean one exponential random variables.
This together with Lemma~\ref{compare FPP}\eqref{it:sac} and the Borel-Cantelli lemma implies
\[
    \mu(\vx) := \lim_{\vy \to \vx, \vy \in \dQ^d} \mu(\vy) \text{
        exists for all $\vx \in \dR^d$ and $\mu(\vx) >0$ whenever $\vx \ne 0$.}
\]
In addition, using Lemma~\ref{compare FPP}\eqref{it:sac} once again
$
    \sum_{\vx \in \dZ^d} \pr(T(\vzero,\vx) \ge \gl \lone{\vx}) < \infty,
$
which implies
\[
    \{\vx \in \dR^d: \linf{\vx} \le t/\gl\} \subseteq \{\vx \in \dR^d: T(\vzero, \vx) \le t\}
\]
for large enough $\gl$. Combining the last two displays, we are in a position to apply the `Denouement' argument (see page 17 of~\cite{D88}) and conclude about the ``shape result" with $A:=\{\vx \in \dR^d: \mu(\vx) \le 1\}$.

Therefore, in order to complete the proof of the theorem it remains to show \eqref{erg}. Fix $\vx \in \dZ^d, k \ge 1$ and let $\nu$ denote the law of the infinite vector $(Y^k_n)_{n \in \dZ}$ and $\varphi$ be the measure preserving transformation on $\dR_+^\dZ$ defined by $(\varphi\omega)_k :=\omega_{k+1}$.
In view of~\cite[Theorem 1.5]{walters82}, it suffices to show that for any two events $A, B$ satisfying $\nu(A), \nu(B) > 0$,
there is an $n \in \dN$ such that $\nu(A \cap \varphi^{-n}B)>0$. To prove this assertion we fix $\eps>0$ and choose $k, l$ large enough
so that there exists $A_{j,l}, B_{j,l} \in \gs\{Y^k_{-k}, Y^k_{-k+1}, \ldots, Y^k_l\}$ satisfying
\begin{align}
    \nu(A\gD A_{j,l}), \nu(B\gD B_{j,l}) \le \eps/4, \label{approx1}
\end{align}
which implies
\begin{align}\label{approx2}
\begin{split}
    |\nu(A \cap \varphi^{-n}B) - \nu(A_{j,l} \cap \varphi^{-n}B_{j,l})|
    & \le \nu(A\gD A_{j,l}) + \nu(\varphi^{-n}B\gD \varphi^{-n}B_{j,l})\\
    & \le \nu(A\gD A_{j,l}) + \nu(B\gD B_{j,l}) \le \eps/2,
\end{split}
\end{align}
as $\varphi$ is measure preserving.

Next we see that applying Corollary \ref{boxbound} we can have $L=L(j,l,\eps)$
large enough such that
\[
    \nu(\Omega^L_{j,l}) \ge 1-\eps/8,
\]
where
\begin{align*}
    \Omega^L_{j,l} := \bigl\{Y^k_i, -j \le i\le l,
    & \text{ is determined by the edge weights in }\\
    &\ \{W_{\la\vz\vw\ra}\mid\vz, \vw \in \ball{(l-j)\vx/2,L}\}\bigr\}.
\end{align*}
Then it is easy to see that
\[
    \left|\nu(A_{j,l}) \nu(\varphi^{-n}B_{j,l}) - \nu(A_{j,l} \cap \Omega^L_{j,l})
    \nu(\varphi^{-n}B_{j,l} \cap \Omega^L_{j+n,l-n})\right| \le \eps/4,
\]
and if $n$ is chosen large enough depending on $L$, then
$A_{j,l} \cap \Omega^L_{j,l}$ and $\varphi^{-n}B_{j,l} \cap
\Omega^L_{j+n,l-n}$ are independent so that
\[
    \left| \nu(A_{j,l} \cap \varphi^{-n}B_{j,l}) - \nu(A_{j,l} \cap \Omega^L_{j,l})
    \nu(\varphi^{-n}B_{j,l} \cap \Omega^L_{j+n,l-n})\right| \le \eps/4.
\]
Combining the last two displays with \eqref{approx1} and
\eqref{approx2}, and recalling that $\varphi$ is measuring preserving,
\[
    |\varphi(A \cap \varphi^{-n}B) -\varphi(A) \varphi(B)| \le 3\eps/2
\]
for large enough $n$. Starting with small enough $\eps>0$, we get $\nu(A \cap \varphi^{-n}B) > 0$.
\end{proof}

%%      ---------------------------------------------------------------------
%%      ------------------------- APPENDIX (OPTIONAL) -----------------------
%%      ---------------------------------------------------------------------

%%      If you have one appendix, uncomment the line \appendix and add
%%      a \section{ *** APPENDIX TITLE ***}. If you have more than
%%      one, uncomment the line \appendices and add a \section{ ***
%%      APPENDIX TITLE ***} command for each appendix title.

%\appendix
%\appendices
%\section{}

%%      Type body of appendix/-ices here.

%%      ---------------------------------------------------------------------
%%      ---------------------------ACKNOWLEDGMENTS (OPTIONAL) ---------------
%%      ---------------------------------------------------------------------

%% ***** UNCOMMENT THE FOLLOWING LINE TO ADD ACKNOWLEDGMENTS.

\ack
We thank Rick Durrett for suggesting the problem and Itai Benjamini, Sourav Chatterjee, Tom LaGatta, Charles Newman, SRS Varadhan and Nikos Zygouras for various helpful comments and discussions while writing this article. The work was done while the second author was a Simons Postdoctoral Fellow at the Courant Institute of Mathematical Sciences, NYU.

%%      ---------------------------------------------------------------------
%%      --------------------------- BIBLIOGRAPHY ----------------------------
%%      ---------------------------------------------------------------------

\frenchspacing
\bibliographystyle{cpam}
\bibliography{lrfpp}

\begin{thebibliography}{10}
\providecommand{\url}[1]{\texttt{#1}}
\providecommand{\urlprefix}{Available at: }
\providecommand{\eprint}[2][]{\url{#2}}

\bibitem{AN86}
Aizenman, M.; Newman, C.~M. Discontinuity of the percolation density in one
  dimensional $1/|x-y|^2$ percolation models. \emph{Communications in
  Mathematical Physics} \textbf{107} (1986), no.~4, 611--647.

\bibitem{A10}
Aldous, D.~J. When knowing early matters: Gossip, percolation and nash
  equilibria. \emph{arXiv preprint arXiv:1005.4846}  (2010).

\bibitem{BR12}
Barbour, A.; Reinert, G. Asymptotic behaviour of gossip processes and small
  world networks. \emph{arXiv preprint arXiv:1202.5895}  (2012).

\bibitem{BB01}
Benjamini, I.; Berger, N. The diameter of long-range percolation clusters on
  finite cycles. \emph{Random Structures \& Algorithms} \textbf{19} (2001),
  no.~2, 102--111.

\bibitem{BBY08}
Benjamini, I.; Berger, N.; Yadin, A. Long-range percolation mixing time.
  \emph{Combinatorics Probability and Computing} \textbf{17} (2008), no.~4,
  487--494.

\bibitem{BKPS11}
Benjamini, I.; Kesten, H.; Peres, Y.; Schramm, O. Geometry of the uniform
  spanning forest: Transitions in dimensions $4, 8, 12,\ldots$. \emph{Selected
  Works of Oded Schramm}  (2011), 751--777.

\bibitem{B02}
Berger, N. Transience, recurrence and critical behavior for long-range
  percolation. \emph{Communications in mathematical physics} \textbf{226}
  (2002), no.~3, 531--558.

\bibitem{B04}
Berger, N. A lower bound for the chemical distance in sparse long-range
  percolation models. \emph{arXiv preprint math/0409021}  (2004).

\bibitem{B08}
Bhamidi, S. First passage percolation on locally treelike networks. i. dense
  random graphs. \emph{Journal of Mathematical Physics} \textbf{49} (2008),
  125\,218.

\bibitem{BHH10}
Bhamidi, S.; van~der Hofstad, R.; Hooghiemstra, G. First passage percolation on
  random graphs with finite mean degrees. \emph{The Annals of Applied
  Probability} \textbf{20} (2010), no.~5, 1907--1965.

\bibitem{B04a}
Biskup, M. On the scaling of the chemical distance in long-range percolation
  models. \emph{The Annals of Probability} \textbf{32} (2004), no.~4,
  2938--2977.

\bibitem{B09}
Biskup, M. Graph diameter in long-range percolation. \emph{arXiv preprint
  math/0406379}  (2009).

\bibitem{B10}
Blair-Stahn, N.~D. First passage percolation and competition models.
  \emph{arXiv preprint arXiv:1005.0649}  (2010).

\bibitem{BCH06}
Borgs, C.; Chayes, J.~T.; van~der Hofstad, R.; Slade, G.; Spencer, J. Random
  subgraphs of finite graphs: Iii. the phase transition for the n-cube.
  \emph{Combinatorica} \textbf{26} (2006), no.~4, 395--410.

\bibitem{CNSW00}
Callaway, D.~S.; Newman, M.~E.~J.; Strogatz, S.~H.; Watts, D.~J. Network
  robustness and fragility: Percolation on random graphs. \emph{Physical Review
  Letters} \textbf{85} (2000), no.~25, 5468--5471.

\bibitem{CI02}
Campanino, M.; Ioffe, D. Ornstein-zernike theory for the bernoulli bond
  percolation on zd. \emph{Annals of probability}  (2002), 652--682.

\bibitem{CMM06}
Cannas, S.~A.; Marco, D.; Montemurro, M.~A. Long range dispersal and spatial
  pattern formation in biological invasions. \emph{Mathematical biosciences}
  \textbf{203} (2006), no.~2, 155--170.

\bibitem{CD11}
Chatterjee, S.; Durrett, R. Asymptotic behavior of aldous' gossip process.
  \emph{The Annals of Applied Probability} \textbf{21} (2011), no.~6,
  2447--2482.

\bibitem{CGS02}
Coppersmith, D.; Gamarnik, D.; Sviridenko, M.: The diameter of a long range
  percolation graph, in \emph{Proceedings of the thirteenth annual ACM-SIAM
  symposium on Discrete algorithms}, Society for Industrial and Applied
  Mathematics, 2002 pp. 329--337.

\bibitem{CD81}
Cox, J.~T.; Durrett, R. Some limit theorems for percolation processes with
  necessary and sufficient conditions. \emph{The Annals of Probability}
  (1981), 583--603.

\bibitem{CS09}
Crawford, N.; Sly, A. Heat kernel upper bounds on long range percolation
  clusters. \emph{arXiv preprint arXiv:0907.2434}  (2009).

\bibitem{DS13}
Ding, J.; Sly, A. Distances in critical long range percolation. \emph{arXiv
  preprint arXiv:1303.3995}  (2013).

\bibitem{D88}
Durrett, R. \emph{Lecture notes on particle systems and percolation}, Wadsworth
  \& Brooks/Cole Advanced Books \& Software, 1988.

\bibitem{FM04}
Filipe, J.~A.; Maule, M.~M. Effects of dispersal mechanisms on spatio-temporal
  development of epidemics. \emph{Journal of theoretical biology} \textbf{226}
  (2004), no.~2, 125--141.

\bibitem{GM05}
Garet, O.; Marchand, R. Coexistence in two-type first-passage percolation
  models. \emph{The Annals of Applied Probability} \textbf{15} (2005), no.~1A,
  298--330.

\bibitem{gm08}
Gou\'er\'e, J.-B.; Marchand, R. Continuous first-passage percolation and
  continuous greedy paths model: linear growth. \emph{The Annals of Applied
  Probability}  (2008), 2300--2319.

\bibitem{GB86}
{Grassberger}, P.: Spreading of epidemic processes leading to fractal
  structures, in \emph{Fractals in Physics: Proceedings of the Sixth Trieste
  International Symposium}, edited by L.~{Pietronero}; E.~{Tosatti}, 1986 pp.
  273--278.

\bibitem{G99}
Grimmett, G.~R. \emph{Percolation}, vol. 321, Springer, 1999.

\bibitem{GK12}
Grimmett, G.~R.; Kesten, H. Percolation since saint-flour. \emph{arXiv preprint
  arXiv:1207.0373}  (2012).

\bibitem{HW65}
Hammersley, J.; Welsh, D. First-passage percolation, subadditive processes,
  stochastic networks, and generalized renewal theory.
  \emph{Bernoulli-Bayes-Laplace Anniversary Volume}  (1965), 61--110.

\bibitem{HS00}
Hara, T.; Slade, G. The scaling limit of the incipient infinite cluster in
  high-dimensional percolation. i. critical exponents. \emph{Journal of
  Statistical Physics} \textbf{99} (2000), no.~5, 1075--1168.

\bibitem{HS00a}
Hara, T.; Slade, G. The scaling limit of the incipient infinite cluster in
  high-dimensional percolation. ii. integrated super-brownian excursion.
  \emph{Journal of Mathematical Physics} \textbf{41} (2000), no.~3, 1244.

\bibitem{HH98}
Hinrichsen, H.; Howard, M. A model for anomalous directed percolation.
  \emph{The European Physical Journal B-Condensed Matter and Complex Systems}
  \textbf{7} (1999), no.~4, 635--643.

\bibitem{H05}
Hoffman, C. Coexistence for richardson type competing spatial growth models.
  \emph{The Annals of Applied Probability} \textbf{15} (2005), no.~1B,
  739--747.

\bibitem{IN88}
Imbrie, J.~Z.; Newman, C.~M. An intermediate phase with slow decay of
  correlations in one dimensional $1/|x-y|^2$ percolation, ising and potts
  models. \emph{Communications in mathematical physics} \textbf{118} (1988),
  no.~2, 303--336.

\bibitem{J99}
Janson, S. One, two and three times $\log n/n$ for paths in a complete graph
  with random weights. \emph{Combinatorics, Probability and Computing}
  \textbf{8} (1999), no.~04, 347--361.

\bibitem{JOWH98}
Janssen, H.~K.; Oerding, K.; Van~Wijland, F.; Hilhorst, H.~J. L{\'e}vy-flight
  spreading of epidemic processes leading to percolating clusters. \emph{The
  European Physical Journal B-Condensed Matter and Complex Systems} \textbf{7}
  (1999), no.~1, 137--145.

\bibitem{K86}
Kesten, H. Aspects of first passage percolation. \emph{Ecole d'Et{\'e} de
  Probabilit{\'e}s de Saint Flour XIV-1984}  (1986), 125--264.

\bibitem{KS91}
Krug, J.; Spohn, H. Kinetic roughening of growing surfaces. \emph{C. Godreche,
  Cambridge University Press, Cambridge}  (1991), 412--525.

\bibitem{L85}
Liggett, T.~M. An improved subadditive ergodic theorem. \emph{The Annals of
  Probability}  (1985), 1279--1285.

\bibitem{LP05}
Lyons, R.; Peres, Y.: Probability on trees and networks. 2005.

\bibitem{MMC11}
Marco, D.~E.; Montemurro, M.~A.; Cannas, S.~A. Comparing short and
  long-distance dispersal: modelling and field case studies. \emph{Ecography}
  \textbf{34} (2011), no.~4, 671--682.

\bibitem{M72}
Mollison, D.: The rate of spatial propagation of simple epidemics, in
  \emph{Proceedings of the Sixth Berkeley Symposium on Mathematical Statistics
  and Probability}, vol.~3, 1972 pp. 579--614.

\bibitem{MN00}
Moore, C.; Newman, M.~E.~J. Epidemics and percolation in small-world networks.
  \emph{Physical Review E} \textbf{61} (2000), no.~5, 5678.

\bibitem{NS86}
Newman, C.~M.; Schulman, L.~S. One dimensional $1/|j-i|^s$ percolation models:
  The existence of a transition for $s> 2$. \emph{Communications in
  Mathematical Physics} \textbf{104} (1986), no.~4, 547--571.

\bibitem{S99}
Schulman, L.~S. Long range percolation in one dimension. \emph{Journal of
  Physics A: Mathematical and General} \textbf{16} (1999), no.~17, L639.

\bibitem{SCB02}
Schwartz, N.; Cohen, R.; Ben-Avraham, D.; Barab{\'a}si, A.~L.; Havlin, S.
  Percolation in directed scale-free networks. \emph{Physical Review E}
  \textbf{66} (2002), no.~1, 015\,104.

\bibitem{S01}
Smirnov, S. Critical percolation in the plane: Conformal invariance, cardy's
  formula, scaling limits. \emph{Comptes Rendus de l'Acad{\'e}mie des
  Sciences-Series I-Mathematics} \textbf{333} (2001), no.~3, 239--244.

\bibitem{SW01}
Smirnov, S.; Werner, W. Critical exponents for two-dimensional percolation.
  \emph{arXiv preprint math/0109120}  (2001).

\bibitem{SW78}
Smythe, R.~T.; Wierman, J.~C. \emph{First-passage percolation on the square
  lattice}, vol. 671, Springer-Verlag Berlin, 1978.

\bibitem{SRC11}
Soubeyrand, S.; Roques, L.; Coville, J.; Fayard, J. Patchy patterns due to
  group dispersal. \emph{Journal of Theoretical Biology} \textbf{271} (2011),
  no.~1, 87--99.

\bibitem{T10}
Trapman, P. The growth of the infinite long-range percolation cluster.
  \emph{The Annals of Probability} \textbf{38} (2010), no.~4, 1583--1608.

\bibitem{HHM01}
Van Der~Hofstad, R.; Hooghiemstra, G.; Van~Mieghem, P. First-passage
  percolation on the random graph. \emph{Probability in the Engineering and
  Informational Sciences} \textbf{15} (2001), no.~02, 225--237.

\bibitem{walters82}
Walters, P. \emph{An introduction to ergodic theory}, vol.~79, Springer-Verlag
  New York, 1982.

\bibitem{WS98}
Watts, D.; Strogatz, S. The small world problem. \emph{Collective Dynamics of
  Small-World Networks} \textbf{393} (1998), 440--442.

\end{thebibliography}

\end{document}